\documentclass[11pt, letterpaper]{article}
\usepackage{bm}
\usepackage{amsmath,amsthm, verbatim,amssymb,amsfonts,amscd, graphicx, graphics, physics}
\usepackage[mathscr]{eucal}
\usepackage{indentfirst}
\usepackage{booktabs}
\usepackage{pict2e}
\usepackage{epic}
\numberwithin{equation}{section}
\usepackage[footnotesize,bf]{caption}
\usepackage[margin=0.98in]{geometry}

\usepackage{epstopdf}
\usepackage{mathtools}
\usepackage[dvipsnames]{xcolor}
\usepackage{algorithmic}
\usepackage[ruled,vlined]{algorithm2e}
\usepackage{authblk}
\usepackage{placeins}

\theoremstyle{plain}
\newtheorem{theorem}{Theorem}[section]

\newtheorem{lemma}{Lemma}
\newtheorem{proposition}{Proposition}
\newtheorem{assumption}{Assumption}

\theoremstyle{definition}
\newtheorem{definition}{Definition}
\newtheorem{remark}[theorem]{Remark}
\newtheorem{myexample}{Example}

\usepackage{threeparttable}
\usepackage{subcaption} 
\usepackage{siunitx}
\usepackage{multirow}

\usepackage[backend=bibtex, firstinits=true,maxbibnames=9,natbib=true,url=false,sorting=nyt,doi=true,backref=false]{biblatex}
\renewbibmacro{in:}{\ifentrytype{article}{}{\printtext{\bibstring{in}\intitlepunct}}}

\usepackage[colorlinks,linkcolor = BrickRed, citecolor=blue]{hyperref}

\def\R{{\mathbb R}}
\def\B{{\mathcal B}}

\def\P{{\mathbb P}}

\def\N{{\mathbb N}}

\def\F{{\mathcal F}}

\def\H{{\mathcal H}}
\def\M{{\mathcal M}}
\def\L{{\mathcal L}}

\def\KL{{\mathsf{KL}}}

\def\ah{{\widehat{a}}}
\def\th{{\widehat{\theta}}}
\def\td{{\widetilde{\theta}}}

\def\G{{\mathcal G}}

\def\e{{\varepsilon}}

\def\t{{\theta}}
\def\tt{{\vartheta}}

\def\ct{{\mathrm{cent}}}

\def\supp{{\mathrm{supp}}}

\newcommand{\<}{\left\langle}
\renewcommand{\>}{\right\rangle}

\renewcommand{\bar}{\overline}

\newcommand{\ito}{It\^o}
\newcommand{\ds}{\mathsf{d}_{\mathrm{SL}_\alpha}}
\newcommand{\dszero}{\mathsf{d}_{\mathrm{SL}_0}}
\newcommand{\dshalf}{\mathsf{d}_{\mathrm{SL}_{0.5}}}
\newcommand{\dsone}{\mathsf{d}_{\mathrm{SL}_{1}}}

\newcommand{\E}{\mathbb{E}}

\newcommand{\psd}{\mathbb{S}^d_{++}}

\newcommand{\lip}{\mathrm{Lip}}

\renewcommand{\d}[1]{\ensuremath{\operatorname{d}\!{#1}}}

\DeclareMathOperator*{\argmin}{arg\,min}
\usepackage{setspace, cleveref}

\begin{document}

\title{Joint stochastic localization and applications}

\date{\today}

\setstretch{1.1}

\author[1]{Tom Alberts}
\author[2]{Yiming Xu}
\author[2]{Qiang Ye}
\affil[1]{Department of Mathematics, University of Utah}
\affil[2]{Department of Mathematics, University of Kentucky}

\maketitle

\begin{abstract}
Stochastic localization is a pathwise analysis technique that has emerged as a powerful tool in high-dimensional probability and sampling. 
In this work, we extend stochastic localization to a joint framework for coupling probability measures and explore its applications in distributional data analysis. We first unify existing stochastic localization processes under Eldan's $\alpha$-scheme and characterize their localization rates. 
Building on this, we introduce a joint scheme to couple probability measures via concurrent $\alpha$-schemes driven by a shared Brownian motion. This construction is canonical and induces a family of metrics on the space of probability measures, which we call Eldan's $\alpha$-distance. 
Alternative variants that extrapolate optimal Gaussian couplings to log-concave measures are also discussed. 
We study the theoretical properties of Eldan's $\alpha$-distance, including its restriction to Gaussian measures and its behavior under affine transformations. For $\alpha = 0$, we show it is topologically equivalent to the $2$-Wasserstein distance for measures supported on a common compact set; we also relate its weighted variants to linearized optimal transport in Wiener space and to score-matching objectives in training diffusion models. Computationally, we develop efficient estimators for Eldan's $\alpha$-distance in the cases $\alpha=0$ and $\alpha=1/2$, with rigorous error guarantees for log-concave and finitely supported measures in the former setting and Gaussian measures in the latter. 
Finally, we apply Eldan's $\alpha$-distance as a scalable surrogate for the $2$-Wasserstein distance to enable fast pairwise distance estimation and approximate computation of Wasserstein barycenters.
\end{abstract}

\vspace{0.5em}

\vspace{0.5em}

\tableofcontents

\newpage

\section{Introduction}

Stochastic localization (SL) is a pathwise analysis technique introduced by \citet{eldan2013thin} in the study of isoperimetric problems. It extends an earlier deterministic localization scheme for analyzing the mixing time of random walks on convex bodies \cite{lovasz1993random}. Recently, SL has garnered significant attention for its pivotal role in advancing several conjectures in convex geometry, most notably the Kannan--Lovász--Simonovits (KLS) conjecture \cite{kannan1995isoperimetric}, Bourgain's hyperplane conjecture \cite{bourgain1986high, bourgain1986geometry}, and the thin-shell conjecture \cite{anttila2003central}. See \cite{lee2017eldan, chen2021almost, jambulapati2022slightly, klartag2022bourgain, guan2024note} for some recent progress on the KLS conjecture that leverages SL. Bourgain's hyperplane conjecture and the thin-shell conjecture have been resolved affirmatively by \citet{klartag2024affirmative} and by \citet{klartag2025thin}, respectively. 

Formally, SL is a measure-valued martingale that deforms a probability measure of interest into lower-dimensional components that are more amenable to analysis; see Section~\ref{sec:bayes} for a detailed exposition. As an explicitly simulable process, SL has attracted significant attention in applied domains, shedding insights into a variety of sampling-related problems including mixing bounds for Markov chains \cite{chen2022localization}, convergence rates of diffusion models \cite{benton2023linear}, and entropy-efficient measure decompositions \cite{eldan2020taming}. It has also led to practical algorithms for sampling from target measures in both discrete and continuous settings \cite{montanari2023sampling, el2022sampling, montanari2023posterior, grenioux2024stochastic, demyanenko2025sampling}. 

In this work, we take a different perspective on SL. Rather than using SL to sample from a single distribution, we develop a joint SL formulation for constructing couplings across multiple distributions simultaneously, a natural extension that enables quantitative comparison and interpolation of distributions beyond sampling from them individually. The couplings induced by our joint SL framework are algorithmic in nature and give rise to a family of distances on the space of probability measures, which encode geometric structure while remaining computationally tractable. We leverage these properties to develop new tools for distributional data analysis.  

In the remainder of this section, we summarize our core technical contributions and compare our results with existing literature on distributional data analysis, with a particular focus on computational optimal transport and sampling for diffusion models.

\subsection{Contributions}

The main theoretical and computational contributions of this work are as follows:

\begin{itemize}

\item We unify several existing SL processes under a single-parameter family, which we refer to as Eldan's $\alpha$-scheme, and characterize how $\alpha$ governs the localization rate of the resulting process (\Cref{thm:01}). For $\alpha>0$, the scheme contains a preconditioning factor involving the pseudoinverse of the covariance to adapt to the local geometry of the evolving measure. To improve numerical stability, we introduce a regularized variant and analyze its localization rate in the case $\alpha = 1/2$ (\Cref{thm:reg}).

\item We extend Eldan's $\alpha$-scheme to a joint framework by running concurrent $\alpha$-schemes using a shared Brownian motion. This construction yields a canonical coupling of the measures and induces a metric on the space of probability measures, which we call Eldan's $\alpha$-distance. To build intuition, we explicitly compute this distance between Gaussian measures and show that it admits a Hilbert space representation: the distance is Euclidean in a suitably embedded space, and interpolates between different geometries as $\alpha$ varies from $0$ to $1$ (\Cref{thm:gauss}). Leveraging this coupling, we give an algorithmic proof of a normal approximation result for log-concave measures under the $2$-Wasserstein metric $W_2$ (\Cref{thm:caf}).

In addition, we propose a local coupling construction between two log-concave measures by further tuning the correlations of the shared Brownian motion. When restricted to Gaussian measures, the induced coupling coincides with the optimal $W_2$ coupling (\Cref{thm:1}).

\item We analyze the theoretical properties of Eldan's $\alpha$-distance under affine transformations (\Cref{prop:shift}). When $\alpha = 0$, the underlying dynamics of the scheme have a constant diffusion coefficient, which makes the analysis tractable. In this case, we establish that Eldan's $\alpha$-distance is topologically equivalent to $W_2$ for probability measures supported on a common compact set (Theorem~\ref{thm:metric-equiv}). Moreover, we show that weighted variants of this distance are closely related to linearized optimal transport in Wiener space (\Cref{thm:kl}) and to the score-matching objectives used in training diffusion models  (\Cref{thm:diffusion}).

\item We introduce a truncated Monte Carlo (MC) estimator for Eldan's $\alpha$-distance, based on simulating the mean process of the joint Eldan's $\alpha$-scheme up to a finite truncation time. The appropriate discretization scheme depends on the choice of $\alpha$, and we provide error analyses for two practically important regimes $\alpha = 0$ and $\alpha = 1/2$ under suitable assumptions.

For $\alpha = 0$, we employ a mixed uniform-geometric scheme for the Euler--Maruyama discretization and establish sharp error bounds for log-concave measures (\Cref{thm:logconcave-0}). We further derive error guarantees for measures with finite support via a pathwise argument (\Cref{thm:eperror}). For $\alpha = 1/2$, numerical stability requires regularization, and we adopt a uniform scheme. In this setting, a general analysis is intractable; instead, we focus on the Gaussian case and obtain sharp bounds on the discretization and regularization errors, respectively (\Cref{thm:1/2-gaussian}). A key feature of the analysis is a decoupling of these two error sources, which we attribute to a semi-implicit structure implied by Eldan's $\alpha$-scheme.
\end{itemize}

To demonstrate the utility of the proposed method, we apply Eldan's $\alpha$-distance as a surrogate for $W_2$ in several tasks in distributional data analysis, including fast computation of pairwise distances among large cohorts of distributions and approximate computation of Wasserstein barycenters, both of which are of ongoing interest in statistics and machine learning \cite{agueh2011barycenters, rabin2011wasserstein, cuturi2014fast, peyre2019computational, haviv2024wasserstein}. We further provide detailed numerical experiments on both simulated data and standard benchmark datasets to support our theoretical findings. The connections between different sections are summarized in the flowchart in \Cref{fig:flowc}. 

\begin{figure}[htbp]
  \centering 
  \includegraphics[width=0.55\textwidth, trim={0.0cm 0.4cm 0.0cm 0.5cm},clip]{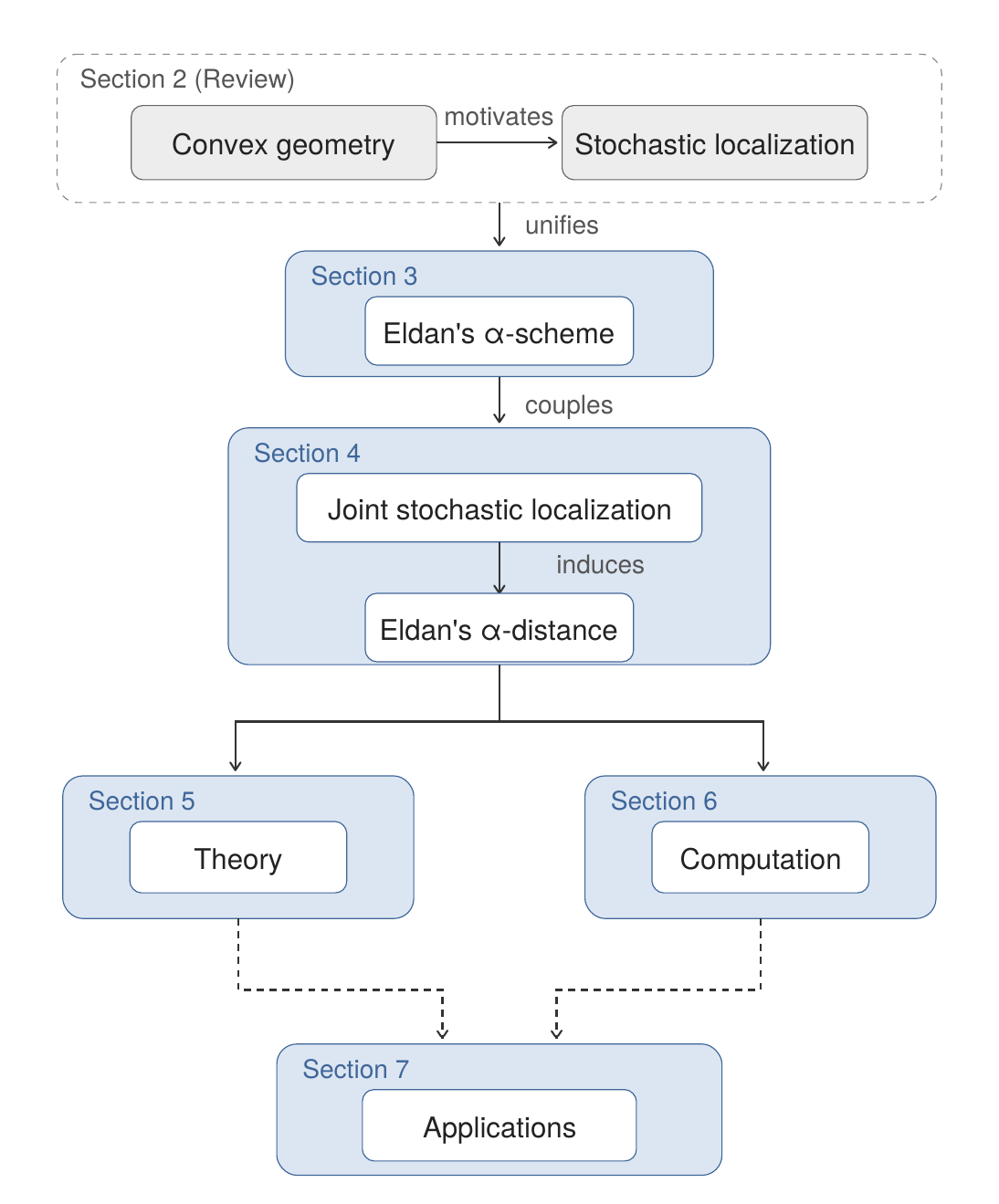}
  \caption{A flowchart that connects different sections in this work.}
  \label{fig:flowc}
\end{figure}

\subsection{Related work}

The idea of joint SL builds on the foundational framework of SL introduced by \citet{eldan2013thin}, which requires substantial technical preparation. We therefore defer a detailed discussion of the relevant literature to \Cref{sec:bayes}. The extension of SL to a joint framework appears to be new. A related construction was independently used by \citet{klartag2025thin} in the recent resolution of the thin-shell conjecture, under the name parallel coupling, which can be viewed as a special case of the joint Eldan's $\alpha$-scheme with $\alpha = 0$.

Eldan's $\alpha$-distance is a distributional metric defined based on the coupling induced by the joint Eldan's $\alpha$-scheme. This viewpoint is closely related to optimal transport distances, such as the 2-Wasserstein distance $W_2$, which is defined via an optimal coupling minimizing the squared Euclidean transport cost. Optimal transport is a rich and well-studied area \cite{villani2021topics}, and Wasserstein distances play a central role in modern machine learning applications involving probability distributions \cite{NIPS2014_f033ed80, pmlr-v70-arjovsky17a}. However, the computational cost of Wasserstein distances can become prohibitive when the underlying measures have large support sizes.

The proposed Eldan's $\alpha$-distance can be used as a computational surrogate for $W_2$. Alternative surrogates have been developed to improve scalability, including entropic regularization \cite{cuturi2013sinkhorn}, sliced Wasserstein distances \cite{rabin2011wasserstein, bonnotte2013unidimensional}, and linearized Wasserstein distances \cite{wang2013linear}. When $\alpha = 0$, Eldan's $\alpha$-distance can be interpreted as an approximate linearization of $W_2$ in Wiener space, an infinite-dimensional analogue of the linearized Wasserstein distance in $\R^d$ with respect to the standard Gaussian reference measure; see~\Cref{sec:dan}. More generally, it extends beyond this regime and induces distinct geometries when $\alpha > 0$.

The estimation of Eldan's $\alpha$-distance relies on Euler--Maruyama discretization of the underlying stochastic differential equations (SDEs). Existing theoretical analyses of SL as a sampling method have predominantly focused on discrete settings arising in statistical physics models \cite{el2022sampling, montanari2023posterior, demyanenko2025sampling}. While the analysis of Euler--Maruyama schemes is classical \cite{kloeden2013numerical}, our setting has a key distinction: the stepsize is chosen based on a time-dependent Lipschitz property rather than a global one. This feature is intrinsic to the localization behavior of Eldan's $\alpha$-schemes and plays a critical role in both algorithm design and error analysis.

Moreover, Eldan's $\alpha$-scheme is closely related to the time-reversal of the Ornstein--Uhlenbeck (OU) process used in diffusion models~\cite{song2020score}, a modern paradigm for generative modeling; see also \Cref{sec:diffusion}. As a result, our approach is related to recent work on sampling-complexity analysis for diffusion models \cite{chen2023sampling, chen2023improved, benton2023linear}. For instance, a similar mixed uniform-geometric scheme was used by \citet{benton2023linear} to establish sharp convergence bounds for diffusion models, leveraging the connection to SL via a time change. However, these works typically measure error in total variation distance or Kullback--Leibler (KL) divergence, whereas our approach is based on a coupling-induced distance. This distinction necessitates tracking pathwise coupling structures, which requires different analysis techniques.

\subsection{Organization}

The rest of the paper is organized as follows. 
\Cref{sec:bayes} reviews SL from two complementary perspectives, including a commonly used formulation and an alternative Bayesian interpretation; it also collects some useful observations based on SL calculus. 
\Cref{sec:3} unifies several existing SL schemes under Eldan's $\alpha$-scheme and analyzes the resulting localization rate; a regularized version is introduced and analyzed for $\alpha = 1/2$. 
\Cref{sec:jsl} extends Eldan's $\alpha$-scheme to a joint framework and introduces Eldan's $\alpha$-distance; it also presents an alternative construction that extends the optimal Gaussian coupling to log-concave measures. 
\Cref{sec:dist} provides a theoretical analysis of Eldan's $\alpha$-distance and explores its connections to optimal transport and diffusion models. 
\Cref{sec:newcomp} develops efficient estimators for Eldan's $\alpha$-distance and establish relevant theoretical guarantees. \Cref{sec:apps} discusses applications of Eldan's $\alpha$-distance in distributional data analysis. 
\Cref{sec:num} presents numerical experiments supporting our theoretical findings. \Cref{sec:conc} concludes the article and discusses directions for future work.

\paragraph{Notation} For $x\in\R^d$, we denote its Euclidean norm by $\|x\|_2$. For $A\in\R^{d\times d}$, we use $\|A\|_2$ and $\|A\|_F$ to denote its spectral norm and Frobenius norm, respectively. We use the notation $A\succeq B$ to denote the Loewner order on positive semi-definite matrices. Given $R>0$, we define $\B_R(\R^d) = \{z\in\R^d: \|z\|_2\leq R\}$. We write $\delta_x$ for the Dirac mass at $x$. For a probability measure $\mu$ on $\R^d$, we denote its support by $\supp(\mu)$. For any measurable map $f: \R^d\to\R^k$, the pushforward of $\mu$ under $f$, denoted by $f\# \mu$, is defined by $(f\#\mu)(S) = \mu(f^{-1}(S))$ for any Borel set $S\subseteq\R^k$.


\section{Stochastic localization}\label{sec:bayes}

This section provides a brief introduction to SL following the technical overview by \citet{eldan2022analysis}. 
We also present an alternative Bayesian perspective that extends \cite[Remark 4.2 (ii)]{klartag2023spectral} to provide additional intuition. 
We conclude with some useful calculations from SL calculus that will be used throughout the article. 

\subsection{Measure-valued random processes}\label{sl:pop}

Let $\mu$ be a probability measure on $\R^d$. Let $W_t$ be a standard Brownian motion in $\R^d$ with filtration $\F_t$. An SL scheme associated with $\mu$ is a measure-valued stochastic process ${\mu_t}$ that is absolutely continuous with respect to $\mu$ (denoted as $\mu_t\ll \mu$), with its density satisfying the following SDEs:
\begin{equation}\label{sl}
\begin{aligned}
\begin{cases}
\d p_t(x) &= p_t(x)\< x-a_t, C_t \d W_t\>\\ 
p_0(x) &= 1
\end{cases}, 
\quad\quad x\in\supp(\mu),
\end{aligned}
\end{equation}
where $\langle\cdot, \cdot\rangle$ denotes the standard inner product on $\R^d$, 
\begin{align}
a_t = \int_{\R^d} x\mu_t(\d x) = \int_{\R^d} xp_t(x) \mu(\d x)
\end{align} 
denotes the mean of $\mu_t$, and $C_t\in\R^{d\times d}$ is an $\F_t$-adapted matrix-valued control process. Typically $C_t$ is a matrix-valued function of $\mu_t$. 

\Cref{sl} is an infinite system of SDEs if $\supp(\mu)$ is infinite. Under such circumstances, it is not immediately clear whether \eqref{sl} has a well-defined solution. Fortunately, \citet{eldan2013thin, eldan2020clt} have shown that, under suitable regularity assumptions on $\mu$ and $C_t$, \eqref{sl} has a unique solution. We will return to this issue in Section~\ref{sec:3}. For now, we assume that \eqref{sl} has a unique solution and derive a few consequences from it.

First, note that $p_t$ is a density a.s. since (i) $p_t$ is nonnegative if started nonnegative and (ii) 
\begin{align}
\d\mu_t(\R^d) = \int_{\R^d}p_t(x)\< x-a_t, C_t \d W_t\>\mu(\d x) = 0\implies\mu_t(\R^d) = 1. \label{density}
\end{align}
Applying \ito's formula to $\log p_t(x)$ and using \eqref{sl},  
\begin{align}
\d{\log p_t(x)} = \< x-a_t, C_t \d W_t\> - \frac{1}{2}\|C_t^\top(x-a_t)\|_2^2 \d t,\label{logSL}
\end{align}
which combined with \eqref{density} implies
\begin{align}
p_t(x) = \frac{\exp\left\{\<\t_t, x\> - \frac{1}{2}\<x,  G_t x\>\right\}}{\int_{\R^d}\exp\left\{\<\t_t, z\> - \frac{1}{2}\<z,  G_t z\>\right\} \mu( \d z)},\label{pdensity}
\end{align}
where 
\begin{equation}\label{finite_system}
\begin{aligned}
\begin{cases}
\d\t_t &= C_tC^\top_t a_t \d t + C_t \d W_t\\
\d G_t &= C_tC_t^\top \d t
\end{cases} 
\end{aligned}
.
\end{equation} 
We refer to the processes $\t_t$ and $a_t$ as the \emph{observation process} and the \emph{(posterior) mean process} of the SL scheme \eqref{sl}, respectively, for reasons that will become clear in Section~\ref{sec:1.2}. 

\Cref{pdensity} shows that $\mu_t$ is a random Gaussian tilt of $\mu$. In the special case $C_t \equiv I$, $G_t = tI$ so that $p_t$ becomes increasingly spiked on $\supp(\mu)$; see Figure~\ref{fig:sl1d} for an illustration. This implies that $\mu_t$ converges to a random point mass as $t \to \infty$, which justifies the name ``stochastic localization.'' Moreover, the location of the converged point mass has the same distribution as $\mu$ thanks to the martingale structure (see the next paragraph). Such observations inspired practical algorithms to sample from a target distribution by running an SL process, for target distributions based on the Sherrington--Kirkpatrick (SK) Gibbs measure \cite{el2022sampling}, spiked models \cite{montanari2023posterior}, binary perceptron models \cite{demyanenko2025sampling}, and other multimodal distributions \cite{grenioux2024stochastic}. 

\begin{figure}[htbp]
  \centering
  \includegraphics[width=0.24\textwidth, trim={0.2cm 0.5cm 3cm 1cm}, clip]{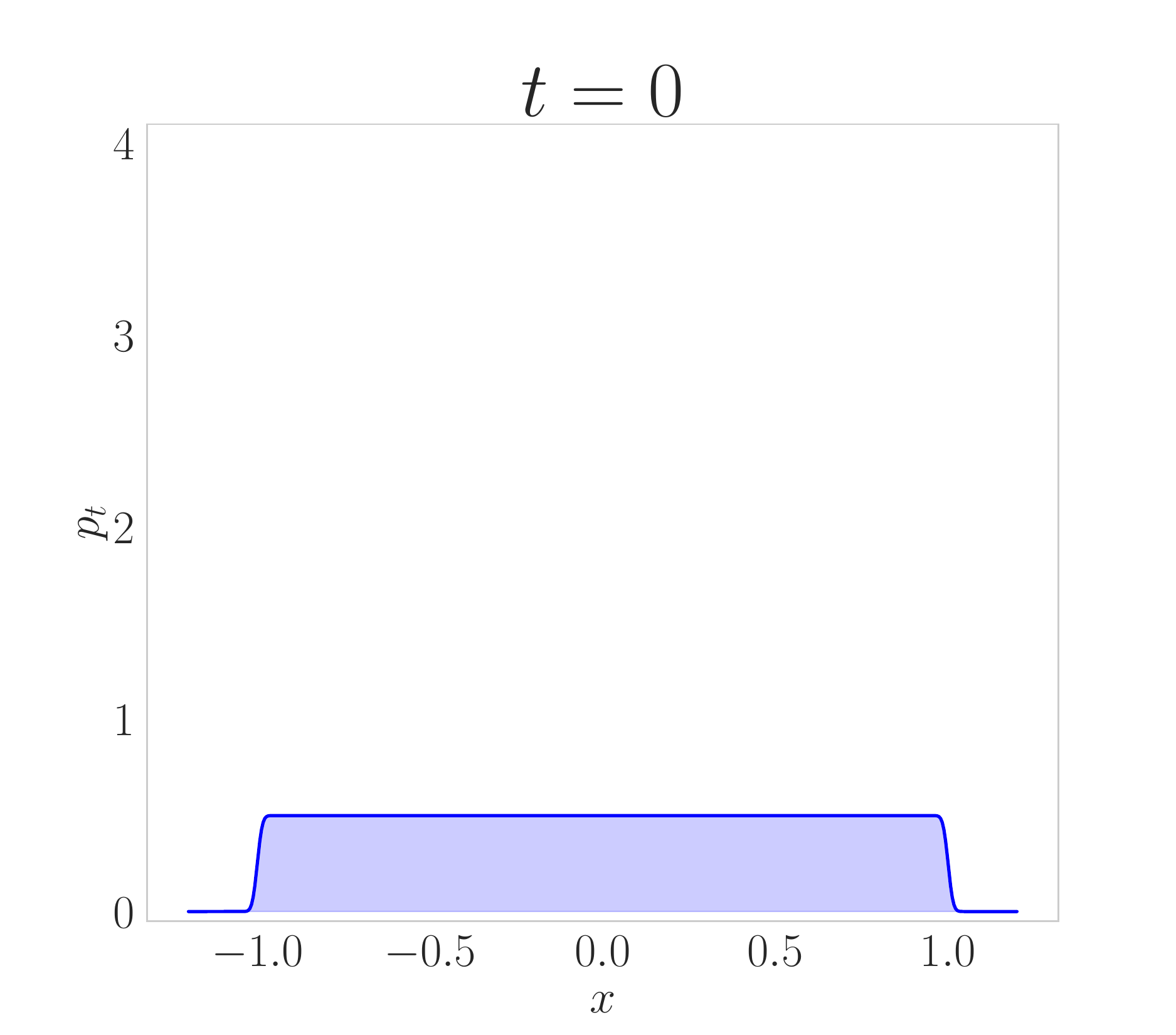}\hfill
  \includegraphics[width=0.24\textwidth, trim={0.2cm 0.5cm 3cm 1cm}, clip]{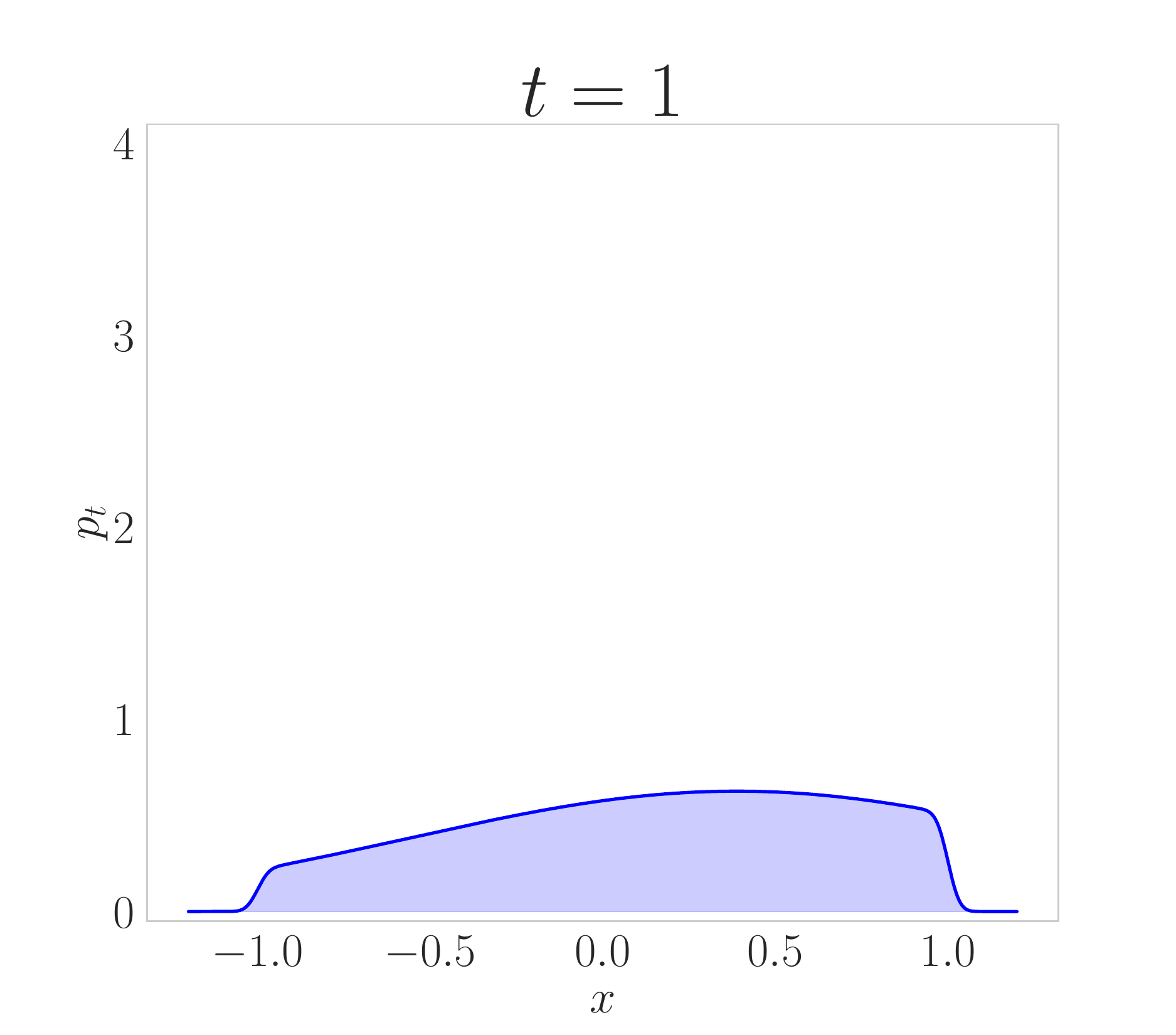}\hfill
  \includegraphics[width=0.24\textwidth, trim={0.2cm 0.5cm 3cm 1cm}, clip]{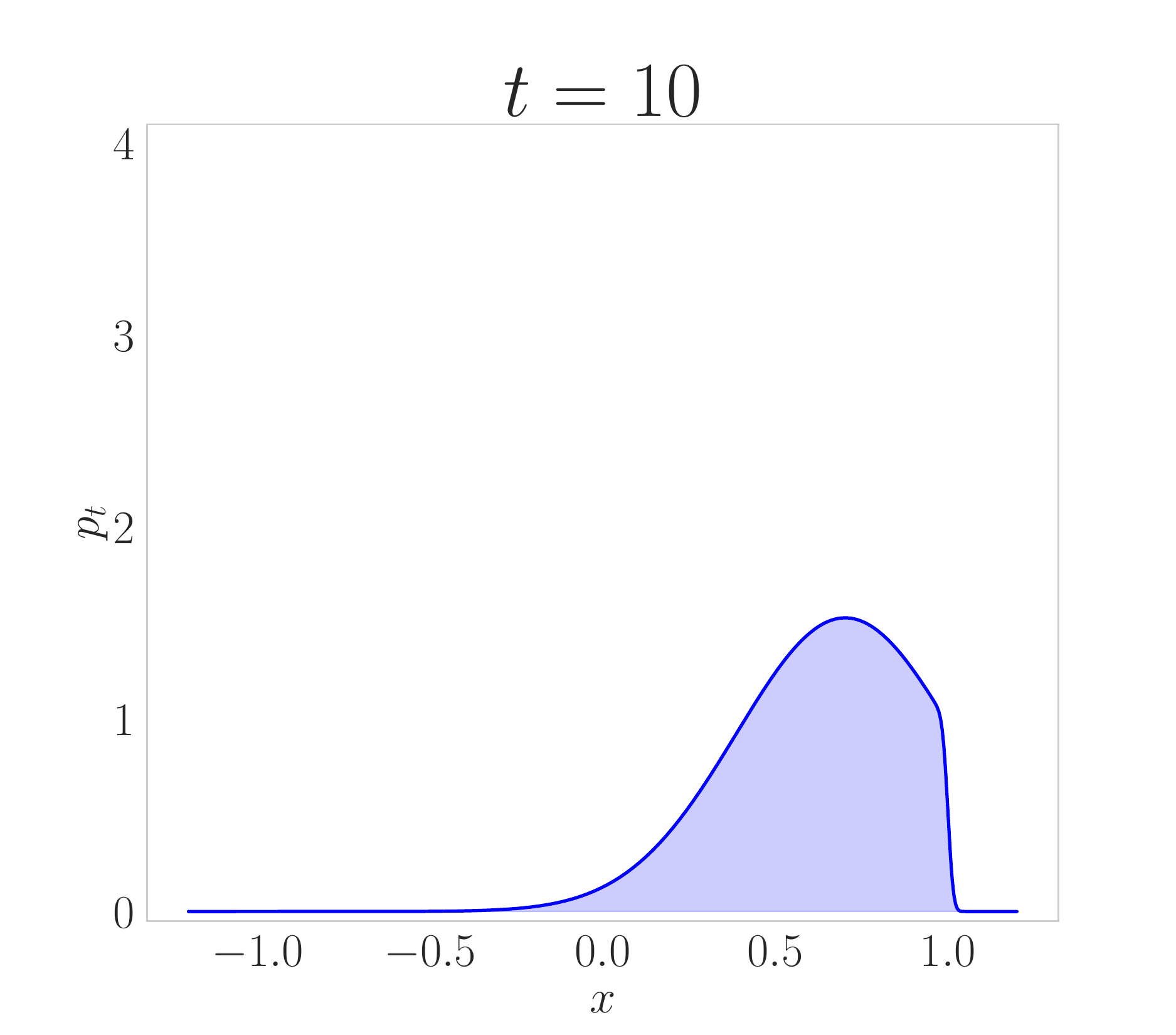}\hfill
  \includegraphics[width=0.24\textwidth, trim={0.2cm 0.5cm 3cm 1cm}, clip]{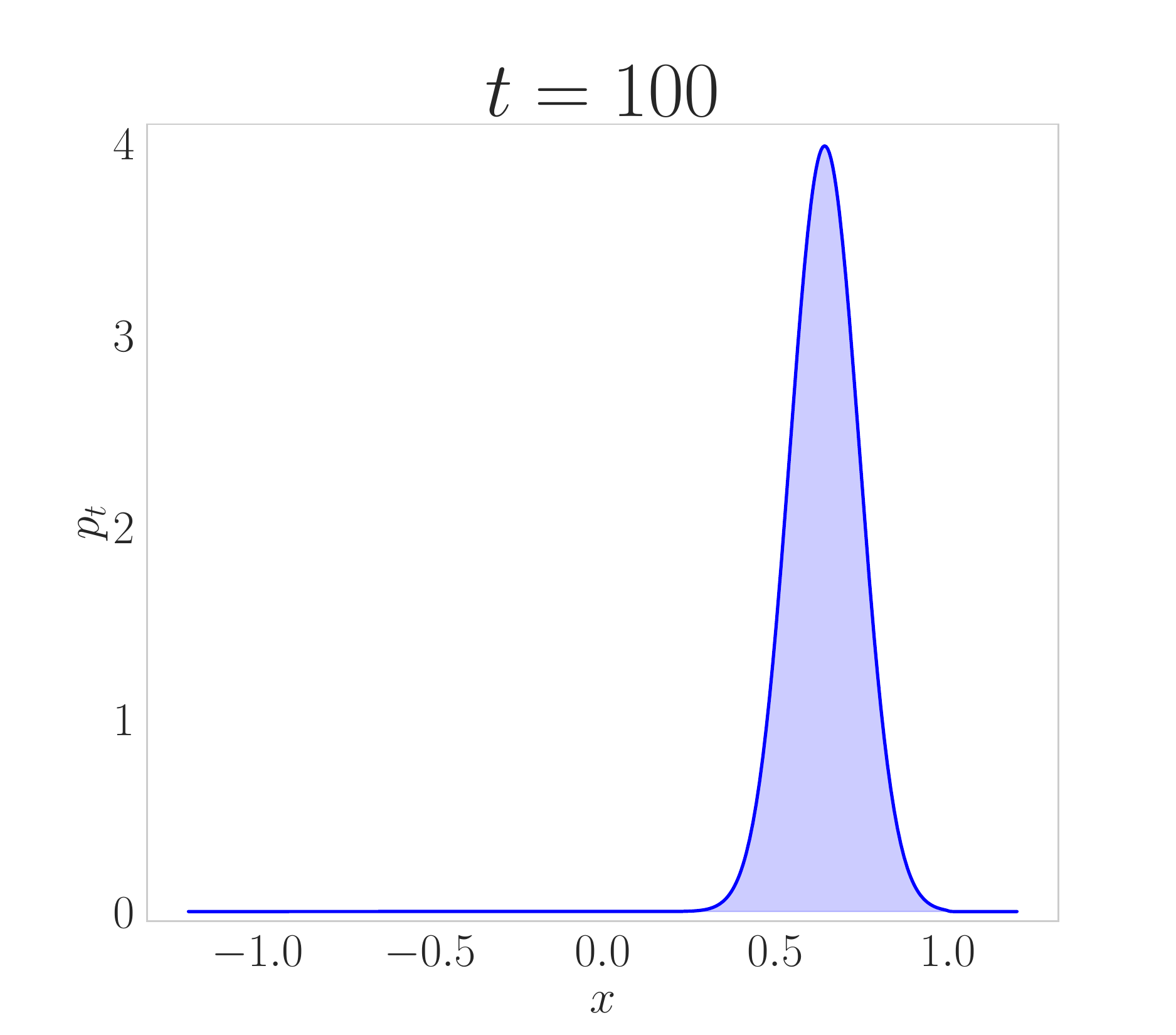}\hfill
  \caption{An illustration of SL associated with the uniform distribution on $[-1, 1]$ with $C_t \equiv I$, $G_t = tI$. As time progresses from left to right, the density $p_t$ becomes increasingly peaked, and thus $\mu_t$ becomes more localized.}\label{fig:sl1d}
  \end{figure}

To see the martingale structure, for any bounded measurable $\varphi: \supp(\mu)\to\R$, $\int_{\R^d}\varphi(x) \mu_t(\d x)$ remains uniformly bounded in $t$ and satisfies 
\begin{align*}
\d{\left(\int_{\R^d}\varphi(x) \mu_t(\d x)\right)} = \int_{\R^d}\varphi(x) \d p_t(x)\mu(\d x) \stackrel{\eqref{sl}}{=} \<\int_{\R^d}\varphi(x)p_t(x)C_t^\top(x-a_t) \mu(\d x), \d W_t\>.
\end{align*}
Hence, $\int_{\R^d}\varphi(x)\mu_t(\d x)$ is a uniformly bounded martingale. Taking $\varphi = \mathbb I_{A}$ the indicator function on any Borel set $A\subseteq\R^d$ and a stopping time $T$, Doob's optional stopping theorem implies that $\E[\mu_T(A)] = \mu(A)$, thereby decomposing $\mu$ as a mixture of $\mu_T$. In applications, $C_t$ and $T$ are often chosen so that $\mu_T$ is either a point mass or supported on some low-dimensional subspace with prescribed properties \cite{eldan2020taming, eldan2022spectral}.

\subsection{A Bayesian perspective}\label{sec:1.2}

In the original work of \citet{eldan2013thin} and a subsequent work \cite{eldan2020clt}, an SL scheme is defined by first specifying the coupled dynamics of $\t_t, G_t$ (or some equivalent variants) in \eqref{finite_system} and then defining SL using \eqref{pdensity}, which is a smooth function of $\t_t, G_t$ indexed by $x$. It is thus helpful to get more intuition about the processes $\t_t$ and $G_t$. With benign abuse of notation, we retain the same notation as before but define it from a different perspective. 

Let $X\sim\mu$, and let $W_t$ be a standard Brownian motion in $\R^d$ independent of $X$, which serves as a noise process. Consider an observation process $\t_t$ defined by the following SDE: 
\begin{align}
\d\t_t = C_tC^\top_tX \d t + C_t \d W_t,\label{obs}
\end{align}
where $C_t = C(t, \t_t)\in\R^{d\times d}$ is a continuous matrix-valued function. 

One may think of \eqref{obs} appearing in a scenario where $X$ can only be observed through some observation process $\t_t$. See Figure~\ref{fig:0} for an illustration where $C_t \equiv I$ and $\t_t = tX + W_t$. 
\begin{figure}[htbp]
  \centering 
  \begin{subfigure}{0.40\textwidth}{\includegraphics[width=\linewidth, trim={1.2cm 1cm 3cm 0cm},clip]{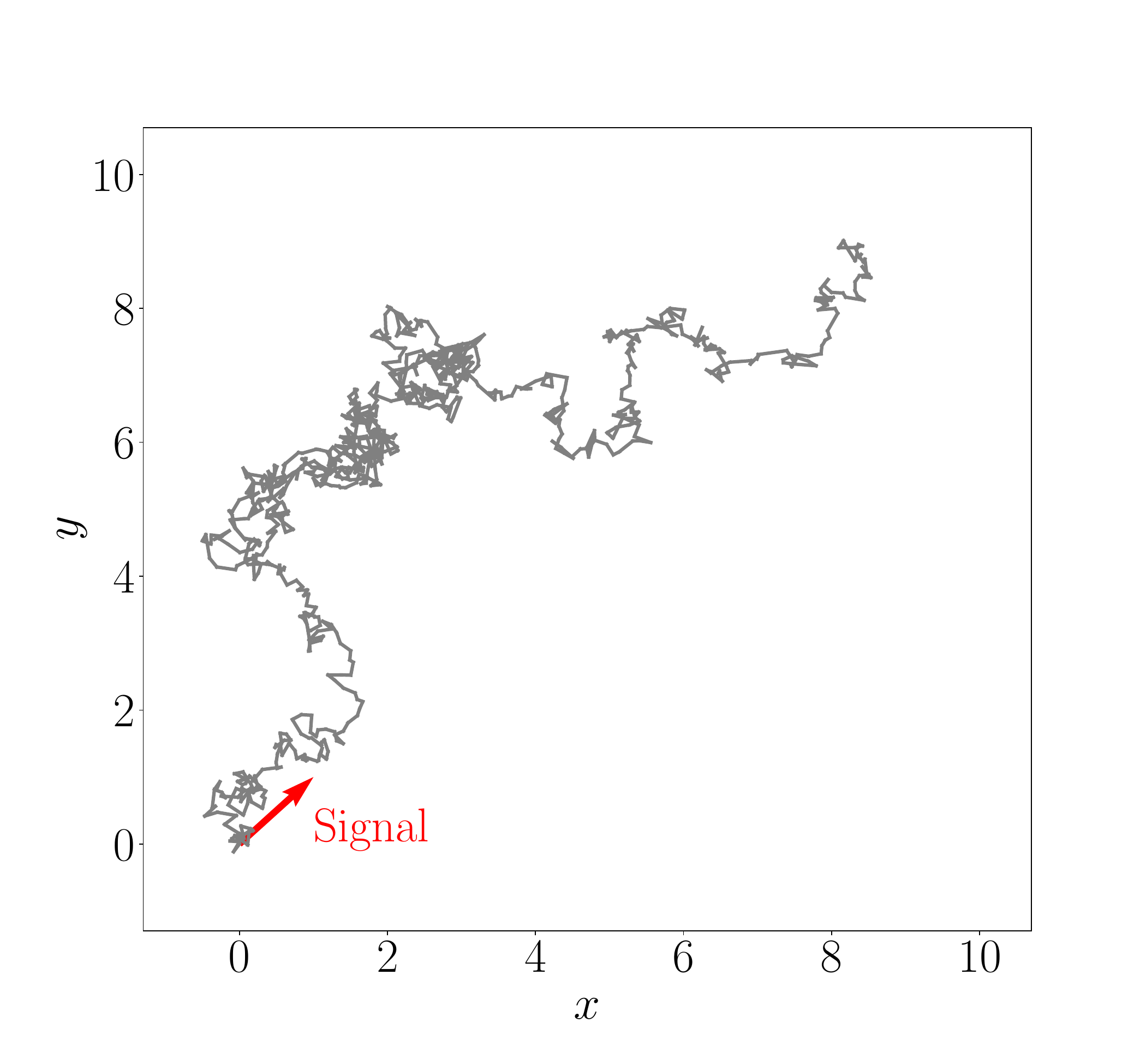}}
\end{subfigure}\hspace{2cm}
  \begin{subfigure}{0.40\textwidth}{\includegraphics[width=\linewidth, trim={1.2cm 1cm 3cm 0cm},clip]{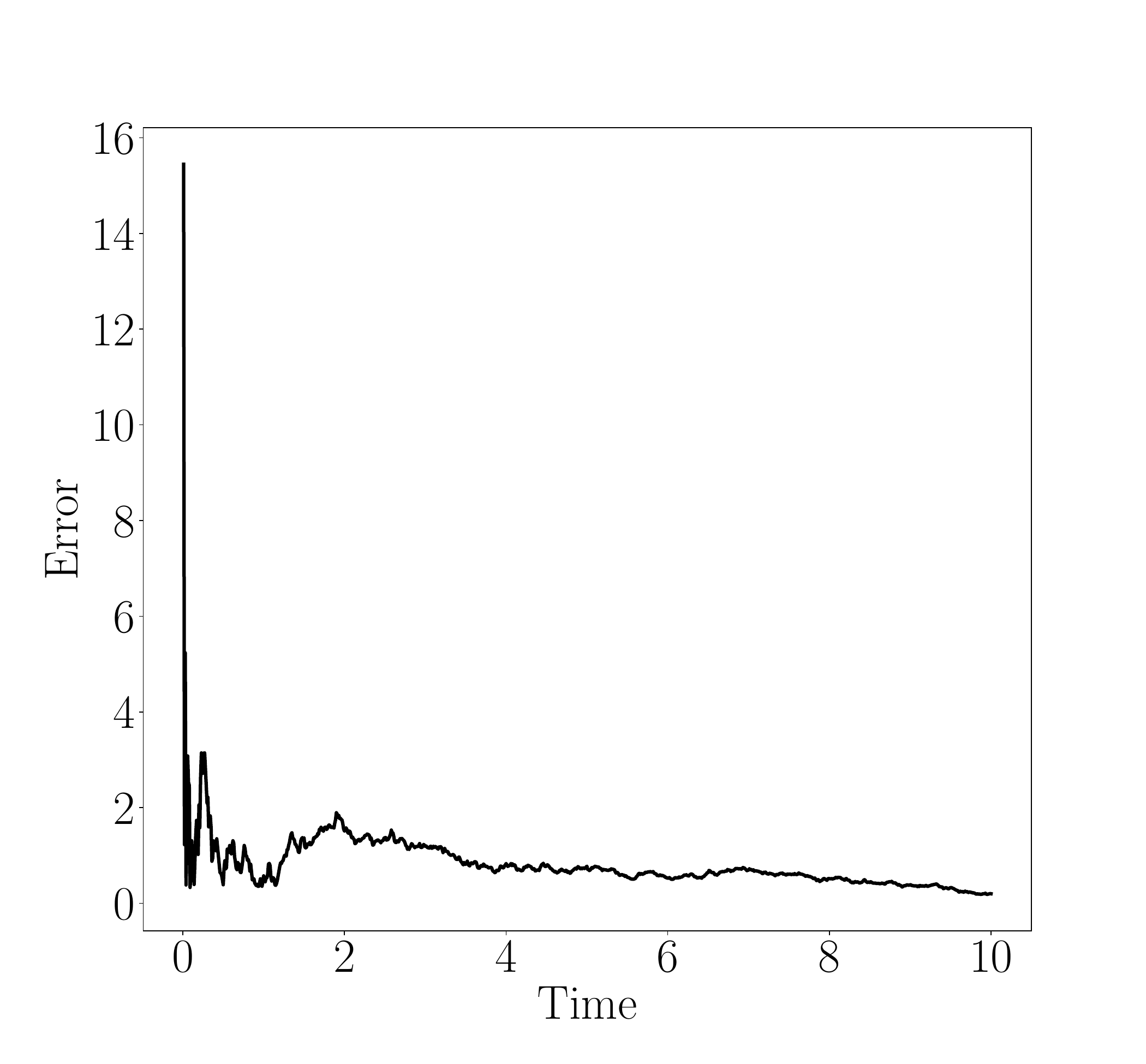}}
\end{subfigure}
\caption{A trajectory of the observation process $\t_t = tX + W_t$ with $X = (1, 1)^\top$ (left). As $t$ grows, the signal-to-noise ratio increases and thus $\t_t$ becomes increasingly informative of the unobserved signal $X$. This can be measured by computing the $\ell_2$ distance (error) between the normalized signal $\t_t/t$ and $X$ (right). } \label{fig:0}
\end{figure} 
Since the drift in \eqref{obs} is not adapted to $W_t$, one cannot simulate $\t_t$ without knowing $X$. Fortunately, we can apply a time-dependent variant of \cite[Theorem 8.4.3]{oksendal2003stochastic} to obtain an \ito~diffusion that has the same law as \eqref{obs}: 
\begin{equation}\label{obs1}
\begin{aligned}
\begin{cases}
\d\t_t &= C_tC^\top_t\E[X\mid\G_t] \d t + C_t \d W_t\\
\t_0 &= 0
\end{cases},
\end{aligned}
\end{equation}
where $\G_t$ is the filtration generated by $\t_t$. 

Viewing \eqref{obs} from a Bayesian perspective, one can consider an associated filtering problem of computing the posterior distribution of $X$ given the observation history $\G_t$ using Bayes' formula. In finite dimensions, the posterior distribution of a random vector $Z_1\in\R^{d_1}$ given a coupled observation $Z_2\in\R^{d_2}$ is proportional to their joint density $p(z_1, z_2)$ with respect to the Lebesgue measure on $\R^{d_1+d_2}$ (provided it exists). This can be generalized to other scenarios by replacing Lebesgue measure with an appropriate reference measure. In our setting, the probability space is the product space $\R^d\times C(\R_+, \R^d)$ ($C(\R_+, \R^d)$ is equipped with the Borel $\sigma$-algebra associated with the topology of uniform convergence on every compact set). One can take the reference measure as the product measure $\mu(\d x)\otimes\L_{\t_t'}(\d \omega)$, where $\L_{\t_t'}(\d\omega)$ is the law of the diffusion process $\d\t'_t = C_t \d W_t$. Let $\L_{x, \t_t}(\d x, \d\omega)$ be the joint distribution of $X$ and $\{\t_s\}_{0\leq s\leq t}$. Under suitable regularity assumptions on $C_t$, $\L_{x, \t_t}(\d x, \d\omega)$ is absolutely continuous with respect to $\mu(\d x)\otimes\L_{\t_t'}(\d \omega)$ and the Radon--Nikodym derivative can be computed using a version of Girsanov's theorem \cite[Theorem 8.6.8]{oksendal2003stochastic}: 
\begin{align}
\frac{\d \L_{x, \t_t}(\d x, \d\omega)}{\d (\mu(\d x)\otimes\L_{\t_t'}(\d \omega))}\bigg |_{x, \{\t_s\}_{0\leq s\leq t}} &= \exp\left\{\frac{1}{2}\<x,  G_t x\>+ \<\int_0^t C_s \d W_s, x\>\right\} \nonumber\\
& \stackrel{\eqref{obs}}{=} \exp\left\{\<\t_t, x\> - \frac{1}{2}\<x,  G_t x\>\right\},\label{girsanov}
\end{align}
where $G_t = \int_0^t C_sC_s^\top \d s$. If we denote the conditional distribution of $X$ on $\G_t$ as $\mu_t$, then an application of Bayes' formula yields
\begin{equation}\label{mu_t}
\begin{aligned}
\mu_t(\d x)&=\frac{\exp\left\{\<\t_t, x\> - \frac{1}{2}\<x,  G_t x\>\right\}\mu(\d x)\L_{\t'}(\d \omega)}{\int_{\R^d}\exp\left\{\<\t_t, z\> - \frac{1}{2}\<z,  G_t z\>\right\}\mu(\d z)\L_{\t'}(\d \omega)},
\end{aligned}
\end{equation}
which recovers the solution to the SL scheme in \eqref{sl}.   

\begin{remark}\label{rem:markov}
The above derivations suggest that an SL process $\mu_t$ associated with $\mu$ can be interpreted as the posterior distribution of $X$ conditional on the observation process $\{\t_s\}_{0\leq s\leq t}$. Similarly, the process $a_t$ can be interpreted as the posterior mean of $X$. In general, $a_t$ depends on the whole trajectory $\{\t_s\}_{0\leq s\leq t}$. When $C_t \equiv I$, $a_t$ is a deterministic function of the terminal point $\t_t$ only, i.e., $a_t = a_t(\t_t)$. Further interpretation of $\t_t$ can be found in \Cref{sec:dan}. 
\end{remark}

\subsection{SL calculus}\label{sec:slcal}

The evolution of SL in \eqref{sl} encodes all information about its associated statistics.
Using stochastic calculus, we derive several useful formulas for the moments of $\mu_t$.
For convenience, we assume that $\mu$ has bounded support; the derivations extend to more general conditions on $\mu$.

We first derive the dynamics of the mean process $a_t$:
\begin{equation}\label{at}
\begin{aligned}
\d a_t &= \int_{\R^d} x \d p_t(x)\mu(\d x)\\
& \stackrel{\eqref{sl}}{=} \int_{\R^d} x\otimes (x-a_t) p_t(x)\mu(\d x)C_t\d W_t\\
& = \int_{\R^d} (x-a_t)\otimes(x-a_t)\mu_t(\d x)C_t\d W_t\\
& = \Sigma_tC_t\d W_t,  
\end{aligned}
\end{equation}
where $\Sigma_t = \int_{\R^d} (x-a_t)^{\otimes 2} \mu_t(\d x)$ is the covariance of $\mu_t$. Therefore, $a_t$ is a local martingale. 
When $\mu$ has bounded support, $a_t$ is bounded and hence is a martingale. Its quadratic variation process is given by
\begin{align}
\d[a_t\otimes a_t] = \Sigma_t C_t C_t^\top\Sigma_t \d t.\label{a-qv}
\end{align}

Similarly, we can compute the dynamics of $\Sigma_t$ as follows:
\begin{equation}\label{dcov}
\begin{aligned}
\d \Sigma_t & =-\Sigma_tC_tC_t^\top\Sigma_t \d t + \M^{(3)}[\mu_t] C_t \d W_t, 
\end{aligned}
\end{equation}
where $\M^{(3)}[\mu_t] = \int_{\R^d}(x-a_t)^{\otimes 3} \mu_t(\d x)$ denotes the third-order centered moment of $\mu_t$. Taking the trace on both sides of \eqref{dcov} and integrating yields
\begin{align}
\tr(\Sigma_t) - \tr(\Sigma_0)  + \int_0^t\tr(\Sigma_sC_sC_s^\top\Sigma_s) \d s = \int_0^t\tr(\M^{(3)}[\mu_s] C_s \d W_s)\eqqcolon M_t. \label{myMt}
\end{align}
Since $M_t$ is an \ito\ integral and lower bounded by $-\tr(\Sigma_0)$, $M_t$ is a supermartingale. 
This observation will recur throughout the remainder of the article, and we state it as a lemma.
\begin{lemma}\label{tom's lemma}
The stochastic process $M_t$ defined in \eqref{myMt} is a supermartingale with respect to the filtration generated by $W_t$. In particular, 
\begin{align*}
\E[\tr(\Sigma_t)]\leq \tr(\Sigma_0)  - \int_0^t\E\left[\tr(\Sigma_sC_sC_s^\top\Sigma_s)\right] \d s. 
\end{align*}
The equality holds if $M_t$ is a martingale.  
\end{lemma}

The computations in \eqref{at} and \eqref{dcov} suggest a moment-generating property of SL: a direct simulation of the dynamics of the $k$th moments of $\mu_t$ requires information on the $(k+1)$th moments. To better understand this moment-generating phenomenon, consider the special case where $C_t \equiv I$. Under this condition, $G_t = tI$ is deterministic, and the observation process $\theta_t$ is the only source of randomness in \eqref{finite_system}. By defining the following tilted cumulant generating function:
\begin{align*}
K_t(\t) = \log\E_{x\sim\mu}\left[\exp\left\{-\frac{t}{2}\|x\|_2^2+\<\t, x\>\right\}\right], 
\end{align*}
we can express $a_t$ and $\Sigma_t$ using the derivatives of $K_t$ evaluated at $\t_t$:
\begin{align}
a_t = \nabla_\t K_t(\t_t),\quad\quad \Sigma_t = \nabla^2_\t K_t(\t_t). \label{mgf}
\end{align}

Note that $K_t$ is a deterministic function determined by $\mu$. The interaction between $a_t$ and $\Sigma_t$ depends closely on the regularity of $K_t$. For certain choices of $\mu$, $K_t$ exhibits favorable properties that can be exploited for obtaining sharp analyses for SL-based sampling. One important example is the class of log-concave measures, which plays a central role in modern sampling theory \cite{Chewi26Book}.

\begin{definition}[Log-concave measures on $\R^d$]
Consider a (nondegenerate) probability measure on $\R^d$ that has density $p$ with respect to Lebesgue measure. The measure is called log-concave if $\nabla^2 \log p \preceq 0$ at all $x \in \mathbb{R}^d$. If there is a $\kappa > 0$ such that $\nabla^2\log p\preceq -\kappa I$ at all $x \in \mathbb{R}^d$, then the measure is called $\kappa$-strongly log-concave.
\end{definition}

For a log-concave measure $\mu$, its associated SL scheme $\mu_t$ is more log-concave than $\mu$ a.s. (by \eqref{mu_t}), and this log-concavity often strictly increases over time. This property makes the analysis of SL for log-concave distributions particularly tractable. In the special case where $\mu$ is Gaussian, $\mu_t$ remains Gaussian (albeit random) a.s. Since Gaussian measures have vanishing third-order centered moments, the diffusion term in \eqref{dcov} vanishes, and consequently $\Sigma_t$ evolves deterministically. Throughout this article we will repeatedly invoke this property when analyzing the dynamics of SL in the Gaussian setting.


\section{Control processes}\label{sec:3}
The property of an SL scheme depends on the control process $C_t$, and different choices of $C_t$ have been considered in various settings. For instance, $C_t = I$ was used in recent progress on the KLS conjecture \cite{lee2017eldan, chen2021almost, jambulapati2022slightly, klartag2022bourgain, guan2024note}, Bourgain's hyperplane conjecture \cite{klartag2024affirmative, guan2024note}, and the thin-shell conjecture \cite{anttila2003central}. Moreover, it has been utilized as a sampling algorithm to sample from a target distribution \cite{el2022sampling, montanari2023posterior, grenioux2024stochastic, demyanenko2025sampling}. An algorithmic advantage of this choice is that simulating $\t_t$ requires only the first-order information of $\mu_t$ and is thus relatively easy to implement. 

In Eldan's original work \cite{eldan2013thin}, $C_t = \Sigma_t^{-1/2}$ was used to isotropize the SL dynamics, where $\Sigma_t$ is the covariance of $\mu_t$. Since $\Sigma_t$ is shrinking on average (see \Cref{thm:01}), it is unclear whether $C_t$ is well-defined for all finite times. \citet{eldan2013thin} showed that for log-concave measures, $\Sigma_t$ is nonsingular for $t\geq 0$ a.s., and based on this established the unique existence of the SL process. The proof relies on a rescaling property of log-concave measures. However, this property does not generally hold for other measures (see Example~\ref{examB}).

In subsequent works on Skorokhod embedding \cite{eldan2016skorokhod, eldan2020clt}, a more extreme choice $C_t = \Sigma_t^{\dagger}$ was employed, where $\dagger$ denotes the pseudoinverse. They aimed to embed $\mu$ as a stopped Brownian motion so that $\Sigma_t$ vanishes within finite time a.s. Importantly, they showed that as long as $\mu$ has bounded support and a smooth density and $C_t$ has a continuous dependence on $\Sigma^\dagger_t$ and $t$, the corresponding SL process is well-defined for all $t>0$ via an induction argument. In particular, whenever $\mu_t$ becomes degenerate, it is supported on a lower-dimensional subspace, as ensured by the regularity assumption on $\mu$, and a new SL process can be defined within that subspace. This procedure is repeated until $\mu_t$ becomes fully localized; see \cite[Section 2.3]{eldan2020clt} for additional details. Besides the abovementioned choices of $C_t$, there are other constructions based on smoothened projections \cite{eldan2022spectral}. 

\subsection{Eldan's $\alpha$-scheme and localization rates}\label{sec:rateuse}

Motivated by the existing choices, we consider a family of SL schemes with control process $C_t$ defined as powers of $\Sigma^\dagger_t$ as follows. Given $0 \leq \alpha\leq 1$, define 
\begin{align}
C_t = (\Sigma^\dagger_t)^\alpha.\label{alphaE}
\end{align}
We call the SL scheme associated with \eqref{alphaE} an \emph{Eldan's $\alpha$-scheme}. Previously mentioned schemes are special cases where $\alpha = 0$, $1/2$, and $1$. 

A natural question is how $\alpha$ affects the corresponding SL scheme. One may approach this from many different angles. Since our focus is on the algorithmic aspect, we consider the average speed of localization measured by $\E[\tr(\Sigma_t)]$. 

\begin{theorem}\label{thm:01}
Let $\mu$ be a probability measure on $\mathbb{R}^d$ with $\supp(\mu)\subseteq\B_R(\R^d)$ for some $R>0$ and a smooth density, and let $\alpha \in [0, 1]$. The SL scheme in \eqref{sl} with $C_t$ defined in \eqref{alphaE} has a unique solution $\mu_t$. Denoting the covariance of $\mu_t$ as $\Sigma_t$, then the following statements hold:  
\begin{enumerate}
\item [(i)] If $0\leq\alpha<1/2$, 
\begin{align*}
\E[\tr(\Sigma_t)] \leq \left[ \frac{(1-2\alpha)t}{d^{1-2\alpha}} + \frac{1}{\tr(\Sigma_0)^{1-2\alpha}} \right]^{-\frac{1}{1-2\alpha}}\leq \frac{d}{[(1-2\alpha)t]^{\frac{1}{1-2\alpha}} + \frac{d}{\tr(\Sigma_0)}};
\end{align*} 
\item [(ii)]  If $\alpha = 1/2$, then $\E[\tr(\Sigma_t)] = \tr(\Sigma_0)e^{-t}$;
\item [(iii)] If $1/2<\alpha\leq 1$, then $\E[\tr(\Sigma_t)]\leq \tr(\Sigma_0)\exp\left(-R^{-2(2\alpha-1)}t\right)$. 
\end{enumerate}
Moreover, for $0\leq\alpha\leq 1/2$, the process $M_t$ defined in \eqref{myMt} is a martingale with respect to the filtration generated by $W_t$.  
\end{theorem}
\begin{proof}
The unique existence of Eldan's $\alpha$-scheme follows from \cite[Propositions 1-2]{eldan2020clt}. Therefore, it remains to check the convergence rate of $\E[\tr(\Sigma_t)]$. Under the choice $C_t = (\Sigma_t^\dagger)^\alpha$ and applying \Cref{tom's lemma}, 
\begin{align}
\E[\tr(\Sigma_t)]\leq -\int_0^t\E[\tr(\Sigma^{2-2\alpha}_s)] \d s + \E[\tr(\Sigma_0)]. \label{mean-cov}
\end{align}
We now use \eqref{mean-cov} to bound $\E[\tr(\Sigma_t)]$ in two separate cases.

When $0\leq\alpha\leq 1/2$, since $x\mapsto x^p$ is operator convex for $p\geq 1$, $\E[\Sigma^{2-2\alpha}_t]\succeq \E[\Sigma_t]^{2-2\alpha}$. Consequently, 
\begin{align}
-\E[\tr(\Sigma^{2-2\alpha}_t)]\leq-\tr(\E[\Sigma_t]^{2-2\alpha})\leq -\frac{\E[\tr(\Sigma_t)]^{2-2\alpha}}{d^{1-2\alpha}}, \label{mod1}
\end{align}
where the second step follows from Jensen's inequality applied to the spectrum of $\E[\Sigma_t]$. Substituting \eqref{mod1} into \eqref{mean-cov} and applying Gr\"{o}nwall's inequality yields (i). When $\alpha = 1/2$, the inequalities in \eqref{mod1} are equalities. Moreover, since the process $M_t$ defined in \eqref{myMt} is a martingale (which is shown in the end), \eqref{mean-cov} becomes an equation. In this case, $\E[\tr(\Sigma_t)]$ can be exactly solved to obtain (ii).  

When $1/2<\alpha\leq 1$, $0\leq 2-2\alpha<1$. By the elementary inequality $(a+b)^p\leq a^p + b^p$ for $a, b>0$ and $0<p<1$, 
\begin{align}
-\E[\tr(\Sigma^{2-2\alpha}_t)] \leq -\E\left[\tr(\Sigma_t)^{2-2\alpha}\right]\leq -R^{-2(2\alpha-1)}\E[\tr(\Sigma_t)], \label{mod3}
\end{align}
where the last step follows from the observation that $0\leq\tr(\Sigma_t)\leq\E_{x\sim\mu_t}[\|x\|_2^2]\leq R^2$. Substituting \eqref{mod3} into \eqref{mean-cov} and applying Gr\"{o}nwall's inequality yields (iii). 

Finally, we show that for $0\leq\alpha\leq 1/2$, $M_t$ is a martingale. By definition, the quadratic variation process of $M_t$ follows the dynamics given by
\begin{align*}
[\d M_t, \d M_t] = \left\|\E_{x\sim\mu_t}\left[\|x-a_t\|_2^2 (\Sigma_t^\dagger)^\alpha (x-a_t)\right]\right\|_2^2 \d t,  
\end{align*}
where $a_t$ is the mean of $\mu_t$. It suffices to show that the integrand is bounded. Since $\supp(\mu_t)\subseteq\B_R(\R^d)$, $\|x-a_t\|_2\leq 2R$ and $\|\Sigma_t\|_2\leq R^2$. The desired result follows by first applying Jensen's inequality to pass the norm inside the expectation, followed by the Cauchy--Schwarz inequality: 
\begin{equation}\label{newcheck}
\begin{aligned}
\left\|\E_{x\sim\mu_t}\left[\|x-a_t\|_2^2 (\Sigma_t^\dagger)^\alpha (x-a_t)\right]\right\|_2^2&\leq \E_{x\sim\mu_t}[\|x-a_t\|_2^4]\cdot\E_{x\sim\mu_t}\left[\|(\Sigma_t^\dagger)^\alpha(x-a_t)\|_2^2\right]\\
&\leq (2R)^4\cdot \tr(\Sigma_t^{1-2\alpha})\leq (2R)^4\cdot d(R^2)^{1-2\alpha}.  
\end{aligned}
\end{equation}
\end{proof}

\begin{remark}\label{rem:exist} 
The smooth density and bounded support assumptions are mainly needed for the existence of the SL scheme. There are other scenarios where Theorem~\ref{thm:01} remains valid; for instance, when $\mu$ is log-concave or has a finite support. In both cases, the existence of Eldan's $\alpha$-scheme follows by applying a similar induction argument in \cite[Proposition 1]{eldan2020clt}. For $\kappa$-strongly log-concave measures with unbounded support, \eqref{mod3} can be adapted by noting that $\mu_t$ remains $\kappa$-strongly log-concave on its affine support so that $\tr(\Sigma_t)\leq d/\kappa$ a.s. \cite[Lemma 1]{eldan2014bounding} and \eqref{newcheck} follows by bounding the fourth moment of $\mu_t$ using the thin-shell constant, see, for example, \cite[Theorem 2]{jambulapati2022slightly}. When $\alpha=0$, assuming $\mu$ has a bounded first moment is sufficient to ensure the SL uniquely exists \cite[Theorem 7.1.2]{lipster1977statistics} and hence the corresponding localization rate in Theorem~\ref{thm:01} holds. 
\end{remark}

For $0\leq\alpha\leq 1/2$, the localization rate of Eldan's $\alpha$-scheme increases from polynomial to exponential in time. For $\alpha>1/2$, the localization rate is at least exponential. While one might expect faster-than-exponential convergence in this regime, this is not generally the case, as the following example demonstrates.

\begin{myexample}\label{examB}
Let $\mu = (\delta_0 + \delta_1)/2$ and consider its associated Eldan's $\alpha$-scheme $\mu_t$. In this case, $\mu_t$ is determined by the density $p_t\coloneqq p_t(1)$, which follows the SDE 
\begin{align*}
\begin{cases}
\d p_t &= 2^{2\alpha-1} [p_t(2-p_t)]^{1-\alpha} \d W_t\\
p_0 &\equiv 1
\end{cases}
\end{align*}
up to the stopping time $\tau = \inf\{t \geq 0: p_t = 0\ \text{or}\ p_t = 2\}$. We further define a strictly earlier stopping time $\tau' = \inf\{t \geq 0: p_t = 1/2\ \text{or}\ p_t = 3/2\} < \tau$. 

When $\alpha = 1$, $p_t = 2W_t + 1$ prior to stopping. Consequently, $\tau$ corresponds to the first exit time of standard Brownian motion $W_t$ from the interval $[-1/2, 1/2]$, while $\tau'$ corresponds to the exit time from $[-1/4, 1/4]$. Since $p_t(2-p_t) \geq 3/4$ on the event $\{\tau' \geq t\}$, a lower bound on $\E[\tr(\Sigma_t)]$ can be established as:
\begin{align*}
\E[\tr(\Sigma_t)] = \frac{1}{4}\E\left[p_t(2-p_t)\mathbb I_{\{\tau\geq t\}}\right] \geq \frac{1}{4}\E\left[p_t(2-p_t)\mathbb I_{\{\tau'\geq t\}}\right] \geq \frac{3}{16}\P(\tau'\geq t). 
\end{align*}
The stopping time $\tau'$ is the first exit time of a standard Brownian motion from a bounded symmetric interval, which is a well-studied object in classical probability theory known to exhibit an exact exponential tail in $t$ \citep{karatzas2014brownian}. This confirms that the expected trace is globally bounded below by an exponential function.

When $\alpha = 1/2$, the SDE simplifies to $\d p_t = \sqrt{p_t(2-p_t)} \d W_t$, which is a scaled instance of the standard Wright--Fisher diffusion with absorbing boundaries \cite[Exercise 5.5]{durrett2018stochastic}. For this diffusion, it is a classical result that $p_t$ hits either $0$ or $2$ in finite time a.s. 
\end{myexample}

On the other hand, for certain distributions, the exponential bound when $\alpha>1/2$ can be further improved to uniform finite-time localization.
 
\begin{myexample}
If $\mu$ is Gaussian, then $\mu_t$ remains Gaussian a.s. according to \cref{mu_t}. Because the third-order centered moments of Gaussian measures vanish, the diffusion term in \cref{dcov} disappears. Consequently, the evolution of $\Sigma_t$ is deterministic, and \cref{mean-cov} holds pathwise without taking the expectation. This, combined with the first inequality in \cref{mod3}, yields 
\begin{align*}
\tr(\Sigma_t)\leq -\int_0^t \tr(\Sigma_s)^{2-2\alpha} \d s + \tr(\Sigma_0)\implies\text{$\Sigma_t = 0$ for $t\geq\frac{\tr(\Sigma_0)^{2\alpha-1}}{2\alpha-1}$ a.s.} 
\end{align*}
In fact, an exact formula for $\tr(\Sigma_t)$ can be found using \eqref{gausscov}. 
\end{myexample}
For general $\mu$, a similar finite-time localization phenomenon is expected if the reverse-Jensen type inequality $\E[\tr(\Sigma_t)^{2-2\alpha}] \gtrsim \E[\tr(\Sigma_t)]^{2-2\alpha}$ holds uniformly in $t$ for fixed $1/2 < \alpha \leq 1$. However, it is not clear if there exist other measures beyond Gaussians that satisfy this. Part of the technical difficulty in obtaining a faster rate for $1/2<\alpha\leq 1$ in general arises from using the expected trace. When working at the trajectory level, it can be shown that, for $\alpha = 1$, $\mu_t$ localizes in finite time a.s. \cite[Proposition 4]{eldan2020clt}, although this convergence time may not be uniformly bounded a.s.

\subsection{Regularization}\label{sec:reg}

In the rest of this article, we mainly focus on Eldan's $\alpha$-scheme for $\alpha = 0$ and $\alpha = 1/2$ due to their favorable properties for constructing couplings between measures within the joint SL framework introduced in \Cref{sec:jsl}. For $\alpha = 0$, the convergence rate is relatively slow, and nonuniform discretization can be employed to accelerate the computation; we discuss this in further detail in \Cref{sec:eldan0num}. For $\alpha = 1/2$, the localization rate is exponential in time $t$. However, when designing numerical algorithms to simulate this scheme, utilizing the pseudoinverse $\Sigma_t^\dagger$ introduces a continuity problem if $\Sigma_t$ becomes degenerate in finite time, as shown in Example~\ref{examB}.

To address this issue, we add a small multiple of the identity matrix to $\Sigma_t$. Fixing $\delta>0$ as a regularization parameter, consider the \textit{$\delta$-regularized Eldan's $\frac{1}{2}$-scheme} with
\begin{align}
&C_t = \Sigma_t(\delta)^{-\frac{1}{2}}, \quad\quad\Sigma_t(\delta)=\Sigma_t + \delta I, 
\end{align}
When $\mu$ has a bounded support, the condition number of $C_t$ is bounded by $\mathcal O(\delta^{-1/2})$, so a large $\delta$ yields better numerical stability. On the other hand, too large a value of $\delta$ may overregularize and thus slow down the localization rate. This is because when $\tr(\Sigma_t)\leq \delta$, the regularized SL scheme resembles Eldan's $0$-scheme (up to a multiplicative constant). The following result makes these observations rigorous.  

\begin{theorem}\label{thm:reg}
Let $\mu$ be a probability measure on $\R^d$ and $\delta>0$. If $\supp(\mu)\subseteq\B_R(\R^d)$ for some $R>0$, then the $\delta$-regularized Eldan's $\frac{1}{2}$-scheme $\mu_t$ exists for all $t\geq 0$, and its covariance $\Sigma_t$ satisfies 
\begin{align}
\E[\tr(\Sigma_t)]\leq\min\left\{\frac{d(R^2+\delta)}{t+\frac{d}{\tr(\Sigma_0)}}, \ \tr(\Sigma_0)e^{-t} + d\delta (1-e^{-t})\right\}.\label{newbt}
\end{align}
\end{theorem}

\begin{proof}
When $\mu$ has bounded support, $C_t$ remains uniformly bounded for all $t$. The existence of $\mu_t$ follows from the standard existence and uniqueness theorem of SDEs \cite[Theorem 5.2.1]{oksendal2003stochastic}. It remains to verify the bounds on the localization rate. 

For the first term in the upper bound in \cref{newbt}, applying \Cref{tom's lemma} yields  
\begin{align}
\E[\tr(\Sigma_t)]\leq -\int_0^t\E[\tr(\Sigma^2_s\Sigma^{-1}_s(\delta ))] \d s + \E[\tr(\Sigma_0)]. \label{reg-1}
\end{align}
Under the assumption $\supp(\mu)\subseteq\B_R(\R^d)$, $\|\Sigma_s(\delta))\|_2\leq R^2+\delta$, so that 
\begin{align}
\E[\tr(\Sigma^{2}_s\Sigma^{-1}_s(\delta))]\geq \frac{\tr(\E[\Sigma^2_s])}{R^2 + \delta}\geq\frac{\E[\tr(\Sigma_s)]^2}{d(R^2+\delta)}.\label{reg-2}
\end{align}
Substituting \cref{reg-2} into \cref{reg-1} and applying Gr\"{o}nwall's inequality yields that
\begin{align*}
\E [ \tr(\Sigma_t)] \leq \frac{d(R^2+\delta)}{t+\frac{d}{\tr(\Sigma_0)}}.
\end{align*}

To improve the estimate for small $t$, we consider an alternative bound for \eqref{reg-2}. Letting $\lambda_1, \ldots, \lambda_d\geq 0$ denote the eigenvalues of $\Sigma_s$, then 
\begin{align*}
\tr(\Sigma^{2}_s\Sigma^{-1}_s(\delta)) = \sum_{i=1}^d\frac{\lambda_i^2}{\lambda_i+\delta}\geq\sum_{i=1}^d(\lambda_i-\delta) = \tr(\Sigma_s) -d\delta. 
\end{align*}
Taking expectation on both sides and then substituting into \eqref{reg-1} yields
\begin{align*}
\E[\tr(\Sigma_t)]\leq -\int_0^t\E[\tr(\Sigma_s)] \d s + d\delta t + \E[\tr(\Sigma_0)].
\end{align*}
An application of the Gr\"{o}nwall type estimate yields
\begin{align*}
\E[\tr(\Sigma_t)]\leq \tr(\Sigma_0)e^{-t} + d\delta (1-e^{-t}). 
\end{align*}
The proof is complete. 
\end{proof}

\Cref{thm:reg} suggests that the localization rate of $\delta$-regularized Eldan's $\frac{1}{2}$-scheme is at least exponential in time until reaching $\log (1/d\delta)$, after which the regularization effect dominates and the localization rate transitions to $\mathcal O(t^{-1})$. The same regularization will be used when designing approximate numerical methods for simulating Eldan's $\frac{1}{2}$-scheme in \Cref{sec:eldan1/2num}.


\section{Joint stochastic localization}\label{sec:jsl}

SL has been recognized as a novel sampling method to sample from a target distribution. A less-explored application of SL is the construction of couplings. The first SL-based coupling construction appeared in the stability analysis of the Brunn--Minkowski inequality \cite[Section 5]{eldan2013thin}. While that application was mainly intended to advance theory, it can be further developed into an algorithm to construct couplings between distributions. 

Conceptually, the SL scheme in \eqref{sl} can be extended to the setting of multiple distributions driven by the same Brownian motion. The control processes are used to tune the correlation of these distributions during the time evolution. For instance, consider two probability measures $\mu$ and $\nu$ on $\R^d$, each associated with an SL scheme as follows:
\begin{equation}\label{coupling}
\begin{aligned}
\d p_t(x) &= p_t(x)\< x-a_{t}, C_{t}\d W_{t}\>, \quad\quad a_{t}  = \int_{\R^d} x \mu_t(\d x)\\
\d q_t(x) &= q_t(x)\< x-b_{t}, D_{t}\d W_{t}\>, \quad\quad b_{t} = \int_{\R^d} x \nu_t(\d x)
\end{aligned}
,
\end{equation}
and 
\begin{align*}
\mu_t (\d x) = p_t(x) \mu (\d x), \quad\quad \nu_t (\d x) = q_t(x) \nu (\d x),
\end{align*} 
where $W_{t}$ is a Brownian motion shared by $p_t$ and $q_t$, and $C_t$, $D_t$ are matrix-valued control processes adapted to $W_t$. This yields a pathwise dependence structure between $\mu_t$ and $\nu_t$ for every fixed $t$. Assuming that both $\mu_t$ and $\nu_t$ localize as $t \to \infty$, a coupling between $\mu$ and $\nu$ can be obtained from $(a_\infty, b_\infty)$, which are random points with marginals $\mu$ and $\nu$ by the martingale property of SL. In the rest of this section, we consider different choices of $C_t$ and $D_t$. 

\subsection{Joint Eldan's $\alpha$-scheme and induced distances}

Following the discussion in \Cref{sec:3}, we define the \emph{joint Eldan's $\alpha$-scheme} between two probability measures $\mu$ and $\nu$ via \eqref{coupling} with their respective control processes chosen as  
\begin{align}
C_t = (\Sigma_t^\dagger)^{\alpha}, \quad\quad D_t = (\Lambda_t^\dagger)^{\alpha}, \label{eldan_control}
\end{align}
where $\Sigma_t$ and $\Lambda_t$ are the covariance processes of $\mu_t$ and $\nu_t$:
\begin{align*}
  \Sigma_t = \E_{\mu_t}[x\otimes x]-a_t\otimes a_t,\quad\quad \Lambda_t = \E_{\nu_t}[x\otimes x]-b_t\otimes b_t.
\end{align*}
In particular, $\mu_t$ and $\nu_t$ are Eldan's $\alpha$-schemes associated with $\mu$ and $\nu$ driven by the same Brownian motion $W_t$, and are therefore defined on the same probability space. See \Cref{fig:emoji} for an illustration. 

\begin{figure}[htbp]
  \centering
  \includegraphics[width=0.9\textwidth, trim={1cm 0cm 2cm 1cm}, clip]{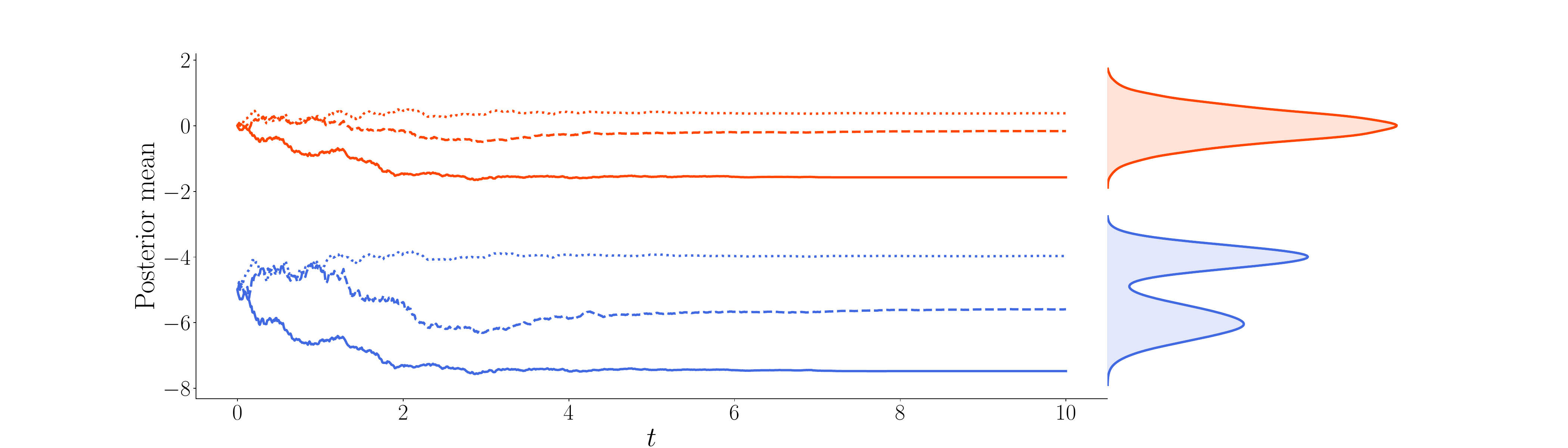}\hfill
  \caption{Three trajectories of the mean process $(a_t, b_t)$, marked by different shapes, produced by the joint Eldan's $\tfrac{1}{2}$-scheme for a Gaussian distribution $\mathcal N(0, 0.5^2)$ (top) and a Gaussian mixture $0.5\mathcal N(-6, 0.5^2) + 0.5\mathcal N(-4, 0.3^2)$ (bottom). In this one-dimensional example, there is a monotonic correspondence between these paths when localized, which empirically suggests a nontrivial coupling induced by the joint's Eldan $\tfrac{1}{2}$-scheme.}\label{fig:emoji}
\end{figure}

A joint Eldan's $\alpha$-scheme is well defined for all $t\geq 0$ whenever each marginal is. Under such circumstances, \Cref{thm:01} implies that the mean process $(a_t, b_t)$ converges to some limit $(a_\infty, b_\infty)$ as $t\to\infty$ a.s. Moreover, $(a_\infty, b_\infty)$ are joint random variables with $a_\infty\sim\mu$ and $b_\infty\sim\nu$ thanks to the martingale structure of SL. This coupling is canonical and can be consistently extended to any countable set of probability measures on $\R^d$ as long as Eldan's $\alpha$-scheme is defined for each measure in the set. When $\alpha = 0$, the resulting coupling was recently used by \citet{klartag2025thin} when resolving the thin-shell conjecture under the name of parallel coupling. 

The coupling between $\mu$ and $\nu$ through $(a_\infty, b_\infty)$ induces a distance between $\mu$ and $\nu$.  
Fix $\alpha\in [0, 1]$ and denote by $\mathcal{P}_\alpha(\R^d)$ the set of probability measures on $\R^d$ for which Eldan's $\alpha$-scheme is well defined. As noted in Remark~\ref{rem:exist}, probability measures with bounded support and smooth densities, log-concave measures, or measures with finite support belong to $\mathcal{P}_\alpha(\R^d)$ for all $\alpha \in [0, 1]$. However, an exact characterization of $\mathcal{P}_\alpha(\R^d)$ is difficult unless $\alpha = 0$. 

\begin{definition}[Eldan's $\alpha$-distance]
Given $\mu, \nu\in\mathcal P_\alpha(\R^d)$, the Eldan's $\alpha$-distance between $\mu$ and $\nu$ is defined by 
\begin{align}\label{eldan-dis}
\ds(\mu, \nu) = \E[\|a_\infty-b_\infty\|_2^2]^{\frac{1}{2}},
\end{align}
where $(a_\infty$, $b_\infty)$ is the limit of the mean process $(a_t, b_t)$ of a joint Eldan's $\alpha$-scheme $(\mu_t, \nu_t)$ associated with $\mu$ and $\nu$. 
\end{definition}

By definition, $\ds$ is nonnegative, satisfies $\ds(\mu,\nu)=0$ if and only if $\mu=\nu$, is symmetric, and obeys the triangle inequality; hence, it is a well-defined metric on $\mathcal P_\alpha(\R^d)$. Applying \eqref{at} to $a_t$ and $-b_t$, respectively, 
\begin{align}
\d (a_t-b_t) &= (\Sigma^{1-\alpha}_t - \Lambda^{1-\alpha}_t)\d W_t, \label{itotimes}
\end{align}
where $(\Sigma_t, \Lambda_t)$ is the covariance process of $(\mu_t, \nu_t)$. This shows that $a_t-b_t$ is a local martingale. Further checking the moment of its quadratic variation process using \Cref{tom's lemma} shows it is a martingale: 
\begin{align*}
\E\left[\int_0^t\tr((\Sigma^{1-\alpha}_s - \Lambda^{1-\alpha}_s)^2) \d s\right] &= \int_0^t \E\left[\tr(\Sigma^{2-2\alpha}_s + \Lambda^{2-2\alpha}_s -2\Sigma_s^{1-\alpha}\Lambda_s^{1-\alpha})\right] \d s\\
&\leq \tr(\Sigma_0)+\tr(\Lambda_0)-\E[\tr(\Sigma_t)]-\E[\tr(\Lambda_t)] - 2\int_0^t\E[\tr(\Sigma_s^{1-\alpha}\Lambda_s^{1-\alpha})] \d s\\
&\leq \tr(\Sigma_0)+\tr(\Lambda_0)<\infty.  
\end{align*}

Assuming the trace of the diffusion terms driving $\Sigma_t$ and $\Lambda_t$ (i.e., the process $M_t$ defined in \Cref{tom's lemma}) are martingales, the first inequality becomes equality. In this case, by \ito's isometry and Doob's martingale convergence theorem, taking $t\to\infty$ and noting $\E[\tr(\Sigma_t)]+ \E[\tr(\Lambda_t)]\to 0$ (as implied by \Cref{thm:01}) yields 
\begin{equation}\label{keyobs}
\begin{aligned}
\ds^2(\mu, \nu) = \lim_{t\to\infty}\E[\|a_t-b_t\|_2^2]=\tr(\Sigma_0) + \tr(\Lambda_0) - 2\int_0^\infty \E[\tr(\Sigma_s^{1-\alpha}\Lambda_s^{1-\alpha})] \d s. 
\end{aligned}
\end{equation}
Further analysis of $\ds(\mu, \nu)$ is provided in \Cref{sec:dist}. For now, we use \eqref{keyobs} to compute $\ds(\mu, \nu)$ for several special distributions to get some intuition on its induced geometry. 

We first consider the case of centered Gaussian measures, for which $\ds$ admits a closed-form representation; this can be easily generalized to the noncentered case using \Cref{prop:shift}~(i). Given the canonical bijection between centered Gaussian measures on $\R^d$ and the cone of $d \times d$ positive semidefinite matrices, $\ds$ naturally induces a metric on the latter. This induced metric is equivalent to a kernel embedding distance under a suitable choice of feature map.

\begin{theorem}\label{thm:gauss}
Let $\mathbb{S}^d$ denote the space of $d\times d$ symmetric matrices, and let $\psd$ denote the cone of strictly positive matrices. For any $\Sigma \in \mathbb{S}_{++}^d$, let $\Sigma = \sum_{i=1}^r \lambda_i P_i$ be its unique spectral decomposition, where $\{\lambda_i\}_{i=1}^r$ are its strictly positive distinct eigenvalues and $P_r$ are the orthogonal projection matrices onto the corresponding eigenspaces.
Consider the feature map $\Phi_\alpha: \psd \to \H\coloneqq L^2(\mathbb{R}_+)\otimes\mathbb{S}^d$ defined by
\begin{equation}
\Phi_\alpha(\Sigma) = \sum_{i=1}^r \varphi_\alpha(\lambda_i)\otimes P_i,
\end{equation}
where the scalar feature map $\varphi_\alpha: \R_+ \to L^2(\R_+)$ is given by
\begin{equation*}
\begin{aligned}
\varphi_\alpha(x)(t) = \begin{cases}
(x^{-(1-2\alpha)}+(1-2\alpha) t)^{-\frac{1-\alpha}{1-2\alpha}}& \ 0\leq\alpha<\frac{1}{2}\\
\sqrt{x}e^{-\frac{t}{2}} & \ \alpha = \frac{1}{2}\\
\left(x^{2\alpha-1} - (2\alpha-1)t \right)_+^{\frac{1-\alpha}{2\alpha-1}}& \ \frac{1}{2}<\alpha\leq 1 
\end{cases}
\end{aligned},
\end{equation*}
with $(x)_+\coloneqq\max\{x, 0\}$, and $\H$ is equipped with the induced inner product $\langle f \otimes A, g \otimes B \rangle_\H = \<f, g\>_{L^2(\R_+)}\tr(A^\top B)$. For two centered nondegenerate Gaussian measures $\mu\sim \mathcal N(0, \Sigma)$ and $\nu\sim \mathcal N(0, \Lambda)$ on $\R^d$, their Eldan's $\alpha$-distance can be represented as $\ds(\mu, \nu) = \|\Phi_\alpha(\Sigma)-\Phi_\alpha(\Lambda)\|_{\H}$. Moreover, $\ds$ attains its minimum at $\alpha = 1/2$, i.e., $\dshalf(\mu, \nu)=\min_{0\leq\alpha\leq 1} \ds(\mu, \nu)$. 
\end{theorem}

\begin{remark}
To interpret the results, we consider a commutative setting where $\Sigma$ and $\Lambda$ are diagonal with diagonal entries $\sigma_1, \ldots, \sigma_d$ and $\lambda_1, \ldots, \lambda_d$, respectively. For convenience, we denote $\mu$ and $\nu$ by $\sigma = (\sigma_1, \ldots, \sigma_d)$ and $\lambda = (\lambda_1, \ldots, \lambda_d)$. Then, 
\begin{align*}
\dszero^2(\sigma, \lambda) &= \sum_{i=1}^d\sigma_i+\lambda_i -\frac{2\sigma_i\lambda_i\log(\sigma_i/\lambda_i)}{\sigma_i-\lambda_i}\\
\dshalf^2(\sigma, \lambda) &= \sum_{i=1}^d\sigma_i+\lambda_i -2\sqrt{\sigma_i\lambda_i} = \sum_{i=1}^d(\sqrt{\sigma_i}-\sqrt{\lambda_i})^2\\
\dsone^2(\sigma, \lambda) & = \sum_{i=1}^d\sigma_i+\lambda_i -2\min\{\sigma_i, \lambda_i\} = \sum_{i=1}^d|\sigma_i-\lambda_i|. 
\end{align*}
All three metrics quantify the entrywise discrepancy between $\sigma$ and $\lambda$ with some suitable means. In particular, $\dshalf$ coincides with the Hellinger distance, and $\dsone$ is the square root of twice the total variation distance. For $\dszero$, the mean used for comparison involves the dual of the logarithmic mean, which commonly arises in the study of Fisher information metric. In fact, assuming $\sigma$ is a probability vector,  for perturbation vector $\d s = (\d s_1, \ldots, \d s_d)$, it follows from the Taylor expansion that 
\begin{align*}
\dszero^2(\sigma, \sigma + \d s) = \frac{1}{3}\sum_{i=1}^d \frac{(\d s_i)^2}{\sigma_i} + o((\d s_i)^2), 
\end{align*} 
which is proportional to the Fisher information metric at $\sigma$: $(\d s)^s = \sum_{i=1}^d (\d s_i)^2/\sigma_i$ (on the probability simplex in $\R^d$). Additional connections between $\dszero$ and other metrics will be discussed in \Cref{sec:dan} and \Cref{sec:diffusion}. As $\alpha$ increases from 0 to 1, $\ds$ interpolates between different geometries.
\end{remark}

\begin{proof}[Proof of \Cref{thm:gauss}]
Let $\Sigma = \sum_{i=1}^d\sigma_iu_iu_i^\top$ and $\Lambda = \sum_{i=1}^d\lambda_i v_iv_i^\top$ be their eigendecompositions. Let $(\mu_t, \nu_t)$ be the joint Eldan's $\alpha$-scheme associated with $(\mu, \nu)$, and let $(a_t, b_t)$ and $(\Sigma_t, \Lambda_t)$ be the corresponding mean and covariance processes, respectively, i.e., $a_0 = b_0 = 0$, $\Sigma_0 = \Sigma$, and $\Lambda_0 = \Lambda$. 

When $\mu$ and $\nu$ are Gaussian, $\mu_t$ and $\nu_t$ remain Gaussian a.s. Since Gaussian measures have zero third-order centered moments, the diffusion terms in \eqref{dcov} associated with $\Sigma_t$ and $\Lambda_t$ vanish (trivial martingales). In this case, we use \eqref{keyobs} to write $\ds^2(\mu, \nu)$ as 
\begin{align}
\ds^2(\mu, \nu) = \tr(\Sigma) + \tr(\Lambda) - 2\int_0^\infty \E[\tr(\Sigma_s^{1-\alpha}\Lambda_s^{1-\alpha})] \d s.\label{keyobs321}
\end{align}
Moreover, the dynamics of $\Sigma_t$ and $\Lambda_t$ in \eqref{dcov} simplify to 
\begin{align}
\d \Sigma_t = -\Sigma_t^{2-2\alpha} \d t, \quad\quad\d \Lambda_t = -\Lambda_t^{2-2\alpha} \d t.  \label{gausstrick}
\end{align}
Solving $\Sigma_t$ and $\Lambda_t$ in \eqref{gausstrick} yields
{\small\begin{equation}\label{gausscov}
\begin{aligned}
\Sigma_t = \begin{cases}
\left( \Sigma^{2\alpha - 1} + (1 - 2\alpha)t I \right)^{\frac{1}{2\alpha - 1}}& 0\leq\alpha<\frac{1}{2}\\
e^{-t}\Sigma & \alpha = \frac{1}{2}\\
\left( \Sigma^{2\alpha-1} - (2\alpha-1)t I \right)_+^{\frac{1}{2\alpha-1}} & \frac{1}{2}<\alpha\leq 1
\end{cases},\quad \Lambda_t = \begin{cases}
\left( \Lambda^{2\alpha - 1} + (1 - 2\alpha)t I \right)^{\frac{1}{2\alpha - 1}}& 0\leq\alpha<\frac{1}{2}\\
 e^{-t}\Lambda & \alpha = \frac{1}{2}\\
\left( \Lambda^{2\alpha-1} - (2\alpha-1)t I \right)_+^{\frac{1}{2\alpha-1}} & \frac{1}{2}<\alpha\leq 1
\end{cases}, 
\end{aligned}
\end{equation}}
where $(\cdot)_+$ is understood in the sense of functional calculus.

Denote by $I_\alpha(x, y) = \<\varphi_\alpha(x), \varphi_\alpha(y)\>_{L^2(\R_+)}$ for $x, y>0$. 
Writing $\Sigma$ and $\Lambda$ using their eigendecomposition and substituting back to \eqref{keyobs321} yields 
\begin{align}
\ds^2(\mu, \nu) = \tr(\Sigma)+ \tr(\Lambda)- 2 \sum_{i=1}^d \sum_{j=1}^d I_\alpha(\sigma_i, \lambda_j)(u_i^\top v_j)^2.\label{finalf}
\end{align}
To obtain the desired result, observe that $I_\alpha(x, x) = x$ for any $x>0$ and $0\leq\alpha\leq 1$ and $\{u_i\}_{i=1}^d, \{v_j\}_{i=1}^d$ are orthonormal bases. Consequently, 
\begin{align*}
\tr(\Sigma) = \sum_{i=1}^d \sum_{j=1}^d I_\alpha(\sigma_i, \sigma_i)(u_i^\top v_j)^2, \quad \tr(\Lambda)= \sum_{i=1}^d \sum_{j=1}^d  I_\alpha(\lambda_j, \lambda_j)(u_i^\top v_j)^2. 
\end{align*} 
Substituting these back to \cref{finalf} yields 
\begin{align*}
\ds^2(\mu, \nu) &= \sum_{i=1}^d\sum_{j=1}^d \left[I_\alpha(\sigma_i, \sigma_i)(u_i^\top v_j)^2+I_\alpha(\lambda_j, \lambda_j)(u_i^\top v_j)^2-2I_\alpha(\sigma_i, \lambda_j)(u_i^\top v_j)^2\right]\\
&= \|\Phi_\alpha(\Sigma)- \Phi_\alpha(\Lambda)\|^2_\H.
\end{align*}

We next show that $\ds$ attains its minimum at $\alpha = 1/2$. Based on the formula in \eqref{finalf}, it suffices to show that for every $x, y>0$, $I_\alpha(x, y)$, as a function of $\alpha$, is maximized at $\alpha = 1/2$. A direct calculation shows that $I_\alpha(x, y)$ is homogeneous in $(x, y)$, i.e., for every $k>0$, $I_\alpha(kx, ky) = kI_\alpha(x, y)$. Therefore, it is enough to consider the situation $y = 1/x$. In this case, $I_{1/2}(x, 1/x) = 1$. For $0\leq \alpha<1/2$, by the AM--GM inequality, 
\begin{align*}
I_\alpha(x, 1/x) &= \int_0^\infty (x^{-(1-2\alpha)}+(1-2\alpha) t)^{-\frac{1-\alpha}{1-2\alpha}}(x^{1-2\alpha}+(1-2\alpha) t)^{-\frac{1-\alpha}{1-2\alpha}} \d t\\
& = \int_0^\infty \left( 1 + (x^{-(1-2\alpha)}+x^{1-2\alpha})(1 - 2\alpha)t+ (1 - 2\alpha)^2t^2 \right)^{-\frac{1-\alpha}{1-2\alpha}} \d t\\
&\leq \int_0^\infty \left( 1 + (1+1)(1 - 2\alpha)t+ (1 - 2\alpha)^2t^2 \right)^{-\frac{1-\alpha}{1-2\alpha}} \d t\\
&= I_\alpha(1, 1) = I_{1/2}(x, 1/x) = 1. 
\end{align*}
Similarly, for $1/2<\alpha\leq 1$, 
\begin{align*}
I_\alpha(x, 1/x) &= \int_0^\infty \left(x^{2\alpha-1} - (2\alpha-1)t \right)_+^{\frac{1-\alpha}{2\alpha-1}}\left(x^{1-2\alpha} - (2\alpha-1)t \right)_+^{\frac{1-\alpha}{2\alpha-1}} \d t\\
& = \int_0^{\frac{\min\{x, 1/x\}^{2\alpha -1}}{2\alpha-1}} \left(1 - (x^{2\alpha-1}+x^{1-2\alpha})(2\alpha-1)t+ (2\alpha-1)^2t^2 \right)^{\frac{1-\alpha}{2\alpha-1}} \d t\\
&\leq \int_0^{\frac{1}{2\alpha-1}} \left( 1 - (1+1)(2\alpha-1)t+ (2\alpha-1)^2t^2 \right)^{\frac{1-\alpha}{2\alpha - 1}} \d t\\
&= I_\alpha(1, 1) = I_{1/2}(x, 1/x) = 1.
\end{align*}
The proof is complete. 
\end{proof}

\Cref{thm:gauss} illustrates how $\ds$ induces different geometries for different values of $\alpha$ when restricted to the space of Gaussian measures. For general measures, Eldan's $\alpha$-distance depends on more than just covariance information, and the dependence structure is more subtle. Nevertheless, for log-concave distributions, heuristics based on covariance comparisons for describing the induced coupling remain partially valid when comparing them to a Gaussian measure. The next result formalizes this in the case $\alpha = 1/2$. 

\begin{theorem}\label{thm:caf}
Let $\mu$ be a $\kappa$-strongly log-concave measure on $\R^d$ with covariance $\Sigma$ and $\nu$ be a Gaussian measure with the same mean as $\mu$ and covariance $\Lambda\succeq I/\gamma$ for some $\gamma >0$. Then, 
\begin{align*}
\dshalf^2(\mu, \nu)\leq \tr(\Sigma) + \tr(\Lambda) - 2\sqrt{\frac{\kappa}{\gamma}}\tr(\Sigma).
\end{align*} 
\end{theorem}

\begin{remark}
Taking $\Lambda = I/\kappa$ (i.e., one can take $\gamma = \kappa$) shows that 
\begin{align*}
\dshalf^2(\mu, \nu)\leq\tr(\Lambda)-\tr(\Sigma) = \frac{d}{\kappa}-\tr(\Sigma).
\end{align*}
Since the $2$-Wasserstein distance $W_2$ (see \eqref{myW2} for the definition of $W_2$) satisfies $W^2_2(\mu, \nu) \leq \dshalf^2(\mu, \nu)$, it also follows that $W^2_2(\mu, \nu) \leq d/\kappa - \tr(\Sigma)$. This bound on the $2$-Wasserstein distance can be alternatively obtained, for instance, using Caffarelli's contraction theorem \cite{caffarelli2000monotonicity}. Indeed, by Caffarelli's contraction theorem, the optimal transport map $\nabla\psi$ that pushes $\nu$ to $\mu$ ($\psi: \R^d\to\R$ is the Brenier potential, which is convex), is $1$-Lipschitz. Using Gaussian integration by parts, 
\begin{align*}
W^2_2(\mu, \nu) = \E_{x\sim\nu}[\|x-\nabla\psi(x)\|_2^2] &= \tr(\Sigma) + \frac{d}{\kappa} - 2\E_{x\sim\nu}[\<x, \nabla\psi(x)\>]\\
&=\tr(\Sigma) + \frac{d}{\kappa} - \frac{2}{\kappa}\E_{x\sim\nu}[\tr(\nabla^2\psi(x))].
\end{align*}
Since $\psi$ is convex and $\nabla\psi$ is $1$-Lipschitz, $0\preceq\nabla^2\psi(x)\preceq I$. Consequently,  
\begin{align*}
\E_{x\sim\nu}[\tr(\nabla^2\psi(x))]\geq\E_{x\sim\nu}\left[\|\nabla^2\psi(x)\|_F^2\right]\geq\kappa\tr(\Sigma),  
\end{align*}
where the last step follows from the Gaussian Poincaré inequality. Substituting this back into the previous equation recovers the same result. This approach relies on the optimal transport map $\nabla\psi$, which is obtained by solving the Monge--Ampère equation and is often difficult to compute.
In contrast, our proof relies on a coupling that can be explicitly constructed by running a joint Eldan's $\frac{1}{2}$-scheme.
\end{remark}

\begin{proof}[Proof of Theorem \ref{thm:caf}]
Let $(\mu_t, \nu_t)$ be the joint Eldan's $\alpha$-scheme associated with $(\mu, \nu)$, and let $(a_t, b_t)$ and $(\Sigma_t, \Lambda_t)$ be the corresponding mean and covariance processes, respectively, i.e., $a_0 = b_0$, $\Sigma_0 = \Sigma$, and $\Lambda_0 = \Lambda$. Because both $\mu$ and $\nu$ are strongly log-concave, according to Remark~\ref{rem:exist}, the trace of the diffusion terms associated with $\Sigma_t$ and $\Lambda_t$ is a martingale. Therefore, by \cref{keyobs}, we have 
\begin{align}
\dshalf^2(\mu, \nu) = \E[\|a_\infty-b_\infty\|_2^2] = \tr(\Sigma) + \tr(\Lambda) - 2\int_0^\infty \E\left[\tr(\Sigma_s^{\frac{1}{2}}\Lambda_s^{\frac{1}{2}})\right] \d s.\label{ding}
\end{align}
Since $\nu$ is Gaussian, $\Lambda_s =e^{-s}\Lambda\succeq e^{-s}I/\gamma$. Since $\mu$ is $\kappa$-strongly log-concave, by \cite[Lemma 2]{eldan2014bounding} (which assumes $\kappa = 1$ but can be easily modified to any $\kappa>0$), $\|\Sigma_s\|_2\leq e^{-s}/\kappa$ a.s., which implies $\Sigma_s^{-1/2}\succeq \sqrt{\kappa}e^{s/2} I$. Putting both estimates together, we have 
\begin{align*}
\tr(\Sigma_s^{\frac{1}{2}}\Lambda_s^{\frac{1}{2}})= \tr(\Sigma_s^{-\frac{1}{2}}\Sigma_s\Lambda_s^{\frac{1}{2}})\geq \frac{e^{-\frac{s}{2}}}{\sqrt{\gamma}}\cdot\sqrt{\kappa}e^{\frac{s}{2}}\tr(\Sigma_s) = \sqrt{\frac{\kappa}{\gamma}}\tr(\Sigma_s).
\end{align*}
Substituting this into \eqref{ding} and further noting that $\E[\tr(\Sigma_s)] = e^{-s}\E[\tr(\Sigma_0)]$,  
\begin{align*}
\dshalf^2(\mu, \nu) \leq  \tr(\Sigma) + \tr(\Lambda) - 2\sqrt{\frac{\kappa}{\gamma}}\tr(\Sigma). 
\end{align*}
\end{proof}

\subsection{Extrapolating the Bures--Wasserstein distance}
 
Although the couplings induced by joint Eldan's $\alpha$-schemes exhibit dependence structures that interpolate between different geometries as $\alpha$ varies from $0$ to $1$, none of them coincides with the optimal coupling in the sense of optimal transport in $L^2$ for Gaussian measures. In particular, when $\alpha = 1/2$ (the choice under which $\ds$ is minimized for Gaussians), the Eldan $\alpha$-distance between two Gaussian measures $\mu\sim \mathcal N(a, \Sigma)$ and $\nu\sim \mathcal N(b, \Lambda)$ is given by
\begin{align*}
\dshalf(\mu, \nu) = \sqrt{\|a-b\|_2^2+\tr(\Sigma)+\tr(\Lambda)-2\tr(\Sigma^{\frac{1}{2}}\Lambda^{\frac{1}{2}})}, 
\end{align*}
which differs from their 2-Wasserstein distance in \eqref{bw}. In this section, we introduce an alternative joint SL scheme whose induced coupling, when restricted to Gaussian measures, coincides with the optimal transport coupling in $L^2$. 

We first recall the definition of the $2$-Wasserstein distance between two probability measures $\mu$ and $\nu$:
 \begin{align}
W_2(\mu, \nu) = \inf_{X\sim\mu, Y\sim\nu}\E[\|X-Y\|_2^2]^{\frac{1}{2}},\label{myW2}
\end{align} 
where the minimum is taken over all couplings between $\mu$ and $\nu$. When both $\mu\sim \mathcal N(a, \Sigma)$ and $\nu\sim \mathcal N(b, \Lambda)$ are Gaussians, $W_2(\mu, \nu)$ admits a closed-form solution:  
\begin{align}
W_2(\mu, \nu) = \sqrt{\|a-b\|_2^2 + \tr(\Sigma) + \tr(\Lambda) - 2\tr(\Sigma^{\frac{1}{2}}\Lambda\Sigma^{\frac{1}{2}})^{\frac{1}{2}}}, \label{bw}
\end{align}
which is also known as the Bures--Wasserstein distance. See \cite[Proposition 7]{givens1984class} for a detailed derivation. For general measures, the right-hand side in \eqref{bw} provides a lower bound on $W_2(\mu, \nu)$ known as the Gelbrich bound \cite{gelbrich1990formula}.

We now consider a joint SL scheme between $\mu$ and $\nu$ by choosing the matrix-valued processes $C_t$ and $D_t$ in \eqref{coupling} as  
\begin{align}
C_t = (\Sigma^{\dagger}_t)^{\frac{1}{2}}, \quad\quad D_t = \Sigma_t^{\frac{1}{2}}\left[\left(\Sigma^{\frac{1}{2}}_{t}\Lambda_{t}\Sigma^{\frac{1}{2}}_{t}\right)^{\dagger}\right]^{\frac{1}{2}}.\label{cd}
\end{align}
We call this scheme the \textit{extrapolation scheme}. The choice of $C_t$ coincides with that in Eldan's $\frac{1}{2}$-scheme. However, $D_t$ depends not only on its own covariance but also on $\mu_t$, indicating that $C_t$ and $D_t$ play asymmetric roles in this joint SL scheme. 

At first sight, it is unclear whether $\nu_t$ will eventually localize as $t\to\infty$. For instance, if $\mu_t$ converges to a point mass at some finite time $T$, then $\nu_t = \nu_{T}$ for all $t\geq T$, regardless of whether $\nu_t$ is localized or not. The next result shows that for log-concave measures, localization is guaranteed when $t\to\infty$. 

\begin{theorem}\label{thm:joint_SL_localization_exists}
If $\mu$ and $\nu$ are log-concave measures on $\R^d$, then a.s., the extrapolation scheme for $\mu$ and $\nu$ (defined by \eqref{coupling} and \eqref{cd}) exists for $t>0$ and localizes when $t\to\infty$.  
\end{theorem}
\begin{proof}
Since $\mu_t$ is an Eldan's $\frac{1}{2}$-scheme, under the log-concavity assumption on $\mu$, both existence and localization for $\mu_t$ follow from \cite[Lemma 2.4]{eldan2013thin}. Moreover, $\Sigma_t$ is invertible for all $t>0$ a.s., so that $\Sigma_t^\dagger = \Sigma_t^{-1}$. For $\nu_t$, before $\Lambda_t$ becomes degenerate, we can define another process
\begin{align}\label{eq:tr}
\d W'_t = \Lambda_t^{\frac{1}{2}}\Sigma_t^{\frac{1}{2}}(\Sigma^{-\frac{1}{2}}_{t}\Lambda^{-1}_{t}\Sigma^{-\frac{1}{2}}_{t})^{\frac{1}{2}} \d W_t.  
\end{align}
The key observation is that $W'_t$ is also a Brownian motion correlated with $W_t$, which can be shown by verifying the diffusion matrix in \eqref{eq:tr} is orthogonal:
\begin{align*}
\Lambda_t^{\frac{1}{2}}\Sigma_t^{\frac{1}{2}}(\Sigma^{-\frac{1}{2}}_{t}\Lambda^{-1}_{t}\Sigma^{-\frac{1}{2}}_{t})^{\frac{1}{2}}\left(\Lambda_t^{\frac{1}{2}}\Sigma_t^{\frac{1}{2}}(\Sigma^{-\frac{1}{2}}_{t}\Lambda^{-1}_{t}\Sigma^{-\frac{1}{2}}_{t})^{\frac{1}{2}}\right)^\top = I. 
\end{align*}
Under $W'_t$, the dynamics of $q_t$ in \eqref{coupling} can be expressed as 
\begin{align*}
\d q_t(x) = q_t(x)\< x-b_{t}, \Lambda_t^{-\frac{1}{2}}\d W'_{t}\>, 
\end{align*}
which is an Eldan's $\frac{1}{2}$-scheme driven by $W'_t$. Since $\nu$ is log-concave, a rescaling argument similar to the one in the proof of \cite[Lemma 2.4]{eldan2013thin} shows that this stopping time is infinity a.s. Hence, $\nu_t$ follows an Eldan's $\frac{1}{2}$-scheme with Brownian motion $W'_t$ for all $t\geq 0$, and the existence and localization results follow. 
\end{proof}

Consequently, the extrapolation scheme is well defined for two Gaussian measures and will localize as time goes to infinity. Hence, it will induce a coupling between them. We show next that the resulting coupling is identical to the $W_2$-optimal coupling. 

\begin{theorem}\label{thm:1}
Let $\mu\sim \mathcal N(a, \Sigma)$ and $\nu\sim \mathcal N(b, \Lambda)$ be Gaussian measures on $\R^d$. Let $(a_\infty, b_\infty)$ be the coupling of $\mu$ and $\nu$ induced by their extrapolation scheme (as defined in \eqref{coupling} and \eqref{cd}). Then, $\E[\|a_\infty-b_\infty\|_2^2] = W_2^2(\mu, \nu)$, where the latter is given by \eqref{myW2} and \eqref{bw}.  
\end{theorem}
\begin{proof}
Let $(\mu_t, \nu_t)$ be the extrapolation scheme associated with $(\mu, \nu)$, and let $(a_t, b_t)$ and $(\Sigma_t, \Lambda_t)$ be the corresponding mean and covariance processes, respectively, i.e., $a_0 = a$, $b_0 = b$, $\Sigma_0 = \Sigma$, and $\Lambda_0 = \Lambda$. Moreover, we can take $(a_\infty, b_\infty) = \lim_{t\to\infty}(a_t, b_t)$, which exists a.s.

For $t\geq 0$, applying \eqref{at} to both $a_t$ and $-b_t$ yields
\begin{align}
\d (a_t-b_t) = (\Sigma_t C_t - \Lambda_t D_t)\d W_t.\label{at-bt}
\end{align}
Then, by using It\^o's isometry and martingale convergence in a manner similar to how \eqref{keyobs} is obtained, 
\begin{equation}\label{cal0}
\begin{aligned}
\E[\|a_\infty-b_\infty\|_2^2] &= \|a-b\|_2^2 + \tr(\Sigma) + \tr(\Lambda) -2\int_0^\infty \tr(\E[(\Sigma_s C_s - \Lambda_s D_s)(\Sigma_s C_s - \Lambda_s D_s)^\top]) \d s\\
& = \|a-b\|_2^2 + \tr(\Sigma) + \tr(\Lambda) -2\int_0^\infty \E\left[\tr(\Sigma_s^{\frac{1}{2}}\Lambda_s\Sigma_s^{\frac{1}{2}})^{\frac{1}{2}}\right] \d s.
\end{aligned}
\end{equation}
Note that both $\mu_t$ and $\nu_t$ are Eldan's $1/2$-scheme driven by suitable Brownian motions (as we have shown in the proof of \Cref{thm:joint_SL_localization_exists}). As a result, $\Sigma_s = e^{-s}\Sigma$ and $\Lambda_s = e^{-s}\Lambda$. Substituting these into \eqref{cal0} and comparing with \eqref{bw} finishes the proof. 
\end{proof}

The extrapolation scheme is intrinsically locally defined and therefore not canonical; it does not readily extend to a coupling of multiple measures. Moreover, unlike the joint Eldan's $\alpha$-scheme, it does not induce a natural notion of distance. These limitations make the extrapolation scheme less convenient from a computational perspective. As a result, we will focus in the sequel on the joint Eldan's $\alpha$-scheme, specifically on the Eldan's $\alpha$-distance induced by the canonical coupling, as a quantitative measure of its coupling structure.


\section{Eldan's $\alpha$-distance}\label{sec:dist}

In this section, we study the theoretical properties of Eldan's $\alpha$-distance $\ds$ induced by joint Eldan $\alpha$-schemes. Rather than restricting to Gaussian measures, we consider general measures for which this distance is well defined. We first discuss the basic properties of $\ds$ under translations, scalings, and orthogonal transformations, and introduce a weighted version based on the full localization trajectory. We then consider the special case $\alpha = 0$ and show that $\dszero$ is topologically equivalent to $W_2$ for measures with uniformly bounded support. Finally, we connect weighted $\dszero$ and other notions in optimal transport, including the linearized optimal transport in Wiener space and the score-matching objectives used in training diffusion models.

\subsection{Basic properties and weighted Eldan's $\alpha$-distance}\label{sec:jsl00}

Eldan's $\alpha$-distance satisfies a similar translation identity and invariance property under orthogonal transformations as $W_2$. When $\alpha = 1/2$, an additional scaling property holds. To establish these properties, we first state a lemma concerning the affine transformation of Eldan's $\alpha$-scheme.

\begin{lemma}\label{lemma:affine}
Let $\mu\in\mathcal P_\alpha(\R^d)$ and $g(x) = Ax + c: \R^d\to\R^d$, where $A\in\R^{d\times d}$ is an invertible matrix and $c\in\R^d$. Denote by $\mu_t$ an Eldan's $\alpha$-scheme associated with $\mu$ with Brownian motion $W_t$. If we define $\mu^g = g\#\mu$ and $\mu^g_t = g\#\mu_t$ (\# denotes the pushforward measure), then $\mu^g_t$ is an SL scheme associated with $\mu^g$ with Brownian motion $W_t$ and control process $C_t = A^{-\top}[(A^{-1}\Sigma^g_tA^{-\top})^\dagger]^\alpha$, where $\Sigma_t^g$ is the covariance of $\mu_t^g$.  
\end{lemma}

\begin{proof}
Denote the density of $\mu_t$ with respect to $\mu$ as $p_t$. By the definition of pushforward measures, $\mu_t^g\ll\mu^g$ with density $p_t^g = p_t\circ g^{-1}$, i.e., $p_t^g(x) = p_t(A^{-1}(x-c))$. Moreover, the mean and covariance of $\mu_t^g$, denoted by $a_t^g$ and $\Sigma_t^g$, are related to the mean and covariance of $\mu_t$, denoted by $a_t$ and $\Sigma_t$, as follows:
\begin{align}
a^g_t = Aa_t + c, \quad\quad\Sigma_t^g = A\Sigma_tA^\top.\label{aff}
\end{align}
The desired result follows by noting that for every $x\in\supp(\mu^g)$, 
\begin{align*}
\d p_t^g(x) &= \d p_t(A^{-1}(x-c)) \stackrel{\eqref{sl}}{=} p_t(A^{-1}(x-c))\<A^{-1}(x-c)-a_t, (\Sigma_t^\dagger)^\alpha\d W_t\>\\
& \stackrel{\eqref{aff}}{=} p_t^g(x)\<x-a^g_t, A^{-\top}[(A^{-1}\Sigma^g_tA^{-\top})^\dagger]^\alpha\d W_t\>.
\end{align*}
\end{proof}

Lemma~\ref{lemma:affine} implies the following properties of Eldan's $\alpha$-distance. 

\begin{theorem}\label{prop:shift}
Let $\mu, \nu\in\mathcal P_\alpha(\R^d)$. 
\begin{enumerate}
\item [(i)] (Translation identity) If $\mu_{\ct} = (I-\E_\mu[x])\#\mu$, $\nu_\ct = (I-\E_\nu[x])\#\nu$, and $c = \E_\mu[x] - \E_\nu[x]$, then $\ds^2(\mu, \nu) = \ds^2(\mu_\ct, \nu_\ct) + \|c\|_2^2$. 
\item [(ii)] (Invariance under orthogonal transformation) For any orthogonal transformation matrix $U\in\R^{d\times d}$, if $\mu^U = U\#\mu$ and $\nu^U = U\#\nu$, then
$\ds(\mu, \nu) = \ds(\mu^U, \nu^U)$.
\item [(iii)] (Homogeneity for $\alpha = 1/2$) Suppose $\alpha = 1/2$. For any $\gamma\in \R$ and $\gamma\neq 0$, if $\mu^\gamma = (\gamma I)\#\mu$ and $\nu^\gamma = (\gamma I)\#\nu$, then $\ds(\mu^\gamma, \nu^\gamma) = |\gamma| \ds(\mu, \nu)$. 
\end{enumerate}
\end{theorem}
\begin{proof}
We begin by introducing some notations. Let $(\mu_t, \nu_t)$ be a joint Eldan's $\alpha$-scheme associated with $\mu$ and $\nu$ driven by Brownian motion $W_t$. The corresponding density process is denoted by $(p_t, q_t)$, the mean process by $(a_t, b_t)$, and the covariance process by $(\Sigma_t, \Lambda_t)$. By definition, $(a_\infty, b_\infty) = \lim_{t\to\infty}(a_t, b_t)$ exists a.s. and provides a coupling defining $\ds(\mu, \nu)$. 

For (i), note that by the bias-variance decomposition, 
\begin{align*}
\ds^2(\mu, \nu) = \E[\|a_\infty-b_\infty\|_2^2] = \E[\|(a_\infty-\E_\mu[x])-(b_\infty-\E_\nu[x])\|_2^2] + \|c\|_2^2.
\end{align*}
Since $a_t-\E_\mu[x] = \E_{g\#\mu_t}[x]$ and $b_t-\E_\nu[x] = \E_{h\#\nu_t}[x]$, where $g = I-\E_\mu[x]$ and $h = I-\E_\nu[x]$, it suffices to show that $(g\#\mu_t, h\#\nu_t)$ is a joint Eldan's $\alpha$-scheme associated with $(\mu_\ct, \nu_\ct)$. This follows immediately from \Cref{lemma:affine}. 

For (ii), let $\mu_t^U = U\#\mu_t$ and $\nu_t^U = U\#\nu_t$, with their mean and covariance processes denoted by $(a^U_t, b^U_t)$ and $(\Sigma^U_t, \Lambda_t^U)$, respectively. Applying Lemma~\ref{lemma:affine} with $A = U$ and $c = 0$ yields that $(\mu_t^U, \nu_t^U)$ is a joint SL scheme associated with $(\mu^U, \nu^U)$ with Brownian motion $W_t$ and control processes $[(\Sigma^U_t)^\dagger]^\alpha U$ and $[(\Lambda^U_t)^\dagger]^\alpha U$. Alternatively, $(\mu_t^U, \nu_t^U)$ is a joint Eldan's $\alpha$-scheme with Brownian motion $W_t' = UW_t$. Consequently, 
\begin{align}
\ds^2(\mu, \nu) = \E[\|a_\infty-b_\infty\|^2_2] = \E[\|Ua_\infty-Ub_\infty\|^2_2]= \ds^2(\mu^U, \nu^U).\label{similar}
\end{align}

The proof of (iii) is similar to (ii). Let $\mu_t^\gamma = (\gamma I)\#\mu_t$ and $\nu_t^\gamma = (\gamma I)\#\nu_t$, with their mean and covariance processes denoted by $(a^\gamma_t, b^\gamma_t)$ and $(\Sigma^\gamma_t, \Lambda_t^\gamma)$, respectively. Applying Lemma~\ref{lemma:affine} with $A = \gamma I$ and $c = 0$ yields that $(\mu_t^\gamma, \nu_t^\gamma)$ is a joint SL scheme associated with $(\mu^\gamma, \nu^\gamma)$ with Brownian motion $W_t$ and control processes $\gamma^{2\alpha -1}[(\Sigma^\gamma_t)^\dagger]^\alpha$ and $\gamma^{2\alpha -1}[(\Lambda^\gamma_t)^\dagger]^\alpha$, which is a joint Eldan's $\alpha$-scheme if $\alpha = 1/2$. The rest of the proof is similar to \eqref{similar}. 
\end{proof}

\begin{remark}
An optimal coupling under $W_2$ also satisfies the homogeneity property, i.e., if $(X, Y)$ is a $W_2$-optimal coupling between $\mu$ and $\nu$, then $(\gamma X, \gamma Y)$ remains optimal under $W_2$ for $\mu^\gamma$ and $\nu^\gamma$. In this sense, Eldan's $\frac{1}{2}$-distance resembles the $W_2$-distance. 
\end{remark}

The definition of $\ds$ accounts only for the final coupling along the joint Eldan's $\alpha$-scheme. In some situations, it is useful to consider the coupling structure along the entire time evolution. More generally, we introduce the following extension of $\ds$ to incorporate the full dynamics arising from the joint Eldan's $\alpha$-scheme.

\begin{definition}[Weighted Eldan's $\alpha$-distance]\label{def:wds}
Let $w$ be a probability measure on $[0, \infty)$ with unbounded support. For $\mu, \nu\in\mathcal P_\alpha(\R^d)$, the $w$-weighted Eldan's $\alpha$-distance between $\mu$ and $\nu$ is defined as 
\begin{align}
\ds^w(\mu, \nu) = \left(\int_0^\infty \E[\|a_t-b_t\|_2^2] \ w (\d t)\right)^{\frac{1}{2}}, \label{wSL}
\end{align}
where $(a_t, b_t)$ is the mean process of the joint Eldan's $\alpha$-scheme associated with $\mu$ and $\nu$. 
\end{definition}

The Eldan's $\alpha$-distance $\ds(\mu, \nu)$ can be formally viewed as a special case of $\ds^w(\mu, \nu)$ with $w = \delta_{\infty}$. The next Lemma shows that \eqref{wSL} is a well-defined distance function on $\mathcal P_\alpha(\R^d)$. 

\begin{lemma}
The $w$-weighted Eldan's $\alpha$-distance in \eqref{wSL} is a well-defined distance function on $\mathcal P_\alpha(\R^d)$. 
\end{lemma}

\begin{proof}
By definition, $\ds^w(\mu, \nu)$ is nonnegative, symmetric, and satisfies the triangle inequality. It remains to verify that for $\mu, \nu\in\mathcal P_\alpha(\R^d)$, $\ds^w(\mu, \nu) = 0$ if and only if $\mu = \nu$. The ``if'' direction is trivial by definition. For the ``only if'' direction, note that if $\mu\neq \nu$, then $\E[\|a_\infty-b_\infty\|_2^2]\geq W_2(\mu, \nu)>0$, i.e., there exists $\e>0$ such that $\P(\|a_\infty-b_\infty\|_2>\e)>\e$. Since $(a_t, b_t)\to (a_\infty, b_\infty)$ a.s., they also converge in probability: $\P(\|a_t-a_\infty\|_2>\e/3)\to 0$ and $\P(\|b_t-b_\infty\|_2>\e/3)\to 0$ as $t\to\infty$. Consequently, there exists a large $T$ so that for all $t\geq T$, $\P(\|a_t-b_t\|_2\geq\e/3)\geq\e/3$. The desired result follows from the fact that $w$ has unbounded support:
\begin{align*}
[\ds^w(\mu, \nu)]^2\geq \int_T^\infty \E[\|a_t-b_t\|_2^2] \ w (\d t)\geq\frac{\e^3}{27}w(\{t: t\geq T\})>0. 
\end{align*} 
\end{proof}

\begin{remark}
When $\mu,\nu$ have bounded support, $a_t-b_t$ is a martingale, so $\|a_t-b_t\|_2^2$ is a submartingale, and thus $\E[\|a_t-b_t\|_2^2]$ is nondecreasing in $t$. Consequently, 
\begin{align*}
\sup_{w: |\supp(w)| = \infty}\ds^w(\mu, \nu) = \ds (\mu, \nu).
\end{align*}
Hence, both $W_2(\mu, \nu)$ and $\ds^w (\mu, \nu)$ are bounded by $\ds (\mu, \nu)$.  
\end{remark}

\subsection{Topological equivalence between $\dszero$ and $W_2$}\label{sec:top}

The analysis of $\ds$ becomes particularly tractable when $\alpha = 0$. Technically, setting $\alpha = 0$ yields a constant control process which simplifies the theoretical analysis. Moreover, $\mathcal P_0(\R^d)$ is sufficiently rich and contains all probability measures on $\R^d$ with finite first moments. The next theorem shows that $\dszero$ is topologically equivalent to $W_2$ for probability measures supported on a common compact set in $\R^d$.  
\begin{theorem}\label{thm:metric-equiv}
The topologies induced by $\dszero$ and $W_2$ are equivalent for probability measures supported on a common compact set in $\R^d$. 
\end{theorem}

\begin{proof}
Let $\nu$ and $\{\nu^{(n)}\}_n$ be probability measures with uniformly bounded support, i.e., $\supp(\nu)\cup(\cup_{n=1}^\infty\supp(\nu^{(n)}))\subset \B_R(\R^d)$ for $R>0$. It suffices to show that $W_2(\nu, \nu^{(n)})\to 0$ as $n\to \infty$ if and only if $\dszero(\nu, \nu^{(n)})\to 0$ as $n\to\infty$. The ``if'' direction is trivial since $W_2$ is bounded by $\dszero$. It remains to establish the ``only if'' direction. 

Let $(a_t, b_t)$ and $(\t_t, \t_t')$ denote the mean process and observation process of the joint Eldan's $0$-scheme associated with $\nu$ and $\nu^{(n)}$, respectively. In this case, one can use \eqref{pdensity} to explicitly write $a_t$ and $b_t$ as a function of $t$ and $(\t_t, \t_t')$:
\begin{align}\label{atbt}
&a_t = a_t(\t_t) = \frac{\int_{\R^d} xh_t(x; \t_t) \nu(\d x)}{\int_{\R^d}h_t(z; \t_t) \nu(\d z)}, &b_t = b_t(\t'_t) =  \frac{\int_{\R^d} xh_t(x; \t'_t) \nu^{(n)}(\d x)}{\int_{\R^d}h_t(z; \t'_t) \nu^{(n)}(\d z)},
\end{align}
where $h_t(x; \t) = \exp\{\<x, \t\> - \frac{t}{2}\|x\|_2^2\}$. 
According to \eqref{finite_system}, one can write $\t_t$ and $\t_t'$ as 
\begin{align}
\t_t = \int_0^t a_s \d s + W_t, \quad\quad \t_t' = \int_0^t b_s \d s + W_t, \label{use1}
\end{align}
where $W_t$ is the Brownian motion used in defining the joint Eldan's 0-scheme associated with $\nu$ and $\nu^{(n)}$.

By the localization property, we denote the a.s.-limit of $(a_t, b_t)$ when $t\to\infty$ as $(a_\infty, b_\infty)$.
The path-continuity of $(a_t, b_t)$ combined with \eqref{use1} implies $(a_\infty, b_\infty) = \lim_{t\to\infty}(\bar{\t}_t, \bar{\t}'_t)$ a.s., where $(\bar{\t}_t, \bar{\t}'_t) := (\t_t, \t'_t)/t$. Combining this with Fatou's lemma gives  
\begin{align*}
\E[\|a_\infty - b_\infty - (\bar{\t}_t- \bar{\t}'_t)\|^2_2] &\leq 2(\E[\|a_\infty -\bar{\t}_t\|_2^2] + \E[\|b_\infty -\bar{\t}'_t\|_2^2]) \\
&\leq \lim_{s \to \infty}  2(\E[\|\bar{\t}_s -\bar{\t}_t\|_2^2] + \E[\|\bar{\t}'_s -\bar{\t}'_t\|_2^2]). 
\end{align*}
By the law-equivalent form of $\t_t$ in \eqref{obs}, we have $\t_t \equiv tX + W_t$ for $X\sim\nu$ independent of $W_t$, from which it follows that  
\[
\lim_{s \to \infty} \E[\|\bar{\t}_s -\bar{\t}_t\|_2^2] =\lim_{s\to\infty} \E\left[\left\|\frac{W_t}{t}-\frac{W_s}{s}\right\|_2^2\right]= \lim_{s\to\infty}dt\left(\frac{1}{t}-\frac{1}{s}\right)^2 + \frac{d(s-t)}{s^2} = \frac{d}{t}.
\]
Similarly, $\lim_{s \to \infty}\E[\|\bar{\t}'_s -\bar{\t}'_t\|_2^2]=d/t$. Consequently, we obtain
\[
\E[\|a_\infty - b_\infty - (\bar{\t}_t- \bar{\t}'_t)\|^2_2] \leq \frac{4d}{t}.
\]
Since the convergence rate here is independent of $n$, given $\e>0$, there exists $T(\e)>0$ (e.g., $T(\e) = 4d/\e$) such that for all $n\geq 1$, 
\begin{align}
\E[\|a_\infty - b_\infty - (\bar{\t}_t- \bar{\t}'_t)\|^2_2]\leq \e\quad\text{for all $t\geq T$}. \label{truncation}
\end{align}
Therefore, to bound $\E[\|a_\infty - b_\infty\|_2^2]$, it suffices to bound $\E[\|\bar{\t}_T- \bar{\t}'_T\|_2^2]$ at finite time $T$, which can be achieved by a stability analysis of the SDEs in \eqref{use1}. 

To this end, we first apply the Cauchy--Schwarz inequality to obtain the following estimate on $\|\t_t - \t'_t\|^2_2$:
\begin{align}
\|\t_t - \t'_t\|^2_2 & \stackrel{\eqref{atbt}+ \eqref{use1}}{=} \left\|\int_0^t a_s(\t_s) - b_s(\t_s') \d s\right\|^2_2\notag\\
&\leq 2\left\|\int_0^t a_s(\t_s) - b_s(\t_s) \d s\right\|^2_2 + 2\left\|\int_0^t b_s(\t_s) - b_s(\t'_s) \d s\right\|^2_2\notag\\
&\leq 2t\int_0^t \underbrace{\|a_s(\t_s) - b_s(\t_s)\|^2_2}_{(\mathrm{i})} + \underbrace{\|b_s(\t_s) - b_s(\t'_s)\|^2_2}_{(\mathrm{ii})} \d s. \label{twobdd}
\end{align}
Terms (i) and (ii) can be further bounded as follows. 
To bound (i), we have 
\begin{align*}
&\mathrm{(i)}\stackrel{\eqref{atbt}}{=}\ \left\|\frac{\int_{\R^d}xh_s(x; \t_s) \nu(\d x)}{\int_{\R^d}h_s(z; \t_s) \nu(\d z)} - \frac{\int_{\R^d} xh_s(x; \t_s) \nu^{(n)}(\d x)}{\int_{\R^d} h_s(z; \t_s) \nu^{(n)}(\d z)}\right\|^2_2\\
\leq&\ 2\underbrace{\left\|\frac{\int_{\R^d}xh_s(x; \t_s) \nu(\d x)}{\int_{\R^d}h_s(z; \t_s) \nu^{(n)}(\d z)} - \frac{\int_{\R^d} xh_s(x; \t_s) \nu^{(n)}(\d x)}{\int_{\R^d} h_s(z; \t_s) \nu^{(n)}(\d z)}\right\|^2_2}_{(\mathrm{i.a})}+2\underbrace{\left\|\frac{\int_{\R^d}xh_s(x; \t_s) \nu(\d x)}{\int_{\R^d}h_s(z; \t_s) \nu(\d z)} - \frac{\int_{\R^d} xh_s(x; \t_s) \nu(\d x)}{\int_{\R^d} h_s(z; \t_s) \nu^{(n)}(\d z)}\right\|^2_2}_{(\mathrm{i.b})}.
\end{align*}
Since $\Omega\coloneqq\B_R(\R^d)$ is compact, for every fixed $\t$, denote the maximum of the Lipschitz constants of $\frac{h_s(x; \t)}{\int_{\R^d}h_s(z; \t)\nu^{(n)}(\d z)}$ and $\frac{x_jh_s(x; \t)}{\int_{\R^d}h_s(z; \t)\nu^{(n)}(\d z)}$ ($j\leq d$) on $\Omega$ by $\lip_\Omega(\t, \nu^{(n)})$. Term (i.a) can be bounded using the Kantorovich--Rubinstein duality and Jensen's inequality: 
\begin{align*}
(\mathrm{i.a}) \leq d\lip^2_\Omega(\t_s, \nu^{(n)})W^2_1(\nu, \nu^{(n)})\leq d\lip^2_\Omega(\t_s, \nu^{(n)})W^2_2(\nu, \nu^{(n)}).
\end{align*}
Term (i.b) can be bounded by a similar reasoning: 
\begin{align*}
(\mathrm{i.b}) &= \left|\frac{\int_{\R^d} h_s(x; \t_s) \nu^{(n)} (\d x)}{\int_{\R^d} h_s(z; \t_s) \nu^{(n)} (\d z)}-\frac{\int_{\R^d} h_s(x; \t_s) \nu (\d x)}{\int_{\R^d} h_s(z; \t_s) \nu^{(n)} (\d z)}\right|^2\left\|\frac{\int_{\R^d}xh_s(x; \t_s) \nu (\d x)}{\int_{\R^d}h_s(z; \t_s) \nu (\d z)}\right\|^2_2\\
&\leq \lip^2_\Omega(\t_s, \nu^{(n)})W^2_2(\nu, \nu^{(n)})\|a_s\|^2_2\\
&\leq \lip^2_\Omega(\t_s, \nu^{(n)})W^2_2(\nu, \nu^{(n)})R^2.
\end{align*}
Combining the estimates on (i.a) and (i.b) yields the following upper bound on (i): 
\begin{align}
(\mathrm{i})\leq 2(d + R^2)\lip^2_\Omega(\t_s, \nu^{(n)})W^2_2(\nu, \nu^{(n)})\leq c_0e^{c_0(1+\|\t_s\|_2+s)}W^2_2(\nu, \nu^{(n)}),  \label{bdd1}
\end{align}
where the last step follows from the fact that $\sup_{n}\lip^2_\Omega(\t_s, \nu^{(n)})\leq \frac{c_0}{2(d + R^2)}e^{c_0(1+\|\t_s\|_2+s)}$ for some $c_0>0$ depending only on $\Omega$ and $d$ (i.e., the denominator is not absorbed into $c_0$ for convenience). 

For (ii), we compute the Jacobian of $b_s$ with respect to $\t$ as 
\begin{align*}
D_\t b_s(\t)&\stackrel{\eqref{atbt}}{=}D_\t \left(\frac{\int_{\R^d} xh_s(x; \t) \nu^{(n)}(\d x)}{\int_{\R^d} h_s(z; \t) \nu^{(n)}(\d z)}\right)\\
& =  \frac{\int_{\R^d}x\otimes x h_s(x; \t) \nu^{(n)}(\d x)}{\int_{\R^d} h_s(z; \t) \nu^{(n)} (\d z)} -       \left(\frac{\int_{\R^d}x h_s(x; \t) \nu^{(n)}(\d x)}{\int_{\R^d} h_s(z; \t) \nu^{(n)} (\d z)}\right)^{\otimes 2},
\end{align*}
which is the covariance of $\xi$ for $\xi(\d x)\propto h_s(x; \t) \nu^{(n)}(\d x)$; see also the moment-generating property in \eqref{mgf}. Since $\supp(\nu^{(n)})\subset\Omega$ for all $n$, $\sup_{\t}\|D_\t b_s(\t)\|_2\leq R^2$. Consequently,  
\begin{align}
\|b_s(\t_s) - b_s(\t'_s)\|_2^2\leq R^4\|\t_s-\t'_s\|_2^2. \label{bdd2}
\end{align}

Substituting \eqref{bdd1}-\eqref{bdd2} into \eqref{twobdd} and taking expectation both sides,   
\begin{align*}
\E[\|\t_t - \t'_t\|_2^2] &\leq 2c_0TW^2_2(\nu, \nu^{(n)})\int_0^T \E[e^{c_0(1+\|\t_s\|_2+s)}] \d s + 2R^4T\int_0^t \E[\|\t_s - \t'_s\|_2^2] \d s\\
&\leq c_1W^2_2(\nu, \nu^{(n)})+ c_1\int_0^t \E[\|\t_s- \t'_s\|_2^2] \d s,
\end{align*}
where the second inequality follows from \eqref{obs} (which implies $\|\t_s\|_2\leq tR+\|W_t\|_2$) and $c_1(c_0, T, R)>0$ depends on $c_0$, $T$, and $R$. Applying Gr\"{o}nwall's inequality, there exists a constant $c_2(c_1)>0$ (independent of $n$) such that 
\begin{align*}
\E[\|\bar{\t}_T- \bar{\t}'_T\|_2^2]\leq\frac{1}{T^2}\E[\|\t_T - \t'_T\|_2^2]\leq \frac{c_2W^2_2(\nu, \nu^{(n)})e^{c_2T}}{T^2}.
\end{align*}

Since $W_2(\nu, \nu^{(n)})\to 0$ as $n\to\infty$, there exists some $N>0$ such that for all $n\geq N$, $\E[\|\bar{\t}_T- \bar{\t}'_T\|_2^2]\leq \e$. This combined with \eqref{truncation} yields $\dszero(\nu, \nu^{(n)})=\E[\|a_\infty - b_\infty\|_2^2]\leq 4\e$ for all $n\geq N$. The proof is concluded by noting that $\e$ is arbitrary. 
\end{proof}

\subsection{Linearized optimal transport in Wiener space}\label{sec:dan}

Eldan's $0$-scheme is also closely related to F\"ollmer processes \cite{follmer1988random}, which exhibit certain entropy-minimization properties when transporting Wiener measure to other measures defined on Wiener space. Following this observation, we connect $\dszero^w$ (under an appropriate choice of $w$) to an approximate form of the linearized optimal transport distance in Wiener space. 

Let $B_t$ be a Brownian motion in $\R^d$ and $\gamma$ be the standard Gaussian measure on $\R^d$ (i.e., $B_1\sim\gamma$). For any $\mu\ll\gamma$ with $\frac{\d\mu}{\d\gamma} = f$, the F\"ollmer drift $u_t(x)$ associated with $\mu$ is defined as the solution to the following optimal control problem: 
\begin{align}
\min_{u}\frac{1}{2}\int_0^1\E[\|u_t(X_t)\|_2^2] \d t\quad\quad\text{where $\d X_t = u_t(X_t) \d t + \d B_t$ satisfies $X_1\sim\mu$}.\label{follmer}
\end{align}
Under mild regularity conditions on $\mu$, the F\"ollmer drift is unique and admits an explicit form: $u_t(x) = \nabla \log P_{1-t}f(x)$, where $P_t$ is the heat semigroup with $P_tf(x) = \E[f(x+B_t)]$ \cite{lehec2013representation}. The process $X_t$ with the F\"ollmer drift is called the F\"ollmer process associated with $\mu$. As such, we may consider the embedding $\mathcal E$ that maps $\mu$ to the path measure $\mathcal E(\mu)$ of its F\"ollmer process. Under this interpretation, Wiener measure $\mathcal E(\gamma)$ serves as a canonical reference measure (as the corresponding F\"ollmer drift vanishes). The inverse of $\mathcal E$ is called the Brownian transport map, which was recently proposed as a unified approach to studying functional inequalities \cite{mikulincer2024brownian}.  

The objective in \eqref{follmer} is the transport cost associated with the Cameron--Martin norm \cite{feyel2004monge}. Consequently, for $\mu, \nu\ll\gamma$, if we transport Wiener measure to $\mathcal E(\mu)$ and $\mathcal E(\nu)$ using their respective F\"ollmer drifts, the $L^2$ difference between these drifts defines the linearized optimal transport distance between $\mathcal E(\mu)$ and $\mathcal E(\nu)$ in the formal tangent space at Wiener measure. We denote this distance by 
\begin{align}
\|\mathcal E(\mu) - \mathcal E(\nu)\|_{\mathbb L^2} = \sqrt{\frac{1}{2}\int_0^1\E[\|u_t(X_t)-v_t(Y_t)\|_2^2] \d t}, \label{jiagu}
\end{align}
where $X_t, Y_t$ are the F\"ollmer processes associated with $\mu$ and $\nu$ jointly defined via the same driving Brownian motion $B_t$, and $u_t, v_t$ are the corresponding F\"ollmer drifts. When $\nu = \gamma$, by Girsanov's theorem, this transport cost is equal to the KL divergence between the path measure of $\{X_t\}_{0\leq t\leq 1}$ and the Wiener measure on $C([0, 1], \R^d)$:
\begin{align}
\|\mathcal E(\mu) - \mathcal E(\gamma)\|^2_{\mathbb L^2}=\frac{1}{2}\int_0^1\E[\|u_t(X_t)\|_2^2] \d t = \KL(\mathcal E(\mu)\,\|\, \mathcal E(\gamma)) = \KL(\mu\,\|\,\gamma),\label{girsa}
\end{align}
where the second equality is known and follows from the fact that $\frac{\d{\mathcal E(\mu)}}{\d{\mathcal E(\gamma)}}(\omega)$ depends only on the terminal value of $\omega$ at $t=1$. 

A useful connection between Eldan's 0-scheme and the F\"ollmer process is as follows. 

\begin{lemma}[\cite{klartag2023spectral}]\label{lemma:foll}
Let $\t_t$ be the observation process of Eldan's $0$-scheme associated with $\mu$ as defined in \eqref{obs1}, and $X_t$ be the F\"ollmer process associated with $\mu$. Then, $\{X_t\}_{0\leq t\leq 1} = \{(1-t)\t_{\zeta(t)}\}_{0\leq t\leq 1}$, where $\zeta(t) = \frac{t}{1-t}$ for $t\in [0, 1]$ and equality is in law.  
\end{lemma}
Lemma~\ref{lemma:foll} implies the following representation of $\|\mathcal E(\mu)-\mathcal E(\nu)\|_{\mathbb L^2}$ using the mean process of the joint Eldan's 0-scheme associated with $\mu$ and $\nu$.  
\begin{theorem}\label{thm:kl}
Let $\gamma$ be the standard Gaussian measure on $\R^d$ and assume that $\mu, \nu\ll\gamma$ have unique F\"ollmer processes. Let $(a_t, b_t)$ be the mean process of the joint Eldan's $0$-scheme associated with $\mu$ and $\nu$. The linearized optimal transport distance between $\mathcal E(\mu)$ and $\mathcal E(\nu)$ can be represented as follows:
\begin{align}
\|\mathcal E(\mu)-\mathcal E(\nu)\|^2_{\mathbb L^2} = \frac{1}{2}\int_0^\infty\E\left[\left\|(1+t)(a_t-b_t) - \int_0^t (a_s-b_s) \d s\right\|_2^2\right] w(\d t), \label{kl_div}
\end{align}
where the weight measure $w$ is defined as $w(\d t) = \frac{1}{(1+t)^2}\mathbb I_{\{t\geq 0\}}\d t$. 
\end{theorem}
\begin{remark}
Formula \eqref{kl_div} does not exactly fit into the definition of $\dszero^w$. However, if we replace $\int_0^t (a_s-b_s) \d s$ by $t(a_t-b_t)$, which serves as an effective approximation for large $t$, then \eqref{kl_div} becomes $[\dszero^w]^2/2$ with the same $w$ in Theorem~\ref{thm:kl}. This establishes a connection between $\dszero^w$ and an approximate version of the linearized optimal transport distance in Wiener space. Unlike the latter, however, the definition of $\dszero^w$ requires no assumptions on $\mu$ and $\nu$ beyond finite second moments.
\end{remark}
\begin{proof}[Proof of Theorem~\ref{thm:kl}]
Denote $X_t$ and $Y_t$ as the F\"ollmer processes associated with $\mu$ and $\nu$, respectively. Let $(\t_t, \t_t')$ be the observation process of a joint Eldan's $0$-scheme associated with $\mu$ and $\nu$ with Brownian motion $W_t$. By Lemma~\ref{lemma:foll} and the definition of $(\t_t, \t_t')$, $X_t$ and $Y_t$ satisfy the following SDEs:
\begin{align*}
\d X_t &= \d ((1-t)\t_{\zeta(t)}) = \left(\frac{a_{\zeta(t)}}{1-t}-\t_{\zeta(t)}\right) \d t + (1-t)\d W_{\zeta(t)}= \left(\frac{a_{\zeta(t)}}{1-t}-\t_{\zeta(t)}\right) \d t + \d B_t,\\
\d Y_t &= \left(\frac{b_{\zeta(t)}}{1-t}-\t'_{\zeta(t)}\right) \d t + \d B_t,  
\end{align*}
where $\d B_t = (1-t)\d W_{\zeta(t)}$ is a Brownian motion (which can be checked using L\'evy's characterization of Brownian motion) shared by both $X_t$ and $Y_t$. Under the uniqueness assumption on the solution to the F\"ollmer process, we have 
\begin{align*}
&\|\mathcal E(\mu)-\mathcal E(\nu)\|^2_{\mathbb L^2}\stackrel{\eqref{jiagu}}{=}\frac{1}{2}\int_0^1\E\left[\left\|\frac{a_{\zeta(t)}}{1-t}-\t_{\zeta(t)}-\frac{b_{\zeta(t)}}{1-t}+\t'_{\zeta(t)}\right\|_2^2\right] \d t\\
=&\ \frac{1}{2}\int_0^\infty\E\left[\left\|(1+t)(a_t-b_t) - \int_0^t (a_s-b_s) \d s\right\|_2^2\right] w(\d t)\quad\quad\text{($t \gets \zeta(t)$ and definition of $w$)},
\end{align*}
proving the desired result. 
\end{proof}

\subsection{Score-matching objectives in diffusion models}\label{sec:diffusion}

We next establish a connection between $\dszero^w$ and the score-matching objectives used in training diffusion models \cite{song2020score}. Diffusion models are a type of generative model that estimate a target distribution through two stochastic processes: a forward process transforming data into noise and a backward process converting a noisy sample into data that mimics the target distribution. These two processes are time reversals of each other, with the backward process typically approximated by a parametric class (e.g., neural networks) and serving as the engine for new data generation. See \cite{tang2024score} for a mathematical overview of the related topics.

To illustrate a connection between score-matching and $\dszero^w$, let $B_t$ be a Brownian motion in $\R^d$ and consider a diffusion model with an OU forward process $X_t$:
\begin{align}
\d X_t = -X_t \d t + \sqrt{2}\d B_t, \quad\quad X_0\sim\mu.\label{forward}
\end{align}
The initial distribution $\mu$ is the target of learning and is assumed to have bounded second moment. 

Fix $T>0$ and consider the time-reversal of $X_t$ over $[0, T]$ defined by $Y_t = X_{T-t}$ for $0\leq t\leq T$, which is a time-inhomogeneous Markov process with its law described by the following SDE \cite{anderson1982reverse}:
\begin{align}
\d Y_t = (Y_t + 2\nabla\log p_{T-t}(Y_t))\d t + \sqrt{2}\d B_t,\quad\quad Y_0\sim X_T,\label{backward}
\end{align}
where $p_{T-t}$ denotes the density of $X_{T-t}$. Note that $p_{T-t}$ depends on $\mu$. Here we have slightly abused the notation $B_t$ as they have different sample paths in \eqref{forward} and \eqref{backward}. The main idea of diffusion models is that one can generate an exact sample of $\mu$ by running the backward process \eqref{backward} up to time $T$, which is, however, a nontrivial task. Fortunately, due to the mixing property of the OU process, the starting point $Y_0=X_T$ is approximately normally distributed when $T$ is moderately large, regardless of $\mu$. The remaining challenge, therefore, lies in simulating the drift in \eqref{backward}. In particular, the term $\nabla\log p_t(x)$, known as Stein's score function of $X_t$, depends on $\mu$ and thus is unavailable in practice. 

A common approach to addressing this issue is to approximate $\nabla \log p_{t}(x)$ using the score-matching mechanism, which can be formulated as the following problem:
\begin{align}
\min_{\tt\in\Theta}\int_0^T \lambda(t)\E_{x\sim X_t}\left[\|\nabla\log p_t(x) - s_t(x; \tt)\|_2^2\right] \d t, \label{poploss0}
\end{align}
where $\lambda(t): [0, T]\to \R_+$ is some positive weight function and $\{s_t(x; \tt): \R^d\to\R^d \mid \tt\in\Theta\subseteq\R^{r}\}$ represents the parametric class chosen for making approximations. This minimization problem can be equivalently converted to a computationally feasible form via integration by parts \cite{hyvarinen2005estimation}:
\begin{align}
\min_{\tt\in\Theta}\int_0^T \lambda(t)\E_{(x_0, x)\sim (X_0, X_t)}\left[\|\nabla\log p_t(x| x_0) - s_t(x; \tt)\|_2^2\right] \d t,\label{poploss1}
\end{align}
where $p_t(x|x_0)$ is the transition kernel from $X_0$ to $X_t$ and is known explicitly. In particular, problem~\eqref{poploss1} can be numerically solved using empirical minimization based on the samples of $X_0$. For theoretical analysis, however, we directly work with \eqref{poploss0}.

Under the choice of the forward process in \eqref{forward}, for fixed $t$, $X_t=e^{-t}X_0+\sqrt{1-e^{-2t}}W$, where $W\sim \mathcal N(0 ,I)$ is independent of $X_0$ and equality is in law. As a result, $\nabla\log p_t(x)$ can be written in terms of the posterior mean function $\E[X_0|X_t = x]$ via Tweedie's formula \cite{robbins1992empirical}:  
\begin{align}
\nabla\log p_t(x) = \frac{e^{t}\E[X_0| X_t = x]}{e^{2t}-1} -\frac{e^{2t}x}{e^{2t}-1}. \label{tweedie}
\end{align}

To relate the objective in \eqref{poploss0} to $\dszero^w$, we need an additional assumption on the parametric class.  For $s(x; \tt)$ in an arbitrary parametric class, its associated backward process (i.e., \eqref{backward} with the drift replaced by $Y_t + 2s(Y_t; \tt)$) may not correspond to the time-reversal of an OU process. In other words, given $\tt$, there may not exist a probability measure $\nu$ such that $Z_t$ defined by $\d Z_t = s_t(Z_t; \tt) \d t + \sqrt{2} \d B_t$ coincides in law with the time-reversal of an OU process \eqref{forward} started from $\nu$. Alternatively, one can directly approximate $\mu$ by $\nu(\cdot; \tt)$ (assumed to have bounded second moment). By virtue of \eqref{tweedie}, this induces a parametrization on scores as
\begin{align}
s_t(x; \tt) = \frac{e^{t}\E[Z^\tt_0| Z^\tt_t = x]}{e^{2t}-1} -\frac{e^{2t}x}{e^{2t}-1},\label{ganparam}
\end{align}
where $Z^\tt_t$ is an OU process started from $\nu(\cdot; \tt)$. With this parametrization, we can rewrite \eqref{poploss0} as the following minimization problem:
\begin{align}
\min_{\tt\in\Theta}\int_0^T \frac{e^{2t}\lambda(t)}{(e^{2t}-1)^2}\E_{x\sim X_t}\left[\|\E[X_0|X_t=x]-\E[Z_0^\tt|Z_t^\tt = x]\|_2^2\right] \d t. \label{poploss}
\end{align}
The next lemma, which shares a similar spirit to Lemma~\ref{lemma:foll}, will be needed for our observation in Theorem~\ref{thm:diffusion}. Let $\t_t$ be the observation process of Eldan's $0$-scheme associated with $\mu$. 

\begin{lemma}[\cite{montanari2023sampling}]\label{lm:am}
Let $\tau(t) = \frac{1}{e^{2t}-1}$. If we define the time-changed reversal of $X_t$ as $\bar{Y}_{t} = X_{\tau(t)}$, then $\{\t_t\}_t = \{\sqrt{t(1+t)} \cdot \bar{Y}_t\}_t$, where equality is in law.  
\end{lemma}

We are now ready to state our main observation.   

\begin{theorem}\label{thm:diffusion}
If the drift in \eqref{backward} is parametrized through $\nu(\cdot, \tt)$ as in \eqref{ganparam}, 
then the minimization problem in \eqref{poploss} is equivalent to the following minimization problem:
\begin{align}
\min_{\tt\in\Theta}\int_{0}^\infty \E_{x\sim\t_t}[\|a_t(x)-b_t(x; \tt))\|_2^2] \ w(\d t), \label{fake}
\end{align}
where $(a_t, b_t)=(a_t(\t_t), b_t(\t_t'; \tt))$ is the mean process of joint Eldan's $0$-scheme associated with $\mu$ and $\nu(\cdot; \tt)$ with $(\t_t, \t_t')$ being the corresponding observation process, and 
\begin{align*}
w(\d t) = \frac{1}{2}\lambda(\tau(t))\tau'(\tau(t))\tau'(t) \mathbb I_{\{t\geq\tau^{-1}(T)\}} \d t
\end{align*}
is a $\sigma$-finite measure with support $[\tau^{-1}(T), \infty)$ for the same $\tau(t)$ in Lemma~\ref{lm:am}.  
\end{theorem}

\begin{remark}\label{rem:gd}
When $w$ is a finite measure, i.e., 
\begin{align}
w(\R) = \frac{1}{2}\int_{\tau^{-1}(T)}^\infty\lambda(\tau(t))\tau'(\tau(t))\tau'(t)\d t = \int_0^T\frac{e^{2t}\lambda(t)}{(e^{2t}-1)^2} \d t < \infty, \label{weightt}
\end{align}
$w$ can be normalized to a probability measure on $\R_+$ with unbounded support, and this will not affect the optimization problem. In this case, we may assume $w(\R) = 1$ and compare the objective in \eqref{fake} with $\dszero^w$ with the same $w$:  
\begin{align*}
\dszero^w(\mu, \nu) = \int_0^\infty \E_{(x, x')\sim (\t_t, \t_t')}[\|a_t(x)-b_t(x'; \tt)\|_2^2] \ w(\d t).
\end{align*}
Hence, \(\dszero^w\) can be seen as an analogue of the score-matching loss, but it compares the mean functions on \textit{coupled} trajectories of $\t_t$ and $\t_t’$ rather than on a single trajectory of $\t_t$. This, in some sense, suggests that $\dszero^w$ is more stringent than the score-matching objectives.

To get additional evidence on the heuristics, assume that $\lambda(t)$ is uniformly bounded on $[0, T]$. Under such circumstances, the integrability condition \eqref{weightt} is equivalent to $\int_0^Tt^{-2}\lambda(t)\d t<\infty$ since $\frac{e^{2t}}{(e^{2t}-1)^2}\sim\frac{1}{4t^2}$ as $t\to 0$. For instance, taking $\lambda(t) \propto t^{1+\e}$ would work for any $\e>0$. On the other hand, a popular choice of $\lambda(t)$ used in practice \cite[Section 3.3]{song2020score} sets 
\begin{align*}
\lambda(t) \propto \frac{1}{\E_{(x_0, x_t)\sim (X_0, X_t)}[\|\nabla\log p_t(x|x_0)\|_2^2]} = \frac{1-e^{-2t}}{d}\sim \frac{2t}{d}\quad\text{as $t\to 0$}, 
\end{align*}
where $p_t(x|x_0)$ is the transition kernel from $X_0$ to $X_t$. This choice, however, yields $w(\R)=\infty$. Consequently, the corresponding $\dszero^w$ becomes $\dszero^w(\mu, \nu)=\infty$ if $\mu\neq\nu$. 
\end{remark}

\begin{proof}[Proof of Theorem~\ref{thm:diffusion}]
Let $\xi(t) = \sqrt{t(1+t)}$ for $t\in [0, \infty]$. Fixing $t>0$, we apply Lemma~\ref{lm:am} to conclude that, a.s.,  
\begin{align*}
\E[X_0| X_t = x] &\stackrel{(\mathrm{i})}{=}  \lim_{\e\to 0}\E[X_\e | X_t = x]\\
 &=  \lim_{\e\to 0}\frac{1}{\xi(\tau^{-1}(\e))}\E[\theta_{\tau^{-1}(\e)}\mid \theta_{\tau^{-1}(t)} = \xi(\tau^{-1}(t)) x]\quad\text{(Lemma~\ref{lm:am})}\\
&\stackrel{(\mathrm{ii})}{=} \lim_{\e\to 0}\E\left[\frac{\tau^{-1}(\e)}{\xi(\tau^{-1}(\e))}X_0 + \frac{W_{\tau^{-1}(\e)}}{\xi(\tau^{-1}(\e))}\mid \theta_{\tau^{-1}(t)}=\xi(\tau^{-1}(t)) x\right]\\
&\stackrel{(\mathrm{iii})}{=}   \E[X_0 \mid \theta_{\tau^{-1}(t)} = \xi(\tau^{-1}(t)) x]\\
&\stackrel{(\mathrm{iv})}{=} a_{\tau^{-1}(t)}(\xi(\tau^{-1}(t)) x),
\end{align*}
where (i) follows from the path-continuity of $X_t$ and the dominated convergence theorem applied to $\sup_{0\leq t\leq 1}\|X_t\|_2$, (ii) follows from the law-equivalent representation of $\t_t$ in \eqref{obs} with $C_t = I$, (iii) follows from the dominated convergence theorem, i.e., 
\begin{align*}
\sup_{\e\leq\tau(1)}\left\|\frac{\tau^{-1}(\e)}{\xi(\tau^{-1}(\e))}X_0 + \frac{W_{\tau^{-1}(\e)}}{\xi(\tau^{-1}(\e))}\right\|_2\leq \|X_0\|_2 +  \sup_{s\leq 1}s\|W_{1/s}\|_2
\end{align*}
and the right-hand side is integrable since $sW_{1/s}$ can be identified as a Brownian motion in $\R^d$ by time inversion, and (iv) follows from Remark~\ref{rem:markov}.   

If $Z^\tt_t$ is an OU process started from $\nu(\cdot; \tt)$, then by a similar reasoning, $\E[Z^\tt_0 | Z^\tt_t = x] = b_{\tau^{-1}(t)}(\xi(\tau^{-1}(t)) x)$ a.s., where $b_t(\t'_t; \tt)$ is the mean process of Eldan's $0$-scheme associated with $\nu(\cdot; \tt)$. Consequently, the proof is completed by noting  
\begin{align*}
&\int_0^T \frac{e^{2t}\lambda(t)}{(e^{2t}-1)^2}\E_{x\sim X_t}\left[\|\E[X_0|X_t=x]-\E[Z_0^\tt|Z_t^\tt = x]\|_2^2\right] \d t \\
=&\ \int_0^T  -\frac{\lambda(t)\tau'(t)}{2}\E_{x\sim X_t}\left[\|a_{\tau^{-1}(t)}(\xi(\tau^{-1}(t))x)- b_{\tau^{-1}(t)}(\xi(\tau^{-1}(t))x; \tt)\|_2^2\right] \d t\quad\text{($\tau'(t) = \frac{-2e^{2t}}{(e^{2t}-1)^2}$)}\\
=&\ \int_{\tau^{-1}(T)}^\infty \frac{\lambda(\tau(t))\tau'(\tau(t))\tau'(t)}{2}\E_{x\sim X_{\tau(t)}}\left[\|a_t(\xi(t)x)- b_t(\xi(t)x; \tt)\|_2^2\right] \d t\quad(\text{$t\gets \tau^{-1}(t)$})\\
=&\  \int_{0}^\infty \E_{x\sim \t_t}\left[\|a_t(x)- b_t(x; \tt)\|_2^2\right] w(\d t).\quad (\text{Lemma~\ref{lm:am} + definition of $w$})
\end{align*}
\end{proof}


\section{Estimation of Eldan's $\alpha$-distance}\label{sec:newcomp}

We now discuss numerical methods for estimating Eldan's $\alpha$-distance. Since the squared distance is defined as an expectation, a natural approach is to use MC simulation. However, this requires simulating a joint Eldan's $\alpha$-scheme over an infinite time horizon, which is impractical and necessitates truncation. In addition, one must account for numerical errors arising from the discretization of the underlying SDEs. 

In what follows, we introduce a truncated MC estimator to estimate Eldan's $\alpha$-distance and provide a general framework for its analysis. We then present specific numerical methods to simulate Eldan's $\alpha$-scheme for $\alpha = 0$ and $\alpha = 1/2$, along with an analysis of their approximation errors. 

\subsection{Truncated MC estimation}\label{sec:overalle}
Consider two probability measures $\mu$ and $\nu$ with bounded support, and whose joint Eldan's $\alpha$-scheme $(\mu_t, \nu_t)$ is well-defined. Let $(a_t, b_t)$ be the corresponding mean process, and let $T>0$ be some truncation time. Suppose that $\{(a^{(i)}_t, b^{(i)}_t)\}_{i=1}^M$ are i.i.d. samples of $(a_t, b_t)$, and $(\ah_t^{(i)}, \widehat{b}_t^{(i)})$ are their simulated counterparts defined in the same probability space. The squared Eldan's $\alpha$-distance, $\ds^2(\mu, \nu)$, can be estimated via a truncated MC simulation:
\begin{align}
\widehat{\ds}^2(\mu, \nu) = \frac{1}{M}\sum_{i=1}^M\|\ah_T^{(i)}-\widehat{b}_T^{(i)}\|_2^2. 
\end{align}

To analyze the difference between $\widehat{\ds}^2(\mu, \nu)$ and $\ds^2(\mu, \nu)$, we consider the expected error and apply the triangle inequality to obtain the following upper bound:
\begin{equation}\label{3steperror}
\begin{aligned}
&\E[|\widehat{\ds}^2(\mu, \nu) - \ds^2(\mu, \nu)|]\\
\leq &\ \underbrace{\left|\ds^2(\mu, \nu)-\E[\|a_T-b_T\|_2^2]\right|}_{\text{(truncation error)}} +\underbrace{\E\left[\left|\frac{1}{M}\sum_{i=1}^M\|a_T^{(i)}-b_T^{(i)}\|_2^2-\E[\|a_T-b_T\|_2^2]\right|\right]}_{\text{(MC error)}}\\
& +\underbrace{\E\left[\left|\frac{1}{M}\sum_{i=1}^M\|\ah_T^{(i)}-\widehat{b}_T^{(i)}\|_2^2-\frac{1}{M}\sum_{i=1}^M\|a_T^{(i)}-b_T^{(i)}\|_2^2\right|\right]}_{\text{(discretization error)}}.
\end{aligned}
\end{equation}

The truncation error can be related to the localization rate using SL calculus (i.e., \ito's isometry applied to \eqref{itotimes}):
\begin{align*}
\text{(truncation error)} = \E[\|a_\infty-b_\infty\|_2^2]-\E[\|a_T-b_T\|_2^2]=\int_T^\infty \E[\tr((\Sigma_t-\Lambda_t)^{2-2\alpha})] \d t, 
\end{align*}
where $\Sigma_t$ and $\Lambda_t$ are the covariance process associated with $(\mu_t, \nu_t)$. For $\alpha = 0$ or $\alpha = 1/2$, 
\begin{align*}
\int_T^\infty \E[\tr((\Sigma_t-\Lambda_t)^{2-2\alpha})] \d t\leq \int_T^\infty \E[\tr(\Sigma^{2-2\alpha}_t)+\tr(\Lambda^{2-2\alpha}_t))] \d t \leq\E[\tr(\Sigma_T+\Lambda_T)], 
\end{align*}
where the last step follows from the same reasoning as \Cref{tom's lemma} with $0$ replaced by $T$ and letting $t\to\infty$. Assuming $\tr(\Sigma_0), \tr(\Lambda_0)\lesssim d$, according to \Cref{thm:01}, for $\alpha = 0$, this is of order $\mathcal O(dT^{-1})$; for $\alpha = 1/2$, this is of order $\mathcal O(de^{-T})$. 

Assuming $\supp(\mu), \supp(\nu)\subseteq\B_R(\R^d)$ for some $R>0$, $|a_t\pm b_t|\leq 2R$ for all $t\geq 0$. By Jensen's inequality,
\begin{align*}
\text{(MC error)}\leq \E\left[\left(\frac{1}{M}\sum_{i=1}^M\|a_T^{(i)}-b_T^{(i)}\|_2^2-\E[\|a_T-b_T\|_2^2]\right)^2\right]^{\frac{1}{2}}\leq \frac{2R^2}{\sqrt{M}}, 
\end{align*}
which is independent of $\alpha$. In the unbounded case, the numerator can be replaced by suitable moments of $\mu$ and $\nu$.  

For the discretization error, assuming the estimates $\ah_T, \widehat{b}_T\in\B_R(\R^d)$ (which hold a.s. for the discretization schemes considered later), 
by the elementary equality $|x^2-y^2| = |x+y||x-y|$,   
\begin{align*}
\text{(discretization error)}&\leq\E\left[|\|\ah_T-\widehat{b}_T\|_2^2-\|a_T-b_T\|_2^2|\right]\\
&\leq 4R\cdot\E\left[\|\ah_T-\widehat{b}_T-(a_T-b_T)\|_2\right]\\
&\leq 4R\cdot(\E[\|\ah_T-a_T\|_2^2]^{\frac{1}{2}} + \E[\|\widehat{b}_T-b_T\|_2^2]^{\frac{1}{2}}).
\end{align*}
Thus, the discretization error depends on the accuracy of approximating the mean process of each marginal. For convenience, we focus on simulating the mean process $a_t$. 

In the remainder of this section, we apply the Euler--Maruyama scheme to simulate $a_t$ on $[0, T]$ using a suitable choice of time discretization 
\begin{align}
0 = t_0<t_1\cdots < t_L = T.\label{timedis}
\end{align}
In contrast to classical settings in numerical SDEs \cite{kloeden2013numerical}, the terminal time $T$ here tends to infinity, which requires a more refined analysis leveraging the localization property of Eldan's $\alpha$-scheme. We discuss the details of the numerical scheme separately for the cases $\alpha = 0$ and $\alpha = 1/2$. 
   
\subsection{Numerical approximation of Eldan's $0$-scheme}\label{sec:eldan0num}

\Cref{thm:01} shows that Eldan's $0$-scheme exhibits a slow localization rate. To ensure the truncation error in \eqref{3steperror} bounded by $\mathcal O(\e)$, one requires $T= \mathcal O(d\e^{-1})$. Standard analyses of the uniform Euler--Maruyama scheme under global Lipschitz conditions yield a multiplicative error factor that grows at least exponentially in $d/\e$ \cite[Theorem 10.2.2]{kloeden2013numerical}.

To compensate for this, we employ a mixed time-stepping scheme:
\begin{align}
\Delta t_i = t_i - t_{i-1}= \max\{1, t_{i-1}\}\cdot h, \label{ratio1}
\end{align}
where $0 < h < 1/2$ is the time step parameter. Thus $t_i$ grows linearly while $t_{i-1} < 1$ and geometrically thereafter, satisfying $t_i = (1+h)^{i-j} t_j$ for any $t_j \geq 1$. The rationale for using an increasing step size at large $t$ is that, as Eldan's $0$-scheme becomes increasingly localized, it grows less sensitive to large steps---an intuition we make rigorous below.

Let $\mu$ be a probability measure on $\R^d$ with $\supp(\mu)\subseteq\B_R(\R^d)$ for some $R>0$. Let $\mu_t$ be Eldan's $0$-scheme associated with $\mu$. To simulate $\mu_t$, we employ the finite-dimensional description in \eqref{finite_system}. In this case, $G_t = tI$, and hence it suffices to simulate the corresponding observation process $\t_t$:
\begin{align}
\d\t_t = a_t(\t_t) \d t + \d W_t. \label{dt2}
\end{align}
To this end, we apply the Euler--Maruyama scheme along the time discretization \eqref{timedis} and denote the resulting approximation by $\td_t$. Specifically, for $t\in (t_{i-1}, t_i]$, 
\begin{align}
\d\td_t &= a_{t_{i-1}}(\td_{t_{i-1}}) \d t + \d W_t, \label{dt10}
\end{align}
where we slightly abused notation to let
\begin{align}
a_t(\t) = \frac{\int_{\R^d}x\exp\left\{-\frac{t}{2}\|x\|_2^2+\<\t, x\>\right\}\mu(\d x)}{\int_{\R^d}\exp\left\{-\frac{t}{2}\|x\|_2^2+\<\t, x\>\right\}\mu(\d x)}  \label{mylittlea}
\end{align}
as a deterministic function; the mean process associated with $\mu_t$ is $a_t = a_t(\t_t)$. Based on this, $a_t$ can be approximated by $\widetilde{a}_t = a_t(\td_t)$. 

Computing $\widetilde{a}_t$ using \eqref{mylittlea} involves integral evaluation, and this procedure may incur additional approximation errors if $\mu$ has a large or infinite support. As a result, further approximation of $\widetilde{a}_t$ is required, and we denote the resulting approximation by
\begin{align}
\ah_t = \ah_t(\th_t), \label{athat}
\end{align}
where $\ah_t$ is assumed to be a statistical estimator constructed independently of $\th_t$. Accordingly, the dynamics in \eqref{dt10} are modified into the following form:
\begin{align}
\d\th_t &= \ah_{t_{i-1}}(\th_{t_{i-1}}) \d t + \d W_t. \label{dt1}
\end{align}

The rest of this section is devoted to analyzing the Euler--Maruyama scheme given by \eqref{dt1}. 
We denote by $e_i = \E[\|\t_{t_i}-\th_{t_i}\|_2^2]$ the $i$th-step mean-squared error (MSE). For certain classes of $\mu$, the localization property of $\mu_t$ can be readily transferred for error analysis. One convenient characterization of such measures is the following time-decaying Lipschitz property of their mean processes.
\begin{assumption}\label{ass:1}
There exist constants $s_0, C_0\geq 0$ such that the drift function $a_t(\t)$ in \eqref{mylittlea} is $(C_0/t)$-Lipschitz continuous for $t\geq s_0$.  
\end{assumption}
Assumption~\ref{ass:1} is a property of $\mu$. It holds, for instance, whenever $\mu$ is $\kappa$-log-concave for some $\kappa\in\R$; since $\kappa$ may be negative, this covers a strictly broader class than log-concave measures.
\begin{proposition}\label{logconcave-td}
Let $\mu$ be $\kappa$-log-concave for some $\kappa\in\R$. Then Assumption~\ref{ass:1} holds for every $C_0>1$ with $s_0\geq\max\{C_0\kappa/(1-C_0),\, 0\}$. If $\kappa\geq 0$, it also holds for $C_0=1$ with $s_0\geq 0$.
\end{proposition}
\begin{proof}
Since $\mu$ is $\kappa$-log-concave, the measure $\rho_t(\t)(\d x)\propto\exp\left\{-\frac{t}{2}\|x\|_2^2+\<\t, x\>\right\}\mu(\d x)$ is $(t+\kappa)$-strongly log-concave for $t\geq -\kappa$ and every $\t\in\R^d$. By \cite[Lemma 1]{eldan2014bounding}, the covariance of $\rho_t(\t)$, which coincides with $\nabla_\t a_t(\t)$ by the generating property (see~\eqref{mgf}), is upper bounded by $I/(t+\kappa)$. This implies that the Lipschitz constant of $a_t(\t)$ is bounded by $1/(t+\kappa)$ for all $t>-\kappa$. Thus, Assumption~\ref{ass:1} holds for any $C_0>1$ with $s_0\geq\max\{C_0\kappa/(1-C_0), 0\}$. When $\kappa\geq 0$, we have $1/(t+\kappa)\leq 1/t$ directly, so $C_0=1$ and $s_0\geq 0$ suffice.
\end{proof}

\begin{theorem}\label{thm:logconcave-0}
Suppose that $\mu$ satisfies Assumption~\ref{ass:1} for some $C_0>0$ and $s_0\geq 1$, $\supp(\mu)\subseteq \B_R(\R^d)$, and $\tr(\Sigma)\leq d$, where $\Sigma$ is the covariance of $\mu$. Fixing $\e< \frac{\sqrt{d}}{Rs_0+C_0}$, set $T = t_L = d/\e>s_0$, and the step size parameter 
\begin{align}
h = \min\left\{\frac{5^{-C_{0, +}}}{R\sqrt{2\log (\frac{d}{\e})}}, \ s_0^{C_{0,+}-\frac{1}{2}}e^{-2R^2s_0}, \ 1\right\}\frac{\e^{C_{0, +}}}{d^{C_{0, +}-\frac{1}{2}}},\quad C_{0, +} \coloneqq \max\{C_0, 1\}.\label{superh}
\end{align}
Let $\ah_t$ denote the simulated mean process defined in \eqref{athat}, and suppose for each $t_i$, $\ah_{t_i}(\th_{t_i})$ is an estimator for $a_{t_i}(\th_{t_i})$ constructed independently from $\th_t$, with its bias and MSE satisfying
\begin{equation}\label{variance-bdd}
\begin{aligned}
\mathrm{(Bias \ bound)}: & \quad \max_{i< L}\left\|\E[\ah_{t_i}(\th_{t_i})-a_{t_i}(\th_{t_i})\mid\th_{t_i}]\right\|_2\leq\frac{R^2h^2}{24}\\[6pt]
\mathrm{(MSE\  bound)}: & \quad\E[\|\ah_{t_i}(\th_{t_i})-a_{t_i}(\th_{t_i})\|_2^2\mid\th_{t_i}]\leq \begin{cases}
\displaystyle\frac{R^2dh}{3} & i<L\\
\e^2 & i = L
\end{cases}
\end{aligned}
\end{equation}
Then, $\E[\|\ah_T - a_T\|_2^2]\le 6C^2_{0, +}\e^2$. In particular, the total number of steps for discretization $L=\mathcal O(h^{-1}\log(d/\e))$. 
\end{theorem}

The assumptions $\tr(\Sigma)\leq d$ and $\e< \frac{\sqrt{d}}{Rs_0+C_0}$ are used to simplify the statements and are not essential. The proof is based on a two-stage recursion analysis together with a bootstrap argument; we defer it to \Cref{pf:001}. The boundedness assumption is mainly used to bound $\|\Sigma_t\|_2$ pathwise and thus can be extended to certain unbounded settings, such as strongly log-concave measures. 

\begin{remark}
Condition~\eqref{variance-bdd} imposes different requirements on the estimation bias and variance of $\ah_{t_i}$, with the bias condition being substantially more stringent. This discrepancy arises because the bias interacts with the discretization error in a non-vanishing way, leading to an error that accumulates over successive time steps; such interaction vanishes for the variance part because of the independence assumption.  If $\ah_{t_i}$ is computed via exact MC, the estimator is unbiased, so only the MSE part of Condition~\eqref{variance-bdd} is active. When exact sampling becomes computationally prohibitive (e.g., when $d$ is large), one typically relies on approximate techniques such as Markov chain Monte Carlo, and the resulting bias is governed by the corresponding mixing rate. We do not consider these additional computational overheads in this work.  
\end{remark}

To assess the computational complexity of the truncated MC estimator for computing Eldan's $0$-distance between two measures satisfying Assumption~\ref{ass:1} with $C_0, s_0\geq 1$, we take $\ah_t$ as an exact MC estimator. We measure the computational cost in terms of operation counts, focusing on its dependence on the precision $\e$ and the dimension $d$. According to the analyses in \Cref{sec:overalle}, computing Eldan's $0$-distance to precision $\e$ requires simulating $\mathcal{O}(\e^{-2})$ trajectories of joint Eldan's $\alpha$-scheme. Each trajectory involves $L=\mathcal{O}(h^{-1}\log(d/\e))$ time discretizations. For the $i$th discretization, the variance of the vanilla MC estimator $\ah_{t_i}$ using $N_i$ samples is bounded by $\mathcal O(N_i^{-1})$. To satisfy the MSE bound in \eqref{variance-bdd}, we require $N_i = \mathcal O((dh)^{-1})$ for the intermediate steps and $N_T = \mathcal O(\e^{-2})$ for the terminal step. Assuming generating each sample in $\R^d$ requires $\mathcal O(d)$ operations, the computational cost to simulate a single trajectory is of order $\mathcal O(d(L(dh)^{-1}+\e^{-2})) = \mathcal O(Lh^{-1}+d\e^{-2})$. Multiplying this by the $\mathcal O(\e^{-2})$ outer trajectories yields a total computational complexity proportional to $\mathcal O(Lh^{-1}\e^{-2} +d\e^{-4})$. Substituting the definitions of $h$ and $L$, the first term dominates, yielding a total complexity of order $\mathcal O(\e^{-2(C_0+1)}d^{2C_{0}-1}\log^2(d/\e))$, where the implicit constant depends on $s_0$. When two measures are log-concave, one can take $C_{0} = s_0=1$, so the overall complexity becomes $\mathcal{O}\left(d\e^{-4}\log^2(d/\e)\right)$, where the implicit constant is absolute.  

Despite its theoretical convenience, Assumption~\ref{ass:1} may not hold for other measures of interest. For instance, when $\mu = (\delta_{-1}+\delta_{1})/2$, $a_t(\t) = \tanh(\t)$, which is independent of $t$ and has a bounded Lipschitz constant. However, in light of \eqref{obs}, $\t_t$ scales linearly in $t$ as $t\to\infty$, and $\t_t/t$ converges to some point in $\supp(\mu)$. A modified version of Assumption~\ref{ass:1} may still hold when tracking the dynamics of $\t_t$. The following lemma makes this observation precise in the case where $\mu$ has a finite support.

\begin{lemma}\label{lm:dsrate}
Assume that $\mu = \sum_{i=1}^n\delta_{x_i}/n$ is an empirical measure supported on $n$ distinct points $x_{1}, \ldots, x_n\in \B_R(\R^d)$. Let $0<\gamma\leq\min_{i\neq j}\|x_i-x_j\|_2$ denote a lower bound on the packing distance of $\supp(\mu)$. Let $\t_t$ be Eldan's $0$-scheme associated with $\mu$ and $\iota_t\in\R^d$ be another process jointly defined with $\t_t$. For all large $t$ satisfying $t\geq \max\{(4\log(8d\gamma^2t)+8d)/\gamma^2, 1\}$, 
\begin{align}
\E[\|a_t(\t_t)-a_t(\iota_t)\|^2_2]\leq\max\left\{8R^2\exp\left\{-\frac{\gamma^2t}{256}+\frac{d}{2}\right\}, \ \frac{520R^2}{\gamma^2t^2}\E\left[\|\t_t-\iota_t\|^2_2\right]\right\}. \label{mixedbdd}
\end{align}
\end{lemma}
The proof is based on a pathwise analysis of $\t_t$ and is deferred to \Cref{pf:002}. The bound in \eqref{mixedbdd} contains a large-deviation term arising from controlling the Brownian motion and a time-dependent Lipschitz estimate that resembles Assumption~\ref{ass:1} when tracking the joint dynamics of $(\t_t, \iota_t)$. Using \Cref{lm:dsrate}, we establish the following bound on the discretization error for $\mu$ with finite support. 

\begin{theorem}\label{thm:eperror}
Let $\mu = \sum_{i=1}^n\delta_{x_i}/n$ be an empirical measure supported on $n$ distinct points $x_{1}, \ldots, x_n\in\B_R(\R^d)$. Let $0<\gamma\leq\min\{\min_{i\neq j}\|x_i-x_j\|_2, 1\}$ denote a lower bound on the packing distance of $\supp(\mu)$. Let $T = d/\e$ and set
\begin{align*}
h = \frac{1}{C_0\max\{R, 1\}}\frac{\e^{2R^2C_1+C_0}}{d^{2R^2C_1+C_0-\frac{1}{2}}}\exp\left\{-\frac{256R^2d}{\gamma^2}\right\},\quad\text{where $C_0 = \frac{520R^2}{\gamma^2}$, $C_1 = \frac{768C_0}{\gamma^2}$}. 
\end{align*}
Let $\ah_t$ denote the simulated mean process defined in \eqref{athat} with exact evaluation at $\th_t$. For all sufficiently small $\e$, $\E[\|\ah_T - a_T\|_2^2]\leq \max\{2, 3R^4\}\e^2$. In particular, the total number of steps for discretization $L$ is of order $\mathcal O(h^{-1}\log(d/\e))$. 
\end{theorem}
The assumption on the exact evaluation of $\ah_t$ is not overly restrictive since $\mu$ has a finite support. The proof builds on \Cref{lm:dsrate} and reuses the recursion-and-bootstrap framework of \Cref{thm:logconcave-0}, but the time-decaying Lipschitz property of Assumption~\ref{ass:1} no longer holds. The key new ingredient is a case analysis around the last step $i_1$ at which the large-deviation term in \eqref{mixedbdd} dominates: before $i_1$ the error stays exponentially small, and after $i_1$ the localized bound takes over and the recursion of \Cref{thm:logconcave-0} applies. The details are deferred to \Cref{pf:003}. 

\begin{remark}
The significance of \Cref{thm:eperror} is theoretical, as it shows that for probability measures with finite support, there exists a $\mathrm{poly}(\e^{-1})$-time algorithm to compute their Eldan's 0-distance within error $\e$, albeit with an implicit constant depending exponentially on $d$. Since our analysis uses \Cref{lm:dsrate}, which holds for any jointly defined process $\iota_t$, the constants $C_0$ and $C_1$ in \Cref{thm:eperror} may be suboptimal and could potentially be improved with a more refined analysis (i.e., the estimate in \eqref{bdd:final1}). Nevertheless, the exponential dependence on $d$ may be unavoidable without imposing further structures on $\mu$. Indeed, standard Brownian motion in $\R^d$ scales as $\mathcal O(\sqrt{dt})$ while the signal strength (assuming the norm of the point is constant order) scales as $\mathcal O(t)$ (e.g., see~\eqref{obs}). Consequently, at least $\mathcal O(d)$ time is required to ensure the observation process $\t_t$ is getting close enough to the signal regime where strong localization kicks in.  
\end{remark}

\subsection{Numerical approximation of Eldan's $\frac{1}{2}$-scheme}\label{sec:eldan1/2num}

Simulation of Eldan's $\alpha$-scheme for $\alpha>0$ is more challenging due to the potential degeneracy of $\Sigma_t^\dagger$. 
To address this issue, we consider a regularized scheme by replacing the control process $C_t=(\Sigma_t^\dagger)^\alpha$ with $C_t=(\Sigma_t+\delta I)^{-\alpha}$ for some $\delta>0$, as introduced in \Cref{sec:reg}. In this section, we focus on the case $\alpha = 1/2$ and consider the $\delta$-regularized Eldan's $\frac{1}{2}$-scheme, whose finite-dimensional description is defined as follows:
\begin{equation}\label{hard1}
\begin{aligned}
\begin{cases}
\d\t^{(\delta)}_{t} &= (\Sigma^{(\delta)}_{t}+\delta I)^{-1}a^{(\delta)}_{t} \d t + (\Sigma^{(\delta)}_{t}+\delta I)^{-\frac{1}{2}} \d W_t\\
\d G^{(\delta)}_{t} &= (\Sigma^{(\delta)}_{t}+\delta I)^{-1} \d t 
\end{cases}
\end{aligned},
\end{equation}
where $a^{(\delta)}_{t} = \E_{\mu^{(\delta)}_{t}}[x]$ and $\Sigma^{(\delta)}_{t} = \E_{\mu^{(\delta)}_{t}}[x\otimes x]-a^{(\delta)}_{t}\otimes a^{(\delta)}_{t}$ are the mean and covariance processes of the SL process $\mu^{(\delta)}_{t}$ defined via \eqref{hard1}:  
\begin{align}
\mu^{(\delta)}_{t}(\d x) = \frac{\exp\left\{-\frac{1}{2}\<x, G^{(\delta)}_{t} x\>+\<\t^{(\delta)}_{t}, x\>\right\} \mu (\d x)}{\displaystyle\int_{\R^d}\exp\left\{-\frac{1}{2}\<z, G^{(\delta)}_{t} z\>+\<\t^{(\delta)}_{t}, z\>\right\} \mu (\d z)}.\label{deltasl}
\end{align}
Note that both $a^{(\delta)}_{t}$ and $\Sigma^{(\delta)}_{t}$ are deterministic functions of $\t^{(\delta)}_{t}$ and $G^{(\delta)}_{t}$. 
According to SL calculus in \Cref{sec:slcal}, $a^{(\delta)}_{t}$ and $\Sigma^{(\delta)}_{t}$ satisfy
\begin{equation}\label{l9401}
\begin{aligned}
\d a_t^{(\delta)} &= \Sigma^{(\delta)}_{t}(\Sigma^{(\delta)}_{t}+\delta I)^{-\frac{1}{2}} \d W_t,\\
\d\Sigma^{(\delta)}_{t} &= -(\Sigma^{(\delta)}_{t})^2(\Sigma^{(\delta)}_{t}+\delta I)^{-1}\d t + \M^{(3)}[\mu^{(\delta)}_{t}] (\Sigma^{(\delta)}_{t}+\delta I)^{-\frac{1}{2}} \d W_t, 
\end{aligned}
\end{equation}
where $\M^{(3)}[\mu^{(\delta)}_{t}] = \E_{\mu^{(\delta)}_{t}}[(x-a^{(\delta)}_{t})^{\otimes 3}]$. Since both $\t^{(\delta)}_{t}$ and $G^{(\delta)}_t$ are random, simulation requires discretization of the joint variables $(\t^{(\delta)}_{t}, G^{(\delta)}_{t})$.

While the regularization parameter $\delta$ resolves the degeneracy issue in practice, it introduces significant challenges for analysis. 
First, the discrepancy between the couplings induced by the exact Eldan's $\frac{1}{2}$-scheme ($\delta = 0$) and its regularized version depends on $\delta$, and one expects this discrepancy to increase with $\delta$. Meanwhile, $\delta$ affects the Lipschitz constants of the drift and diffusion coefficients in \eqref{hard1}; a direct analysis of the approximation error of $\t^{(\delta)}_t$ would yield a bound that worsens as $\delta$ decreases, and transferring such an error to $a^{(\delta)}_{t}$ via a global Lipschitz bound would result in an error that is too large to be useful when $\delta$ is small. Since $a^{(\delta)}_{t}$ remains well controlled based on its SDE, one might hope to leverage a time-decaying Lipschitz condition, similar to Assumption~\ref{ass:1}, to counterbalance the degeneracy effect through localization. However, formalizing this intuition is intractable due to the randomness of $G^{(\delta)}_t$. 

To gain a more concrete understanding of these discretization errors, we specialize to the case where $\mu\sim \mathcal N(0, \Sigma)$ is a centered Gaussian measure on $\R^d$. The Gaussian assumption renders the dynamics of $\Sigma^{(\delta)}_{t}$ (and hence $G^{(\delta)}_t$) deterministic, thereby allowing for a precise pathwise analysis that quantifies the simulation errors arising from both regularization and discretization. Even in this special setting, the analysis is not trivial. Without loss of generality, we assume $\Sigma$ is nondegenerate and write its eigendecomposition as $\Sigma = \sum_{k=1}^d\sigma_ku_ku_k^\top$, where $\sigma_k>0$ and $u_k$'s are eigenvectors that form an orthonormal basis in $\R^d$. 

We now consider a uniform Euler--Maruyama scheme along the discretization in \eqref{timedis}, applied to the regularized scheme in \eqref{hard1} with step size $h>0$, i.e., $t_i = ih$ and $T = Lh$. For $t\in (t_{i-1}, t_i]$, 
\begin{equation}\label{hard2}
\begin{aligned}
\begin{cases}
\d\th^{(\delta)}_{t} &= (\widehat{\Sigma}^{(\delta)}_{t_{i-1}}+\delta I)^{-1}\widehat{a}^{(\delta)}_{t_{i-1}} \d t + (\widehat{\Sigma}^{(\delta)}_{t_{i-1}}+\delta I)^{-\frac{1}{2}} \d W_t\\
\d {\widehat{G}^{(\delta)}_{t}} &= (\widehat{\Sigma}^{(\delta)}_{t_{i-1}}+\delta I)^{-1} \d t 
\end{cases}
\end{aligned},
\end{equation}
where $\widehat{a}^{(\delta)}_{t} = \E_{\widehat{\mu}^{(\delta)}_{t}}[x]$ and $\widehat{\Sigma}^{(\delta)}_{t} = \E_{\widehat{\mu}^{(\delta)}_{t}}[x\otimes x]-\widehat{a}^{(\delta)}_{t}\otimes \widehat{a}^{(\delta)}_{t}$ are the mean and covariance processes of the simulated SL process $\widehat{\mu}^{(\delta)}_{t}$ defined via \eqref{hard2}:  
\begin{align}
\widehat{\mu}^{(\delta)}_{t}(\d x) = \frac{\exp\left\{-\frac{1}{2}\<x, \widehat{G}^{(\delta)}_{t} x\>+\<\th^{(\delta)}_{t}, x\>\right\} \mu (\d x)}{\displaystyle\int_{\R^d}\exp\left\{-\frac{1}{2}\<z, \widehat{G}^{(\delta)}_{t} z\>+\<\th^{(\delta)}_{t}, z\>\right\} \mu (\d z)}.\label{deltasl1}
\end{align}
Note that $\widehat{a}^{(\delta)}_{t}$ and $\widehat{\Sigma}^{(\delta)}_{t}$ can be exactly evaluated under the Gaussian assumption. 
An application of \ito's formula reveals the dynamics of $\widehat{a}^{(\delta)}_{t}$ and $\widehat{\Sigma}^{(\delta)}_{t}$ on $(t_{i-1}, t_i]$:
\begin{equation}\label{l9402}
\begin{aligned}
\d {\widehat{a}^{(\delta)}_t} &= -\widehat{\Sigma}_t^{(\delta)}(\widehat{\Sigma}_{t_{i-1}}^{(\delta)}+\delta I)^{-1}(\ah_t^{(\delta)}-\ah_{t_{i-1}}^{(\delta)}) \d t + \widehat{\Sigma}_t^{(\delta)}(\widehat{\Sigma}^{(\delta)}_{t_{i-1}}+\delta I)^{-\frac{1}{2}}\d W_t\\
\d {\widehat{\Sigma}^{(\delta)}_{t}} &=-(\widehat{\Sigma}^{(\delta)}_t)^2 (\widehat{\Sigma}^{(\delta)}_{t_{i-1}}+\delta I)^{-1} \d t
\end{aligned},
\end{equation}
where the diffusion term for ${\widehat{\Sigma}^{(\delta)}_{t}}$ vanishes because Gaussian distributions have zero third-order centered moments (and similarly for $\Sigma^{(\delta)}_{t}$ in \eqref{l9401}). Consequently, both $\Sigma^{(\delta)}_{t}$ and $\widehat{\Sigma}^{(\delta)}_{t}$ are deterministic (and hence $G^{(\delta)}_{t}$ and $\widehat{G}^{(\delta)}_{t}$) and their dynamics are described via piecewise autonomous ODEs. This special structure allows us to analyze the error of $\widehat{a}^{(\delta)}_{t}$ directly rather than first controlling $(\th^{(\delta)}_{t}, \widehat{G}^{(\delta)}_{t})$ and then propagating the error to $\widehat{a}^{(\delta)}_{t}$ via the Lipschitz property. 

Since $\Sigma^{(\delta)}_{0}=\widehat{\Sigma}^{(\delta)}_{0}=\Sigma$, the dynamics of $\Sigma^{(\delta)}_{t}$ and $\widehat{\Sigma}^{(\delta)}_{t}$ can be equivalently characterized through their eigenvalues. Writing $\Sigma^{(\delta)}_t$ and $\widehat{\Sigma}^{(\delta)}_t$ as
\begin{align}
\Sigma^{(\delta)}_t = \sum_{k=1}^d\sigma^{(\delta)}_{k, t}u_ku_k^\top, \quad\quad \widehat{\Sigma}^{(\delta)}_t = \sum_{k=1}^d\widehat{\sigma}^{(\delta)}_{k, t}u_ku_k^\top, \label{eigfun}
\end{align}
where $\sigma^{(\delta)}_{k, t}$ and $\widehat{\sigma}^{(\delta)}_{k, t}$ are the eigenvalues of $\Sigma^{(\delta)}_t$ and $\widehat{\Sigma}^{(\delta)}_t$ associated with the eigenvector $u_k$, respectively, the covariance dynamics in \eqref{l9401} and \eqref{l9402} can be represented in terms of $\sigma^{(\delta)}_{k, t}$ and $\widehat{\sigma}^{(\delta)}_{k, t}$ as follows. For $1\leq k\leq d$ and $t\in (t_{i-1}, t_i]$, 
\begin{align}\label{niceodes}
\begin{cases}
\d \sigma^{(\delta)}_{k, t} &= -\frac{(\sigma^{(\delta)}_{k, t})^2}{\sigma^{(\delta)}_{k, t}+\delta} \d t\\
\sigma^{(\delta)}_{k, 0} &= \sigma_k
\end{cases}, 
\quad\quad 
\begin{cases}
\d {\widehat{\sigma}^{(\delta)}_{k, t}} &= -\frac{(\widehat{\sigma}^{(\delta)}_{k, t})^2}{\widehat{\sigma}^{(\delta)}_{k, t_{i-1}}+\delta} \d t\\
\widehat{\sigma}^{(\delta)}_{k, 0} &= \sigma_k
\end{cases}.
\end{align}
It is worth noting that $\widehat{\sigma}^{(\delta)}_{k, t}$ does not correspond to the standard forward Euler discretization of $\sigma^{(\delta)}_{k, t}$ due to the dependence on the current continuous-time value in the numerator. Instead, it corresponds to the forward Euler discretization of the inverse $1/\sigma^{(\delta)}_{k, t}$, which follows by the chain rule. This structure is not surprising; the Euler--Maruyama scheme is applied to $G^{(\delta)}_{t}$, which is related to $\Sigma^{(\delta)}_{t}$ through the precision update $\Sigma^{(\delta)}_{t} = (\Sigma^{-1} + G^{(\delta)}_{t})^{-1}$ for Gaussian measures. Consequently, this semi-implicit discretization ensures that $\widehat{\sigma}^{(\delta)}_{k, t}$ remains strictly positive for all step sizes $h>0$, rendering the covariance update unconditionally stable. This observation plays an important role in the proof of \Cref{thm:1/2-gaussian}. 

To analyze the approximation error, let $a_t = a^{(0)}_t$ and $\Sigma_t = \Sigma^{(0)}_{t}$ be the mean and covariance processes of the exact Eldan's $\frac{1}{2}$-scheme when $\delta = 0$. The MSE of $\ah^{(\delta)}_T$ for $a_T$ can be bounded as 
\begin{align*}
\E[\|\ah^{(\delta)}_{T}-a_T\|_2^2]\leq 2\E[\|a^{(\delta)}_{T}-a_T\|_2^2] + 2\E[\|\ah^{(\delta)}_{T}-a^{(\delta)}_{T}\|_2^2]. 
\end{align*} 
The first term arises from regularization, and the second from the discretization of the regularized scheme. In a similar spirit to \Cref{thm:gauss}, the next lemma relates both errors to the eigenvalue dynamics of $\Sigma^{(\delta)}_{t}$ and $\widehat{\Sigma}^{(\delta)}_{t}$. 

\begin{lemma}\label{lm:eigs}
Let $\widehat{a}^{(\delta)}_t$ denote the simulated mean process in \eqref{l9402} obtained from the uniform Euler--Maruyama scheme \eqref{hard2} with step size $h$, and let $a_t$ denote the exact Eldan's $\frac{1}{2}$-scheme defined on the same probability space. For $\mu\sim \mathcal N(0, \Sigma)$, let $\sigma^{(\delta)}_{k, t}$ and $\widehat{\sigma}^{(\delta)}_{k, t}$ denote the $k$th eigenvalues of $\Sigma^{(\delta)}_t$ and $\widehat{\Sigma}^{(\delta)}_t$, respectively, as defined in \eqref{eigfun}, and write $\sigma_{k, t} = \sigma^{(0)}_{k, t}$. For $T = Lh$, 
\begin{equation}\label{eq:0159}
\begin{aligned}
\E[\|a^{(\delta)}_{T}-a_T\|_2^2] &= \int_0^T \sum_{k=1}^d\left(f(\sigma_{k, t}, \sigma_{k, t}; 0)- f(\sigma^{(\delta)}_{k, t}, \sigma^{(\delta)}_{k, t}; \delta)\right)^2 \d t\\
\E[\|\ah^{(\delta)}_{T}-a^{(\delta)}_{T}\|_2^2] &= \sum_{i=1}^L\int_{t_{i-1}}^{t_i}\sum_{k=1}^d\left(f(\widehat{\sigma}^{(\delta)}_{k, t_i}, \widehat{\sigma}^{(\delta)}_{k, t_{i-1}}; \delta)- f(\sigma^{(\delta)}_{k, t}, \sigma^{(\delta)}_{k, t}; \delta)\right)^2 \d t
\end{aligned},
\end{equation}
where $f(x, y; \delta) = x/\sqrt{y+\delta}$ for all $x, y>0$ and $\delta\geq 0$ and $t_i = ih$. 
\end{lemma}

\begin{proof}
For the regularization error, applying \eqref{at} to $a^{(\delta)}_t$ and $-a_t$ shows that $a^{(\delta)}_t-a_t$ is a martingale:
\begin{align*}
\d (a^{(\delta)}_t-a_t) &= \left(\Sigma_t^{(\delta)}(\Sigma_t^{(\delta)}+\delta I)^{-\frac{1}{2}}-\Sigma^{\frac{1}{2}}_t\right) \d W_t. 
\end{align*}
Substituting the eigendecompositions of $\Sigma_t$ and $\Sigma_t^{(\delta)}$ in \eqref{eigfun} into the integrand and invoking \ito's isometry yields the equation in \eqref{eq:0159}.

For the discretization error, we first consider $\ah^{(\delta)}_t$ on an interval $(t_{i-1}, t_i]$. Note that $\ah^{(\delta)}_t$ is a time-inhomogeneous OU process and thus not a martingale in general (the mean-reverting drift appears because of discretization). Nevertheless, it can be explicitly solved by applying \ito's formula to $(\widehat{\Sigma}^{(\delta)}_t)^{-1}(\ah^{(\delta)}_t-\ah^{(\delta)}_{t_{i-1}})$ (which is a martingale): 
\begin{align*}
\ah^{(\delta)}_t = \ah_{t_{i-1}}^{(\delta)} + \widehat{\Sigma}_t^{(\delta)}(\widehat{\Sigma}_{t_{i-1}}^{(\delta)}+\delta I)^{-\frac{1}{2}} (W_t - W_{t_{i-1}}). 
\end{align*}
Evaluating $t$ at $t_i$ yields
\begin{align*}
\ah^{(\delta)}_{t_i} - \ah_{t_{i-1}}^{(\delta)} = \int_{t_{i-1}}^{t_i}\widehat{\Sigma}_{t_i}^{(\delta)}(\widehat{\Sigma}_{t_{i-1}}^{(\delta)}+\delta I)^{-\frac{1}{2}} \d W_t.
\end{align*}
Combining this with the SDE of $-a^{(\delta)}_t$, and noting $T = t_L$, yields 
\begin{align*}
\ah^{(\delta)}_{T} - a^{(\delta)}_T &= \sum_{i=1}^L  \left(\ah^{(\delta)}_{t_i}-\ah^{(\delta)}_{t_{i-1}}-(a^{(\delta)}_{t_i}-a^{(\delta)}_{t_{i-1}})\right)\\
& = \sum_{i=1}^L\int_{t_{i-1}}^{t_i}\widehat{\Sigma}_{t_i}^{(\delta)}(\widehat{\Sigma}_{t_{i-1}}^{(\delta)}+\delta I)^{-\frac{1}{2}} - \Sigma_t^{(\delta)}(\Sigma_t^{(\delta)}+\delta I)^{-\frac{1}{2}} \d W_t. 
\end{align*}
Since the summands are independent, the desired result follows by applying \ito's isometry. 
\end{proof}

\Cref{lm:eigs} demonstrates that the regularization and discretization errors reduce to the perturbation and discretization of the ODEs in \eqref{niceodes}, which can be analyzed using classical numerical analysis tools. The following result formalizes this observation to bound these errors in terms of $\delta$ and $h$. 

\begin{theorem}\label{thm:1/2-gaussian}
Let $\widehat{a}^{(\delta)}_t$ denote the simulated mean process in \eqref{l9402} obtained from the uniform Euler--Maruyama scheme \eqref{hard2} with step size $h$ applied to that of the $\delta$-regularized Eldan's $\frac{1}{2}$-scheme $a^{(\delta)}_t$, and let $a_t=a^{(0)}_t$. If $\mu\sim \mathcal N(0, \Sigma)$, then 
\begin{align*}
\E[\|a^{(\delta)}_{T}-a_T\|_2^2]\lesssim dT\delta\quad\quad\mathrm{(Regularization\ error)}. 
\end{align*}
Moreover, if we further assume $2hT\leq 1$ and $\delta\leq e^{-T}$, then the following hold true:
\begin{align*}
\E[\|\ah^{(\delta)}_{T}-a^{(\delta)}_{T}\|_2^2]\lesssim \max\{d, \tr(\Sigma)\}h^2\quad\quad\mathrm{(Discretization\ error)}. 
\end{align*}
The implicit constants in both estimates are absolute. 
\end{theorem}

The proof relies on a meticulous analysis of the coupled system in \eqref{niceodes} and is provided in \Cref{euler:ode}. To ensure the total error is of order $\mathcal O(\e^2)$, one can first choose $\delta$ sufficiently small, e.g., $\delta = (\e/d)^2/\log(d/\e)$ and $T = \log (d/\e)$ so that the regularization error is of order $\mathcal O(\e^2)$. The choice of $T$ here is in accordance with the truncation error in \eqref{3steperror} to ensure it is of order $\mathcal O(\e)$. For the discretization error, assuming $\tr(\Sigma)\lesssim d$, then one can choose $h = \e/\sqrt{d}$ to ensure the discretization error is of order $\mathcal O(\e^2)$. Up to logarithmic factors, this matches the lower bound (when $C_{0, +} = 1$) on the step size parameter obtained for simulating Eldan's $0$-scheme in \eqref{superh}. 

\begin{remark}
The discretization error depends neither on the truncation time $T$ nor the regularization parameter $\delta$. The former is due to an exact analysis of the localization effect; the latter arises partially because we directly analyze the dynamics of the mean process rather than that of the observation process. Moreover, the semi-implicit update rule in \eqref{niceodes} corresponds to a forward Euler scheme applied in the inverse domain of the covariance process's spectrum, which is unconditionally stable. This also plays a critical role in mitigating any hidden $\delta$-dependence in the analysis of the mean process.
\end{remark}

A limitation of \Cref{thm:1/2-gaussian} is that it relies essentially on the Gaussian assumption. Specifically, the proof uses the fact that Gaussian measures have vanishing third-order centered moments, which renders the covariance process $\Sigma_t$ deterministic and reduces its dynamics to a scalar ODE for each eigenvalue, thus leading to a sharp error bound. For general measures, the covariance SDE \eqref{dcov} retains a stochastic diffusion term driven by $\mathcal M^{(3)}[\mu_t]$, which prevents the decoupled eigenvalue analysis underlying \Cref{lm:eigs}. Extending \Cref{lm:eigs} to general measures is therefore highly nontrivial and would require new techniques; we leave this as an open problem.

We note that the $\alpha = 0$ case does not face this obstruction, since the constant control process $C_t = I$ avoids any dependence on $\Sigma_t$ and allows for a direct analysis of discretization errors. In this sense, \Cref{thm:1/2-gaussian} is best understood as establishing the correctness of the discretization scheme in a tractable special case and providing heuristic baselines, rather than as evidence of algorithmic guarantees for the $\alpha = 1/2$ scheme in the broader setting.

\subsection{Proofs of technical results}
\subsubsection{Proof of \Cref{thm:logconcave-0}}\label{pf:001}
For $t\in (t_{i-1}, t_i)$, taking the difference between \eqref{dt2} and \eqref{dt1}, 
\begin{align*}
\d (\t_{t} - \th_{t}) = (a_t(\t_t) - \ah_{t_{i-1}}(\th_{t_{i-1}})) \d t. 
\end{align*}
Applying the chain rule, 
\begin{align*}
\d \|\t_{t} - \th_t\|_2^2 = 2\<\t_{t} - \th_t,  a_t(\t_t) - \ah_{t_{i-1}}(\th_{t_{i-1}})\> \d t.
\end{align*}
Integrating $t$ over $(t_{i-1}, t_i)$ and reorganizing terms, 
\begin{align*}
\|\t_{t_i} - \th_{t_i}\|_2^2 &= \|\t_{t_{i-1}} - \th_{t_{i-1}}\|_2^2 + 2\int_{t_{i-1}}^{t_i}\<\t_t - \th_t,  a_t(\t_t) - \ah_{t_{i-1}}(\th_{t_{i-1}})\> \d t\\
& = \|\t_{t_{i-1}} - \th_{t_{i-1}}\|_2^2 + 2\int_{t_{i-1}}^{t_i}\<A_1 + A_2 + A_3 + A_4,  B_1 + B_2 + B_3\> \d t,
\end{align*}
where 
\begin{equation*}
\begin{aligned}
A_1 &\coloneqq \int_{t_{i-1}}^t (a_s(\t_s)-a_{t_{i-1}}(\t_{t_{i-1}}))\, \d s 
&\qquad
B_1 &\coloneqq a_t(\t_t) - a_{t_{i-1}}(\t_{t_{i-1}}) \\
A_2 &\coloneqq (a_{t_{i-1}}(\t_{t_{i-1}})-a_{t_{i-1}}(\th_{t_{i-1}}))(t-t_{i-1}) 
&\qquad
B_2 &\coloneqq a_{t_{i-1}}(\t_{t_{i-1}})-a_{t_{i-1}}(\th_{t_{i-1}}) \\
A_3 &\coloneqq (a_{t_{i-1}}(\th_{t_{i-1}})-\ah_{t_{i-1}}(\th_{t_{i-1}}))(t-t_{i-1}) 
&\qquad
B_3 &\coloneqq a_{t_{i-1}}(\th_{t_{i-1}}) - \ah_{t_{i-1}}(\th_{t_{i-1}}) \\
A_4 &\coloneqq \t_{t_{i-1}} - \th_{t_{i-1}}
\end{aligned}. 
\end{equation*}
Taking the expectation on both sides, 
\begin{align*}
e_i = e_{i-1} + 2\sum_{k=1}^4\sum_{j=1}^3\int_{t_{i-1}}^{t_i}\E[\<A_k,  B_j\>] \d t. 
\end{align*}

Since $a_t(\t_t)$ is a martingale, conditioning on the information of both $\th_t$ and $\ah_t$ before $t_{i-1}$ and taking expectation,  
\begin{align*}
\int_{t_{i-1}}^{t_i}\E[\<A_1,  B_2\>] \d t&=\int_{t_{i-1}}^{t_i}\E[\<A_1,  B_3\>] \d t \\
&= \int_{t_{i-1}}^{t_i}\E[\<A_2,  B_1\>] \d t = \int_{t_{i-1}}^{t_i}\E[\<A_3,  B_1\>] \d t = \int_{t_{i-1}}^{t_i}\E[\<A_4,  B_1\>] \d t = 0. 
\end{align*}

For the term with integrand $\E[\<A_1,  B_1\>]$, we leverage the calculus of SL to obtain an upper bound:
\begin{align*}
\int_{t_{i-1}}^{t_i}\E[\<A_1,  B_1\>] \d t &= \int_{t_{i-1}}^{t_i}\int_{t_{i-1}}^{t}\E\left[\<a_s(\t_s)-a_{t_{i-1}}(\t_{t_{i-1}}), a_t(\t_t)-a_{t_{i-1}}(\t_{t_{i-1}})\>\right] \d s \d t\\
& = \int_{t_{i-1}}^{t_i}\int_{t_{i-1}}^{t}\E\left[\|a_s(\t_s)-a_{t_{i-1}}(\t_{t_{i-1}})\|_2^2\right] \d s \d t\quad\text{(since $a_t$ is a martingale)}\\
& = \int_{t_{i-1}}^{t_i}\int_{t_{i-1}}^{t}\int_{t_{i-1}}^s \E[\tr(\Sigma_r^2)] \d r \d s \d t \quad\text{(\ito's isometry)}\\
&\leq R^2\int_{t_{i-1}}^{t_i}\int_{t_{i-1}}^{t}\int_{t_{i-1}}^s \E[\tr(\Sigma_r)] \d r \d s \d t \quad\text{($\|\Sigma_r\|_2\leq R^2$ under $\supp(\mu_r)\subseteq\B_R(\R^d)$)}\\
&\leq R^2\int_{t_{i-1}}^{t_i}\int_{t_{i-1}}^{t}\int_{t_{i-1}}^s \frac{d}{r+1} \d r \d s \d t \quad\text{(\Cref{thm:01} and $\tr(\Sigma_0)\leq d$)}\\
&\leq \frac{R^2 d}{6(t_{i-1}+1)}(\Delta t_i)^3\quad\text{(\cref{ratio1})}\\
&\leq \frac{R^2 dh (\Delta t_i)^2}{6}. 
\end{align*}

The term with integrand $\E[\<A_3, B_3\>]$ is solely associated with the MSE of $\ah_{t_{i-1}}$ and can be bounded using the MSE condition in \eqref{variance-bdd}:
\begin{align*}
\int_{t_{i-1}}^{t_i}\E[\<A_3,  B_3\>] \d t &= \frac{1}{2}\E[\|a_{t_{i-1}}(\th_{t_{i-1}}) - \ah_{t_{i-1}}(\th_{t_{i-1}})\|_2^2] (\Delta t_i)^2\leq\frac{R^2 dh (\Delta t_i)^2}{6}.  
\end{align*}

For the remaining terms, we bound them using the following piecewise Lipschitz bound: 
\begin{align}\label{key}
\E[\|a_{t_{i-1}}(\t_{t_{i-1}}) - a_{t_{i-1}}(\th_{t_{i-1}})\|^2_2]\leq \begin{cases}
R^4e_{i-1}& t_{i-1}\leq s_0\\
\frac{C^2_0e_{i-1}}{t_{i-1}^2}& t_{i-1}>s_0
\end{cases},
\end{align}
where the first bound follows from the fact that $\rho_t(\t)(\d x)\propto\exp\left\{-\frac{t}{2}\|x\|_2^2+\<\t, x\>\right\}\mu(\d x)$ satisfies $\supp(\rho_t)\subseteq\B_R(\R^d)$ and the moment-generating property in \eqref{mgf}, and the second bound follows from Assumption~\ref{ass:1}. 

We first consider terms that do not involve $\ah_{t_{i-1}}$. Using \eqref{ratio1}, the Cauchy--Schwarz inequality, and that $s_0\geq 1$ (which implies $\Delta t_i = t_{i-1}h$), we have  
\begin{align*}
\int_{t_{i-1}}^{t_i}\E[\<A_2,  B_2\>] \d t &= \frac{1}{2}\E[\|a_{t_{i-1}}(\t_{t_{i-1}}) - a_{t_{i-1}}(\th_{t_{i-1}})\|^2_2](\Delta t_i)^2\\
&\leq \begin{cases}
\frac{R^4(\Delta t_i)^2 e_{i-1}}{2} & t_{i-1}\leq s_0\\
\frac{C^2_0h^2 e_{i-1}}{2}&t_{i-1}> s_0
\end{cases}\\
\int_{t_{i-1}}^{t_i}\E[\<A_4,  B_2\>] \d t &= \E\left[\<\t_{t_{i-1}} - \th_{t_{i-1}}, a_{t_{i-1}}(\t_{t_{i-1}})-a_{t_{i-1}}(\th_{t_{i-1}})\>\right]\Delta t_i\\
&\leq \Delta t_i \sqrt{e_{i-1}}\E[\|a_{t_{i-1}}(\t_{t_{i-1}}) - a_{t_{i-1}}(\th_{t_{i-1}})\|^2_2]^{\frac{1}{2}}\\
&\leq \begin{cases}
R^2\Delta t_i e_{i-1}& t_{i-1}\leq s_0\\
C_0h e_{i-1}& t_{i-1}> s_0
\end{cases}
\end{align*}

For terms involving $\ah_{t_{i-1}}$, since the construction of $\ah_{t_{i-1}}$ is independent of $\th_{t_{i-1}}$, we can bound the corresponding error terms using the bias of $\ah_{t_{i-1}}$ instead of the full MSE. By first conditioning on $\th_{t_{i-1}}$ and using the bias condition in \eqref{variance-bdd}, 
\begin{align*}
\int_{t_{i-1}}^{t_i}\E[\<A_3,  B_2\>] \d t &= \int_{t_{i-1}}^{t_i}\E[\<A_2,  B_3\>] \d t\\
 &=  \frac{1}{2}\E\left[\<a_{t_{i-1}}(\t_{t_{i-1}})-a_{t_{i-1}}(\th_{t_{i-1}}), \E\left[a_{t_{i-1}}(\th_{t_{i-1}}) - \ah_{t_{i-1}}(\th_{t_{i-1}})\mid\th_{t_{i-1}}\right]\>\right](\Delta t_i)^2\\
 &\leq \frac{1}{2}\E[\|a_{t_{i-1}}(\t_{t_{i-1}})-a_{t_{i-1}}(\th_{t_{i-1}})\|_2^2]^{\frac{1}{2}}\left\|\E\left[a_{t_{i-1}}(\th_{t_{i-1}}) - \ah_{t_{i-1}}(\th_{t_{i-1}})\mid\th_{t_{i-1}}\right]\right\|_2(\Delta t_i)^2\\
 &\leq \begin{cases}
 \frac{R^4(\Delta t_i)^2h^2\sqrt{e_{i-1}}}{48}& t_{i-1}\leq s_0\\
 \frac{C_0R^2\Delta t_i h^3\sqrt{e_{i-1}}}{48}& t_{i-1}> s_0
 \end{cases}\\
 &\stackrel{(*)}{\leq} \frac{R^2\Delta t_i h^2\sqrt{e_{i-1}}}{48},
 \end{align*}
 and 
 \begin{align*}
 \int_{t_{i-1}}^{t_i}\E[\<A_4,  B_3\>] \d t & = \E\left[\<\t_{t_{i-1}} - \th_{t_{i-1}}, \E\left[a_{t_{i-1}}(\th_{t_{i-1}}) - \ah_{t_{i-1}}(\th_{t_{i-1}})\mid\th_{t_{i-1}}\right]\>\right] \Delta t_i\\
 &\leq \E[\|\t_{t_{i-1}} - \th_{t_{i-1}}\|_2^2]^{\frac{1}{2}}\left\|\E\left[a_{t_{i-1}}(\th_{t_{i-1}}) - \ah_{t_{i-1}}(\th_{t_{i-1}})\mid\th_{t_{i-1}}\right]\right\|_2\Delta t_i\\
 &\leq \frac{R^2\Delta t_i h^2\sqrt{e_{i-1}}}{24}.  
\end{align*}
For step $(*)$, we used the following facts implied by our choice of $\e$ and $h$. For $t_{i-1}\leq s_0$,  
\begin{align*}
R^2\Delta t_i\leq R^2s_0h\leq R^2 s_0\cdot\frac{\e}{R\sqrt{d}}\leq 1, 
\end{align*}
and for $t_{i-1}>s_0$, 
\begin{align*}
C_0h\leq \frac{C_0\e}{R\sqrt{d}}\leq 1.
\end{align*} 

Combining the above estimates yields
\begin{align}\label{rec0}
e_i \leq \begin{cases}
\left(1+ R^2\Delta t_i\right)^2e_{i-1} +\frac{R^2\Delta t_i h^2\sqrt{e_{i-1}}}{6}+ \frac{2}{3}R^2 dh (\Delta t_i)^2& t_{i-1}\leq s_0\\
\left(1+ C_0h\right)^2e_{i-1} + \frac{R^2\Delta t_i h^2\sqrt{e_{i-1}}}{6}+\frac{2}{3}R^2 dh (\Delta t_i)^2& t_{i-1}> s_0
\end{cases}
\end{align}
Note that $e_0=0$ and the bound on $e_i$ is increasing in $i$. Assuming that $e_{i-1}\leq 2d^2$, then 
\begin{align*}
\frac{R^2\Delta t_i h^2\sqrt{e_{i-1}}}{6}\leq \frac{\sqrt{2}R^2\Delta t_i h^2d}{6}\leq \frac{R^2 dh (\Delta t_i)^2}{3}\quad\quad\text{(Since $\Delta t_i\geq h$)}. 
\end{align*}
In this case, the bound in \eqref{rec0} implies
\begin{align}\label{rec}
e_i \leq \begin{cases}
\left(1+ R^2\Delta t_i\right)^2e_{i-1} + R^2 dh (\Delta t_i)^2& t_{i-1}\leq s_0\\
\left(1+ C_0h\right)^2e_{i-1} + R^2 dh (\Delta t_i)^2& t_{i-1}> s_0
\end{cases},
\end{align}
which is easier to analyze. 

We show next that $e_i\leq 2d^2$ remains true up to $i = L$ via a bootstrap argument. In particular, we will verify that the right-hand side of \eqref{rec}, which is increasing in $i$, remains bounded by $2d^2$ up to $i=L$ (which implies for all $j\leq L$ by monotonicity). We claim this verification suffices to imply $e_i\leq 2d^2$ for all $i\leq L$. To see this, we argue by contradiction. Note that the right-hand side of \eqref{rec} remains a valid upper bound on $e_i$ up to some positive index $\ell >0$ since $e_0 = 0<2d^2$; this validity can be violated at $\ell$ only if $e_j>2d^2$ for some $j<\ell$. To apply the contradiction argument, assume $e_i>2d^2$ for some $i\leq L$. Then there exists $0<\ell\leq i$ such that $e_\ell>2d^2$ while $e_j\leq 2d^2$ for all $j<\ell$ (i.e., $j=\ell-1$). But the latter is a contradiction since $e_\ell$ should be bounded by the right-hand side of \eqref{rec}, and hence by $2d^2$ as ensured by the verification. 

Denoting the right-hand side of \eqref{rec} as $f_i$, we now use \eqref{rec} to obtain an upper bound on $f_L$. 
Let $i_0$ denote the smallest index such that $t_{i_0-1}< s_0\leq t_{i_0}<2s_0$ (since $s_0\geq 1$ and $h<1/2$). For $i<i_0$, $f_i\leq f_{i_0}$, which can be explicitly computed and bounded through the recursion:  
\begin{equation}\label{burnin}
\begin{aligned}
f_{i_0} &= \sum_{j=1}^{i_0} \left( \prod_{k=j+1}^{i_0} (1+R^2\Delta t_k)^2 \right) R^2 d h(\Delta t_j)^2\\
&\leq \left(\sum_{j=1}^{i_0}\exp\left\{2R^2 (t_{i_0}-t_j)\right\}\Delta t_j\right) R^2 ds_0 h^2\quad\quad\text{($\Delta t_j = \max\{1, t_{j-1}\} h\leq s_0h$)}\\
&\leq \left(\int_0^{t_{i_0}}\exp\left\{2R^2 (t_{i_0}-s)\right\} \d s\right) R^2 ds_0 h^2\\
&\leq \frac{1}{2}\exp\left\{2R^2t_{i_0}\right\}ds_0 h^2\\
&\leq \exp\left\{4R^2s_0\right\}ds_0 h^2\quad\quad\text{(since $t_{i_0}\leq 2s_0$)}. 
\end{aligned}
\end{equation}

For $i > i_0$, we switch to the regime $t_{i-1}> s_0$. Unrolling the recursion from step $i_0$ to the final step $L$ (where $t_L = T$) yields 
\begin{equation}\label{needu}
\begin{aligned}
f_L &= (1+C_0h)^{2(L-i_0)} f_{i_0} + R^2 d h^3 \sum_{k=i_0+1}^L (1+C_0h)^{2(L-k)} t^2_k\\
&\leq (1+C_0h)^{2(L-i_0)} f_{i_0} + R^2 d h^3 \sum_{k=i_0+1}^L (1+C_0h)^{2(L-k)} (1+h)^{2k}\\
&\leq (1+C_{0, +}h)^{2(L-i_0)} f_{i_0} + (L-i_0)R^2 d h^3 (1+C_{0, +}h)^{2L}\quad\text{($C_{0, +}=\max\{C_0, 1\}$)},
\end{aligned}
\end{equation}
where the first inequality follows from that
\begin{align*} 
t_k = (1+h)^{k-q}t_{q}\leq (1+h)^{k-q}t_{q}(1+h)\leq (1+h)^k,\quad\quad q\coloneqq \left\lceil\frac{1}{h}\right\rceil. 
\end{align*}

For convenience, we assume $C_0\geq 1$; the case $C_0<1$ is the same as the case $C_0=1$. Since $T=t_L=t_{i}(1+h)^{L-i}$ for $i\leq L$ with $t_i\geq 1$, setting $i = q$ yields $T\geq (1+h)^{L-q}$. Rearranging terms and noting $h\leq 1/2$ yields $(1+h)^L\leq (1+h)^{1/h}(1+h)T\leq 5T$. Consequently, 
\begin{align*}
(1+C_0h)^{2L}\leq (1+h)^{2C_0L}\leq (5T)^{2C_0}. 
\end{align*}
Meanwhile, setting $i=i_0$ yields 
\begin{align*}
L - i_0 = \frac{\log T - \log t_{i_0}}{\log (1+h)} \leq \frac{\log T - \log s_0}{\log (1+h)}\leq\frac{2\log (T/s_0)}{h}\quad\quad\text{(since $h<1/2$)}, 
\end{align*}
we have
\begin{align}
(1+C_0h)^{2(L-i_0)}\leq (1+h)^{2C_0(L-i_0)}\leq\left(\frac{T}{s_0}\right)^{2C_0}.\label{getit}
\end{align}
Putting these estimate into the bound on $f_L$ and plugging the estimate for $f_{i_0}$, 
\begin{align}
f_L \leq  T^{2C_0} h^2 d \left(s_0^{1-2C_0}\exp\left\{4R^2s_{0}\right\} + 2R^225^{C_0}\log T \right).\label{final524}
\end{align}
Under the choice of $h$ in \eqref{superh}, substituting $T = d/\e$ in to \eqref{final524}, $f_L \leq 2d^2$. 
Consequently, $e_L\leq 2d^2$. The proof is completed by applying the Lipschitz property under Assumption~\ref{ass:1} and the terminal variance bound in \cref{variance-bdd}: 
\begin{align*}
\E[\|a_T(\t_T)-\ah_T(\th_T)\|_2^2]&\leq 2(\E[\|a_T(\th_T)-\ah_T(\th_T)\|_2^2]+\E[\|a_T(\t_T)-a_T(\th_T)\|_2^2])\\
&\leq 2\e^2 + \frac{2C^2_0e_L}{T^2}\leq 2\e^2 + 4C^2_0\e^2\leq 6C^2_0\e^2.  
\end{align*}

\subsubsection{Proof of \Cref{lm:dsrate}}\label{pf:002}

Without loss of generality, we assume that $0<\E[\|\t_t-\iota_t\|_2^2]<\infty$. Since the expectation depends only on the joint law of $(\t_t, \iota_t)$, and by the law-equivalent form of $\t_t$ in \eqref{obs}, we have $\t_t \equiv tX + W_t$ for $X\sim\mu$ independent of $W_t$. Thus, we may consider a distributionally equivalent version of $(\t_t, \iota_t)$ on a different probability space where the first component is $tX + W_t$. With a mild abuse of notation, we continue to denote this process as $(\t_t, \iota_t)$. 

We now evaluate $\E[\|a_t(\t_t)-a_t(\iota_t)\|^2_2]$ restricted to the partitions associated with $\mathcal A_1$ and $\mathcal A_2$ defined as follows: 
\begin{align}
\mathcal A_1 = \{\|W_t\|_2\leq c_0t\},\quad\quad \mathcal A_2 = \{\|\t_t-\iota_t\|_2\leq c_0 t\}, 
\end{align}
where $c_0 = \gamma/8$. When $t$ is large, $\mathcal A^\complement_1$ is some rare event that can be bounded using a large-deviation type estimate: $\P(\mathcal A^\complement_1) \leq \exp\{-c_0^2t/4+d/2\}$. Since $\sup_{\t\in\R^n}\|a_t(\t)\|_2\leq R$, $\|a_{t}(\t_{t}) - a_{t}(\iota_t)\|^2_2\leq 4R^2$ and 
\begin{align}
\E[\|a_{t}(\t_{t}) - a_{t}(\iota_t)\|^2_2; \mathcal A_1^\complement]\leq 4R^2\exp\left\{-\frac{c_0^2t}{4}+\frac{d}{2}\right\}.\label{bdd:w1} 
\end{align}
On the other hand, by the definition of $\mathcal A_2$, 
\begin{align}
\E\left[\|a_{t}(\t_{t}) - a_{t}(\iota_t)\|^2_2; \mathcal A_1\cap\mathcal A^\complement_2\right]&\leq\E\left[4R^2; \mathcal A_1\cap\mathcal A^\complement_2\right]\leq \frac{4R^2}{c_0^2t^2}\E\left[\|\t_t-\iota_t\|^2_2; \mathcal A_1\cap\mathcal A^\complement_2\right].\label{bdd:final1} 
\end{align} 

It remains to consider the expectation on the event $\mathcal A_1\cap\mathcal A_2$. In this setting, for large $t$, $\t_t/t = X+W_t/t$ is near one of the points in $\supp(\mu)$, say, $x_i$, while $\iota_t/t$ is close to $\t_t/t$. This allows us to leverage the local Lipschitz constant of $a_t$ near $x_i$ to bound $\|a_{t}(\t_{t}) - a_{t}(\th_{t})\|^2_2$. 

To implement this idea, we use the total probability formula to write 
\begin{align}
&\E\left[\|a_{t}(\t_{t}) - a_{t}(\iota_t)\|^2_2; \mathcal A_1\cap\mathcal A_2\right]\nonumber\\
=&\ \sum_{i=1}^n\E\left[\|a_{t}(\t_{t}) - a_{t}(\iota_{t})\|^2_2\mid \mathcal A_1\cap\mathcal A_2\cap\{X=x_i\}\right]\P(\mathcal A_1\cap\mathcal A_2\cap\{X=x_i\}).\label{bdd:main}
\end{align}
Conditioning on $\mathcal A_1\cap\mathcal A_2\cap\{X=x_i\}$ and defining $\rho_t(\t)(\d x)\propto \exp\{-t\|x\|_2^2/2+\<\t, x\>\}\mu(\d x)$,  
by the fundamental theorem of calculus and the generating property discussed in \eqref{mgf},
\begin{align}\label{ftseju}
\|a_{t}(\t_{t}) - a_{t}(\iota_{t})\|_2 &= \left\|\left(\int_0^1 \nabla_\t a_t(\lambda\t_t+(1-\lambda)\iota_t) \d\lambda\right)(\t_t-\iota_t)\right\|_2\nonumber\\
& = \left\|\left(\int_0^1\mathrm{Cov}[\rho_t(\lambda\t_t+(1-\lambda)\iota_t)]\d\lambda\right)(\t_t-\iota_t)\right\|_2\nonumber\\
&\leq \int_0^1\left\|\mathrm{Cov}[\rho_t(\lambda\t_t+(1-\lambda)\iota_t)]\right\|_2\d\lambda\|\t_t-\iota_t\|_2\nonumber\\
&\leq \max_{\|\t/t-x_i\|_2\leq 2c_0}\|\mathrm{Cov}[\rho_t(\t)]\|_2\|\t_t-\iota_t\|_2. 
\end{align}
Under our choice of $c_0$, most of the mass of $\rho_t(\t)$ lives on $x_i$. Indeed, for any $\t = x_i t+\xi t$ with $\|\xi\|_2\leq 2c_0<\gamma/4\leq \|x_i-x_j\|_2/4$, and $j\neq i$, by the Cauchy--Schwarz inequality, 
\begin{align*}
\frac{\rho_t(\t)(\d x_j)}{\rho_t(\t)(\d x_i)} &= \frac{\exp\left\{-t\left(\frac{1}{2}\|x_j\|_2^2-\<x_i+\xi, x_j\>\right)\right\}}{\exp\left\{-t(\frac{1}{2}\|x_i\|_2^2-\<x_i+\xi, x_i\>)\right\}}= \exp\left\{-t\left(\frac{1}{2}\|x_i-x_j\|^2_2+\<\xi, x_i-x_j\>\right)\right\}\\
&\leq \exp\left\{-\frac{t}{4}\|x_i-x_j\|^2_2\right\}\implies \rho_t(\t)(\d x_j)\leq \exp\left\{-\frac{t}{4}\|x_i-x_j\|^2_2\right\}.  
\end{align*}
Using these estimates, we bound $\|\mathrm{Cov}[\rho_t(\t)]\|_2$ via a multiscale argument. Let $\mathcal S_k = \{x_j: k\gamma\leq \|x_i-x_j\|_2< (k+1)\gamma\}$ for $k\in\N$. Since $S_k$ is a $\gamma$-packing, $\{\B_{\gamma/2}(x_j)\}_{x_j\in \mathcal S_k}$ are almost disjoint. It follows from volume comparison that 
\begin{align*}
|\mathcal S_k|\left(\frac{\gamma}{2}\right)^d\leq \left[\left(k+\frac{3}{2}\right)\gamma\right]^d - \left[\left(k-\frac{1}{2}\right)\gamma\right]^d\implies |\mathcal S_k|\leq 2^{d+1} d\left(k+\frac{3}{2}\right)^{d-1}, 
\end{align*}
where we used that $(x+a)^d - (x-b)^d \le d (x+a)^{d-1} (a+b)$ for $a, b>0$ and $x>b$. 
Consequently, for $t$ satisfying $t\geq (4\log(8d\gamma^2t)+8d)/\gamma^2$ (which holds for all sufficiently large $t$), 
\begin{align*}
\|\mathrm{Cov}[\rho_t(\t)]\|_2&\leq \tr(\mathrm{Cov}[\rho_t(\t)])\leq \E_{Z\sim\rho_t(\t)}[\|Z-x_i\|^2_2]\\
&\leq \sum_{k=1}^\infty\sum_{j: x_j\in \mathcal S_k}\|x_j-x_i\|_2^2\exp\left\{-\frac{t}{4}\|x_i-x_j\|^2_2\right\}\\
&\leq \sum_{k=1}^\infty |\mathcal S_k| (k+1)^2\gamma^2\exp\left\{-\frac{\gamma^2kt}{4}\right\}\\
&\leq \sum_{k=1}^\infty 2^{d+1} d\left(k+\frac{3}{2}\right)^{d-1}\cdot (k+1)^2\gamma^2\exp\left\{-\frac{\gamma^2kt}{4}\right\}\\
&\leq 2^{d+4} d\left(\frac{5}{2}\right)^{d-1}\gamma^2\exp\left\{-\frac{\gamma^2t}{4}\right\}\\
&\leq \frac{1}{t},  
\end{align*}
where the penultimate step follows by noting that the series is bounded by a geometric series with rate $1/2$ with the same first term under our choice of $t$, and the last step follows by direct verification. Since the upper bound is independent of $i$, substituting this into \eqref{ftseju} yields 
\begin{align*}
\|a_{t}(\t_{t}) - a_{t}(\iota_{t})\|_2\leq\frac{\|\t_t-\iota_t\|_2}{t}, 
\end{align*}
which combined with \eqref{bdd:main} yields
\begin{align}
\E\left[\|a_{t}(\t_{t}) - a_{t}(\iota_t)\|^2_2; \mathcal A_1\cap\mathcal A_2\right]\leq\frac{\E\left[\|\t_{t} - \iota_t\|^2_2; \mathcal A_1\cap\mathcal A_2\right]}{t^2}. \label{bdd:final2}
\end{align}
Adding \eqref{bdd:w1}, \eqref{bdd:final1}, and \eqref{bdd:final2} together yields that, for $t\geq \max\{(4\log(8d\gamma^2t)+8d)/\gamma^2, 1\}$, 
\begin{align*}
\E\left[\|a_{t}(\t_{t}) - a_{t}(\iota_{t})\|^2_2\right]&\leq 4R^2\exp\left\{-\frac{c_0^2t}{4}+\frac{d}{2}\right\} + \frac{4R^2}{c_0^2t^2}\E\left[\|\t_t-\iota_t\|^2_2; \mathcal A_1\cap\mathcal A^\complement_2\right] + \frac{\E\left[\|\t_{t} - \iota_t\|^2_2; \mathcal A_1\cap\mathcal A_2\right]}{t^2}\\
&\leq 4R^2\exp\left\{-\frac{\gamma^2 t}{256}+\frac{d}{2}\right\} + \frac{1}{t^2}\left(\frac{256R^2}{\gamma^2}+1\right)\E\left[\|\t_t-\iota_t\|^2_2\right] \\
&\leq 4R^2\exp\left\{-\frac{\gamma^2 t}{256}+\frac{d}{2}\right\} + \frac{260R^2}{\gamma^2t^2}\E\left[\|\t_t-\iota_t\|^2_2\right] \\
&\leq\max\left\{8R^2\exp\left\{-\frac{\gamma^2 t}{256}+\frac{d}{2}\right\}, \frac{520R^2}{\gamma^2t^2}\E\left[\|\t_t-\iota_t\|^2_2\right]\right\},  
\end{align*}
where we have used that $\gamma\leq 2R$ for the third step. The proof is complete.

\subsubsection{Proof of \Cref{thm:eperror}}\label{pf:003}

Unlike \Cref{thm:logconcave-0}, the bound \eqref{mixedbdd} is a maximum of a localized term mirroring Assumption~\ref{ass:1} and a large-deviation term with no analogue there. The argument splits on which term controls the dynamics. When the localized term dominates throughout (scenario (i)), the analysis reduces to that of \Cref{thm:logconcave-0}. Otherwise (scenario (ii)), we track the last index $i_1$ at which the large-deviation term dominates; since this term decays in $t$, its dominance forces $e_{i_1}$ to be exponentially small, after which the localized recursion controls all subsequent accumulation.

Since $T=d/\e$ and $\gamma\leq 1$, by changing the $t$ term in the self-referential bound on $s_0$ in \Cref{lm:dsrate} by its upper bound $T = d/\e$, we can choose $s_0$ such that 
\begin{align}
s_0 = C_1\log (\frac{d}{\e})+\frac{128d}{\gamma^2}\geq\frac{4}{\gamma^2}\log \left(\frac{8d^2\gamma^2}{\e}\right)+\frac{8d}{\gamma^2},\quad C_1 = \frac{768C_0}{\gamma^2}. \label{mys0}
\end{align}

After $s_0$, two possible scenarios can happen: (i) the second term in the bound in \eqref{mixedbdd} dominates throughout the iteration to the end, or (ii) the first bound in \eqref{mixedbdd} becomes dominant for at least one step. 

The scenario (i) is identical to the setting in the proof of \Cref{thm:logconcave-0} with $C_0 = 520R^2/\gamma^2>2$. Let $i_0$ be the smallest index such that $t_{i_0}>s_0$ and choose
\begin{align*}
h = \frac{1}{C_0\max\{R, 1\}}\frac{\e^{2R^2C_1+C_0}}{d^{2R^2C_1+C_0-\frac{1}{2}}}\exp\left\{-\frac{256R^2d}{\gamma^2}\right\}<\frac{1}{2}.  
\end{align*}
By a similar bound as \eqref{final524}, 
\begin{equation}\label{111321}
\begin{aligned}
e_L &\leq  T^{2C_0} h^2 d \left(s_0^{1-2C_0}\exp\left\{4R^2s_{0}\right\} + 2R^225^{C_0}\log T \right)\\
&\leq T^{2C_0} h^2 d \left(\exp\left\{4R^2s_{0}\right\} + 2R^225^{C_0}\log T  \right)\\
&= T^2\left(\frac{d}{\e}\right)^{2C_0-2}\frac{1}{C^2_0\max\{R^2, 1\}}\frac{\e^{4R^2C_1+2C_0}}{d^{4R^2C_1+2C_0-1}}\exp\left\{-\frac{512R^2d}{\gamma^2}\right\} d \\
&\ \ \ \ \cdot\left[\left(\frac{d}{\e}\right)^{4R^2C_1}\exp\left\{\frac{512R^2d}{\gamma^2}\right\} + 2R^225^{C_0}\log\left(\frac{d}{\e}\right)\right]\\
&\leq \frac{2T^2\e^2}{C^2_0}, 
\end{aligned}
\end{equation}
which implies that 
\begin{align}
\E[\|\ah_T-a_T\|_2^2]\leq\frac{C^2_0e_L}{T^2}\leq 2\e^2. \label{Cvsye1}
\end{align}

For scenario (ii), since the first term in \eqref{mixedbdd} dominates at least once after $s_0$, let $i_1$ denote the last index with $t_{i_1}>s_0$ such that the first term in \eqref{mixedbdd} dominates. After this point, the second term in \eqref{mixedbdd} dominates. In this case, the bound \eqref{needu} holds with $i_0$ replaced by $i_1$:
\begin{equation}\label{jusiq}
\begin{aligned}
e_L \leq (1+C_0h)^{2(L-i_1)} e_{i_1} + (L-i_1)R^2 d h^3 (1+C_0h)^{2L}.
\end{aligned}
\end{equation}
Under the same choice of $h$ and assuming $\e$ is sufficiently small, the second term on the right-hand side of \eqref{jusiq} can be bounded as 
\begin{align*}
(L-i_1)R^2 d h^3 (1+C_0h)^{2L}&\leq 2R^225^{C_0}dh^2T^{2C_0}\\
&\leq 2R^225^{C_0}d\left(\frac{1}{C^2_0R^2}\frac{\e^{4R^2C_1+2C_0}}{d^{4R^2C_1+2C_0-1}}\right)\left(\frac{d}{\e}\right)^{2C_0}\\
&\leq 2\e^2. 
\end{align*}
For the first term, by definition, 
\begin{align}
8R^2\exp\left\{-\frac{\gamma^2t_{i_1}}{256}+\frac{d}{2}\right\}\geq \frac{520R^2}{\gamma^2t_{i_1}^2}\E\left[\|\t_{t_{i_1}}-\iota_{t_{i_1}}\|^2_2\right] \geq \frac{520R^2e_{i_1}}{t_{i_1}^2}\quad\quad\text{(since $\gamma\leq 1$)}. \label{gem}
\end{align}
Multiplying on both sides of \eqref{gem} by $(1+C_0h)^{2(L-i_1)}$, which satisfies the following bound similar to \eqref{getit}:
\begin{align*}
(1+C_0h)^{2(L-i_1)}\leq (1+h)^{2C_0(L-i_1)}\leq \left(\frac{T}{t_{i_1}}\right)^{2C_0} = \left(\frac{d}{\e t_{i_1}}\right)^{2C_0}, 
\end{align*}
and simplifying terms yields
\begin{align*}
(1+C_0h)^{2(L-i_1)} e_{i_1}&\leq(1+C_0h)^{2(L-i_1)} t_{i_1}^2\exp\left\{-\frac{\gamma^2 t_{i_1}}{256}+\frac{d}{2}\right\}\\
&\leq\left(\frac{d}{\e t_{i_1}}\right)^{2C_0}t_{i_1}^2\exp\left\{-\frac{\gamma^2 s_0}{256}+\frac{d}{2}\right\}\quad\text{(since $t_{i_1}>s_0$)}\\
&\leq \exp\left\{-\frac{\gamma^2 s_0}{256}+\frac{d}{2}+2C_0\log (\frac{d}{\e})\right\}\\
&\leq \exp\left\{-C_0\log (\frac{d}{\e})\right\}\quad\text{(\cref{mys0})}\\
&\leq \e^2. \quad\text{(since $C_0>2$)}
\end{align*}
Consequently, $e_L\leq 3\e^2$. Since $a_t$ is $R^2$-Lipschitz,  
\begin{align}
\E[\|\ah_T-a_T\|_2^2]\leq R^4e_L\leq 3R^4\e^2.\label{Cvsye2} 
\end{align}
The proof is completed by combining the estimates in \eqref{Cvsye1} and \eqref{Cvsye2}.

\subsubsection{Proof of \Cref{thm:1/2-gaussian}}\label{euler:ode}

\paragraph{Regularization error.}
We first analyze the perturbation error associated with regularization. Fixing $k$ and applying the Cauchy--Schwarz inequality, 
\begin{equation}\label{break1}
\begin{aligned}
&\int_0^T \left(f(\sigma_{k, t}, \sigma_{k, t}; 0)- f(\sigma^{(\delta)}_{k, t}, \sigma^{(\delta)}_{k, t}; \delta)\right)^2 \d t\\
\leq&\  2\int_0^T \left(f(\sigma_{k, t}, \sigma_{k, t}; 0)- f(\sigma^{(\delta)}_{k, t}, \sigma^{(\delta)}_{k, t}; 0)\right)^2 \d t+2\int_0^T \left(f(\sigma^{(\delta)}_{k, t}, \sigma^{(\delta)}_{k, t}; 0)- f(\sigma^{(\delta)}_{k, t}, \sigma^{(\delta)}_{k, t}; \delta)\right)^2 \d t.
\end{aligned}
\end{equation}
Since 
\begin{align*}
- \sigma^{(\delta)}_{k, t}\leq-\frac{(\sigma^{(\delta)}_{k, t})^2}{\sigma^{(\delta)}_{k, t}+\delta}\leq - \sigma^{(\delta)}_{k, t} +\delta, 
\end{align*}
substituting this into \eqref{niceodes} and applying Gr\"{o}nwall's inequality (similar to the proof of \Cref{thm:reg}) yields that 
\begin{align}
\sigma_k e^{-t}=\sigma_{k, t}\leq\sigma^{(\delta)}_{k, t}\leq \sigma_k e^{-t} + \delta (1-e^{-t}). \label{useful1}
\end{align}
This gives an upper bound on the first term in \eqref{break1}: 
\begin{align*}
\int_0^T \left(f(\sigma_{k, t}, \sigma_{k, t}; 0)- f(\sigma^{(\delta)}_{k, t}, \sigma^{(\delta)}_{k, t}; 0)\right)^2 \d t&\leq\int_0^T (\sqrt{\sigma_k e^{-t} + \delta (1-e^{-t})}-\sqrt{\sigma_k e^{-t}})^2 \d t\\
&\leq \int_0^T \delta(1-e^{-t}) \d t\\
&\leq \delta T.  
\end{align*}
A matching bound on the second term in \eqref{break1} can be obtained similarly: 
\begin{align*}
\int_0^T \left(f(\sigma^{(\delta)}_{k, t}, \sigma^{(\delta)}_{k, t}; 0)- f(\sigma^{(\delta)}_{k, t}, \sigma^{(\delta)}_{k, t}; \delta)\right)^2 \d t &= \int_0^T \left(\sqrt{\sigma^{(\delta)}_{k, t}}-\frac{\sigma^{(\delta)}_{k, t}}{\sqrt{\sigma^{(\delta)}_{k, t}+\delta}}\right)^2 \d t\\
&\leq \int_0^T \left(\sqrt{\sigma^{(\delta)}_{k, t}}-\sqrt{\max\{\sigma^{(\delta)}_{k, t}-\delta, 0\}}\right)^2 \d t\\
&\leq \delta T. 
\end{align*}
Substituting both estimates into \eqref{break1} and summing over $k$ yields 
\begin{align*}
\sum_{k=1}^d\int_0^T \left(f(\sigma_{k, t}, \sigma_{k, t}; 0)- f(\sigma^{(\delta)}_{k, t}, \sigma^{(\delta)}_{k, t}; \delta)\right)^2 \d t\leq 4Td\delta. 
\end{align*}

\paragraph{Discretization error.}
We next analyze the error associated with discretization. Fixing $k$ and considering the interval $(t_{i-1}, t_i]$, we apply the Cauchy--Schwarz inequality to obtain 
\begin{equation}\label{heide}
\begin{aligned}
&\int_{t_{i-1}}^{t_i}\left(f(\widehat{\sigma}^{(\delta)}_{k, t_i}, \widehat{\sigma}^{(\delta)}_{k, t_{i-1}}; \delta)- f(\sigma^{(\delta)}_{k, t}, \sigma^{(\delta)}_{k, t}; \delta)\right)^2 \d t\\
\leq&\ 3\int_{t_{i-1}}^{t_i} \left(f(\widehat{\sigma}^{(\delta)}_{k, t_i}, \widehat{\sigma}^{(\delta)}_{k, t_{i-1}}; \delta)- f(\widehat{\sigma}^{(\delta)}_{k, t_{i-1}}, \widehat{\sigma}^{(\delta)}_{k, t_{i-1}}; \delta)\right)^2\d t + 3\int_{t_{i-1}}^{t_i} \left(f(\sigma^{(\delta)}_{k, t_{i-1}}, \sigma^{(\delta)}_{k, t_{i-1}}; \delta)-f(\sigma^{(\delta)}_{k, t}, \sigma^{(\delta)}_{k, t}; \delta)\right)^2\d t\\
& + 3h\left(f(\widehat{\sigma}^{(\delta)}_{k, t_{i-1}}, \widehat{\sigma}^{(\delta)}_{k, t_{i-1}}; \delta)-f(\sigma^{(\delta)}_{k, t_{i-1}}, \sigma^{(\delta)}_{k, t_{i-1}}; \delta)\right)^2. 
\end{aligned}
\end{equation}
The first two terms concern the dynamics of $\widehat{\sigma}^{(\delta)}_{k, t}$ and $\sigma^{(\delta)}_{k, t}$ on $(t_{i-1}, t_i]$, respectively. The last term is related to the numerical approximation of $\sigma^{(\delta)}_{k, t_{i-1}}$ by $\widehat{\sigma}^{(\delta)}_{k, t_{i-1}}$. In the following, we bound each of these terms separately.  

For the first term, it can be seen from the ODEs in \eqref{niceodes} that $\widehat{\sigma}^{(\delta)}_{k, t}$ is decreasing on $(t_{i-1}, t_i]$, and its absolute derivative is bounded by $(\widehat{\sigma}^{(\delta)}_{k, t_{i-1}})^2/(\widehat{\sigma}^{(\delta)}_{k, t_{i-1}}+\delta)$. Consequently, 
\begin{align*}
(\widehat{\sigma}^{(\delta)}_{k, t_i}-\widehat{\sigma}^{(\delta)}_{k, t_{i-1}})^2\leq \frac{(\widehat{\sigma}^{(\delta)}_{k, t_{i-1}})^4h^2}{(\widehat{\sigma}^{(\delta)}_{k, t_{i-1}}+\delta)^2}\implies \left(f(\widehat{\sigma}^{(\delta)}_{k, t_i}, \widehat{\sigma}^{(\delta)}_{k, t_{i-1}}; \delta)- f(\widehat{\sigma}^{(\delta)}_{k, t_{i-1}}, \widehat{\sigma}^{(\delta)}_{k, t_{i-1}}; \delta)\right)^2\leq \widehat{\sigma}^{(\delta)}_{k, t_{i-1}}h^2. 
\end{align*}
Substituting this into the first term yields
\begin{align}
\int_{t_{i-1}}^{t_i} \left(f(\widehat{\sigma}^{(\delta)}_{k, t_i}, \widehat{\sigma}^{(\delta)}_{k, t_{i-1}}; \delta)- f(\widehat{\sigma}^{(\delta)}_{k, t_{i-1}}, \widehat{\sigma}^{(\delta)}_{k, t_{i-1}}; \delta)\right)^2\d t\leq\widehat{\sigma}^{(\delta)}_{k, t_{i-1}}h^3.\label{tme1}
\end{align}

For the second term, since $g(x; \delta)\coloneqq f(x, x; \delta)$ is concave on $[0, \infty)$ and $\sigma^{(\delta)}_{k, t}$ is decreasing in $t$, 
\begin{align*}
\left(f(\sigma^{(\delta)}_{k, t_{i-1}}, \sigma^{(\delta)}_{k, t_{i-1}}; \delta)-f(\sigma^{(\delta)}_{k, t}, \sigma^{(\delta)}_{k, t}; \delta)\right)^2&\leq |g'(\sigma^{(\delta)}_{k, t_i}; \delta)|^2(\sigma^{(\delta)}_{k, t_{i-1}}-\sigma^{(\delta)}_{k, t_i})^2\\
& = \frac{(\sigma^{(\delta)}_{k, t_i}+2\delta)^2(\sigma^{(\delta)}_{k, t_{i-1}}-\sigma^{(\delta)}_{k, t_i})^2}{4(\sigma^{(\delta)}_{k, t_i}+\delta)^3}. 
\end{align*}
Since the absolute derivative of $\sigma^{(\delta)}_{k, t}$ on $(t_{i-1}, t_i]$ is bounded by $(\sigma^{(\delta)}_{k, t_{i-1}})^2/(\sigma^{(\delta)}_{k, t_{i-1}}+\delta)$ and $h\leq 1/2$, 
\begin{align}
\sigma^{(\delta)}_{k, t_{i}}\geq \sigma^{(\delta)}_{k, t_{i-1}} - \frac{(\sigma^{(\delta)}_{k, t_{i-1}})^2h}{\sigma^{(\delta)}_{k, t_{i-1}}+\delta}\geq (1-h)\sigma^{(\delta)}_{k, t_{i-1}}\implies \sigma^{(\delta)}_{k, t_{i-1}}\leq (1+2h)\sigma^{(\delta)}_{k, t_{i}}. \label{jiajia1}
\end{align} 
Substituting this into the above bound yields 
\begin{align*}
\left(f(\sigma^{(\delta)}_{k, t_{i-1}}, \sigma^{(\delta)}_{k, t_{i-1}}; \delta)-f(\sigma^{(\delta)}_{k, t}, \sigma^{(\delta)}_{k, t}; \delta)\right)^2\leq  \frac{(\sigma^{(\delta)}_{k, t_i}+2\delta)^2(\sigma^{(\delta)}_{k, t_i})^2h^2}{(\sigma^{(\delta)}_{k, t_i}+\delta)^3}\leq 2(\sigma^{(\delta)}_{k, t_i}+2\delta)h^2,
\end{align*}
which implies 
\begin{align}
\int_{t_{i-1}}^{t_i} \left(f(\sigma^{(\delta)}_{k, t_{i-1}}, \sigma^{(\delta)}_{k, t_{i-1}}; \delta)-f(\sigma^{(\delta)}_{k, t}, \sigma^{(\delta)}_{k, t}; \delta)\right)^2\d t\leq 2(\sigma^{(\delta)}_{k, t_i}+2\delta)h^3.\label{tme2}
\end{align}

To analyze the last term, we move to the inverse spectrum domain by letting $y_t=1/\sigma^{(\delta)}_{k, t}$ and $z_t = 1/\widehat{\sigma}^{(\delta)}_{k, t}$. An application of the chain rule using \eqref{niceodes} shows that the reciprocal function $y_t$ satisfies the following ODE:
\begin{align*}
\frac{\d y_t}{\d t} = k(y_t),\quad\quad k(y_t) \coloneqq \frac{y_t}{1+\delta y_t}, 
\end{align*}
and $z_t$ corresponds to the forward Euler scheme for solving $y_t$ along the uniform time discretization $0=t_0<t_1<\cdots <t_L = T = Lh$: for $t\in (t_{i-1}, t_i]$, 
\begin{align*}
\frac{\d z_t}{\d t} = k(z_{t_{i-1}}).  
\end{align*}
A direct computation shows that $|k'(y)| = 1/(1+\delta y)^2\leq 1$. Consequently, 
\begin{equation}\label{jghsy1a}
\begin{aligned}
y_{t_i}-z_{t_i} &= y_{t_{i-1}}-z_{t_{i-1}} + (k(y_{t_{i-1}})-k(z_{t_{i-1}})) h + \int_{t_{i-1}}^{t_i} k(y_t)-k(y_{t_{i-1}}) \d t\\
&\leq y_{t_{i-1}}-z_{t_{i-1}} + (k(y_{t_{i-1}})-k(z_{t_{i-1}})) h + h(k(y_{t_i}) - k(y_{t_{i-1}})), 
\end{aligned}
\end{equation}
where the last step follows from the fact that $k(y)$ is increasing in $y$ and $y_t$ is increasing in $t$. Here we could have used a further bound $(k(y_{t_{i-1}})-k(z_{t_{i-1}})) h\leq y_{t_{i-1}}-z_{t_{i-1}}$. However, such a bound ignores the fact that $k'(y)\to 0$ as $y\to\infty$, which turns out to be crucial to obtain the desired result. 

Alternatively, we use the following inequality. For $y\geq z\geq 0$, 
\begin{align*}
k(y)-k(z) = \frac{y-z}{(1+\delta y)(1+\delta z)}\leq\frac{k(y)}{y}(y-z). 
\end{align*}
Substituting $y = y_{t_{i-1}}$ and $z = z_{t_{i-1}}$ into \eqref{jghsy1a} yields 
\begin{align*}
y_{t_i}-z_{t_i}\leq y_{t_{i-1}}-z_{t_{i-1}} + \frac{k(y_{t_{i-1}})(y_{t_{i-1}}-z_{t_{i-1}})}{y_{t_{i-1}}} h + h(k(y_{t_i}) - k(y_{t_{i-1}})).
\end{align*}
Dividing both sides by $y_{t_i}$ and defining the relative error $w_i = (y_{t_i}-z_{t_i})/y_{t_i}$, and by the convexity of $y_t$, 
\begin{align*}
w_i\leq \frac{y_{t_{i-1}}+hk(y_{t_{i-1}})}{y_{t_i}}w_{i-1} + \frac{h(k(y_{t_i}) - k(y_{t_{i-1}}))}{y_{t_{i}}}\leq w_{i-1} + \frac{h(k(y_{t_i}) - k(y_{t_{i-1}}))}{y_{t_{i}}}.  
\end{align*}
In other words, the amplification factor for the recursion bound on the relative error $w_i$ is non-expansive, which manifests the localization property of the underlying dynamics. Since $w_0=0$, summing over $i$ and using the fundamental theorem of calculus, 
\begin{align*}
w_i\leq \sum_{j=1}^{i}\frac{h(k(y_{t_j}) - k(y_{t_{j-1}}))}{y_{t_{j}}} &= h\sum_{j=1}^{i}\frac{1}{y_{t_j}}\int_{t_{j-1}}^{t_j}\frac{\mathrm{d}}{\d s}k(y_s) \d s\\
& = h\sum_{j=1}^{i}\frac{1}{y_{t_j}}\int_{t_{j-1}}^{t_j}k'(y_s)k(y_s) \d s\\
&\leq h\sum_{j=1}^{i}\frac{1}{y_{t_j}}\int_{t_{j-1}}^{t_j}k(y_s) \d s\quad\quad\text{(since $|k'(y)|\leq 1$)}\\
&\leq h^2 \sum_{j=1}^{i}\frac{k(y_{t_j})}{y_{t_j}}\quad\quad\text{(since $k(y)$ is increasing)}\\
&\leq ih^2 = ht_i. 
\end{align*}
Consequently, 
\begin{align}
\frac{y_{t_i}-z_{t_i}}{y_{t_i}} =  \frac{\widehat{\sigma}^{(\delta)}_{k, t_i}-\sigma^{(\delta)}_{k, t_i}}{\widehat{\sigma}^{(\delta)}_{k, t_i}}\leq ht_i, \label{hrtsd}
\end{align}
which implies
\begin{align}
\frac{\sigma^{(\delta)}_{k, t_i}}{\widehat{\sigma}^{(\delta)}_{k, t_i}}\geq 1-ht_i\geq 1- hT\geq\frac{1}{2}\quad\quad\text{(since $hT\leq\frac{1}{2}$)}. \label{needuu}
\end{align}
Moreover, squaring both sides of \eqref{hrtsd} and reorganizing terms, 
\begin{align*}
(\widehat{\sigma}^{(\delta)}_{k, t_i}-\sigma^{(\delta)}_{k, t_i})^2\leq h^2t_i^2(\widehat{\sigma}^{(\delta)}_{k, t_i})^2. 
\end{align*}
Combining this with the concavity of $g$ yields the following bound on the last term:
\begin{equation}\label{tme3}
\begin{aligned}
& h\left(f(\widehat{\sigma}^{(\delta)}_{k, t_{i-1}}, \widehat{\sigma}^{(\delta)}_{k, t_{i-1}}; \delta)-f(\sigma^{(\delta)}_{k, t_{i-1}}, \sigma^{(\delta)}_{k, t_{i-1}}; \delta)\right)^2\\
\leq &\ h|g'(\sigma^{(\delta)}_{k, t_{i-1}}; \delta)|^2(\widehat{\sigma}^{(\delta)}_{k, t_{i-1}}-\sigma^{(\delta)}_{k, t_{i-1}})^2\\
\leq&\ \frac{h(\sigma^{(\delta)}_{k, t_{i-1}}+2\delta)^2}{4(\sigma^{(\delta)}_{k, t_{i-1}}+\delta)^3}\cdot h^2t_{i-1}^2(\widehat{\sigma}^{(\delta)}_{k, t_{i-1}})^2\\
\leq&\ 4\sigma^{(\delta)}_{k, t_{i-1}}t_{i-1}^2h^3\quad\quad\text{(\cref{needuu})}\\
\leq&\ 4 (\sigma_k e^{-t_{i-1}} + \delta (1-e^{-t_{i-1}}))t_{i-1}^2h^3\quad\quad\text{(\cref{useful1})}\\
\leq&\ 8 \max\left\{1, \sigma_k\right\} e^{-t_{i-1}}t_{i-1}^2h^3\quad\quad\text{(since $\delta\leq e^{-t_{i-1}}$)} 
\end{aligned}
\end{equation}

Putting \eqref{tme1}, \eqref{tme2}, and \eqref{tme3} into \eqref{heide}, the per-step contribution for each $k$ is bounded by
\begin{align*}
\int_{t_{i-1}}^{t_i}\left(f(\widehat{\sigma}^{(\delta)}_{k, t_i}, \widehat{\sigma}^{(\delta)}_{k, t_{i-1}}; \delta)- f(\sigma^{(\delta)}_{k, t}, \sigma^{(\delta)}_{k, t}; \delta)\right)^2 \d t
\leq 3\widehat{\sigma}^{(\delta)}_{k, t_{i-1}}h^3 + 6(\sigma^{(\delta)}_{k, t_i}+2\delta)h^3 + 24\max\{1, \sigma_k\}e^{-t_{i-1}}t_{i-1}^2h^3.
\end{align*}
Using \eqref{useful1}, \eqref{needuu}, and $\delta\leq e^{-T}\leq e^{-t}$, each of $\widehat{\sigma}^{(\delta)}_{k, t_{i-1}}$, $\sigma^{(\delta)}_{k, t_i}$, and $\delta$ is bounded by $2\max\{1,\sigma_k\}e^{-t_{i-1}}$, so the three terms combine into
\begin{align*}
\int_{t_{i-1}}^{t_i}\left(f(\widehat{\sigma}^{(\delta)}_{k, t_i}, \widehat{\sigma}^{(\delta)}_{k, t_{i-1}}; \delta)- f(\sigma^{(\delta)}_{k, t}, \sigma^{(\delta)}_{k, t}; \delta)\right)^2 \d t
\lesssim \max\{1, \sigma_k\} (t_{i-1}+1)^2 e^{-t_{i-1}} h^3.
\end{align*}
Summing over $i$ and recognizing the result as a Riemann sum with mesh $h$, 
\begin{align*}
\sum_{i=1}^L \max\{1, \sigma_k\} (t_{i-1}+1)^2 e^{-t_{i-1}} h^3
&= \max\{1, \sigma_k\} h^2\sum_{i=1}^L (t_{i-1}+1)^2 e^{-t_{i-1}} h\\
&\lesssim \max\{1, \sigma_k\} h^2\int_0^\infty (t+1)^2 e^{-t} \d t\\
&\lesssim \max\{1, \sigma_k\} h^2. 
\end{align*}
Summing over $k$ and using $\sum_{k=1}^d \max\{1, \sigma_k\}\leq d + \tr(\Sigma)$ finishes the proof: 
\begin{align*}
\sum_{i=1}^L\int_{t_{i-1}}^{t_i}\sum_{k=1}^d\left(f(\widehat{\sigma}^{(\delta)}_{k, t_i}, \widehat{\sigma}^{(\delta)}_{k, t_{i-1}}; \delta)- f(\sigma^{(\delta)}_{k, t}, \sigma^{(\delta)}_{k, t}; \delta)\right)^2 \d t
\lesssim (d + \tr(\Sigma)) h^2 \lesssim \max\{d, \tr(\Sigma)\} h^2. 
\end{align*}


\section{Applications}\label{sec:apps}

In this section, we employ joint SL as a computational tool for distributional data analysis. We begin by applying joint Eldan's $\alpha$-distance as a computationally efficient surrogate for the $W_2$ distance over a large collection of distributions. We then discuss how to use it to obtain an approximate solution to the Wasserstein barycenter problem.

\subsection{Surrogate for the $W_2$-distance over large cohorts}

In modern data science, data is routinely modeled by probability distributions. Analyzing such data requires comparing probability distributions in a geometrically meaningful way. A common first step is to compute pairwise distances between distributions under a suitable metric. Among the most widely used metrics are those from optimal transport, in particular $W_2$.

One notorious aspect of $W_2$ is its computational bottleneck when scaling to large datasets, either in terms of support size or cohort size. Computing pairwise $W_2$ distances for $m$ distributions requires $\mathcal O(m^2)$ evaluations. Since the $W_2$ space admits no global isometric Euclidean embedding due to its nonnegative curvature, this quadratic cost cannot be circumvented without distortion by embedding the distributions into a Euclidean space and replacing optimal transport solves with fast $\ell_2$ comparisons. Moreover, given two probability measures $\mu_1$ and $\mu_2$, each with support size $n$, computing $W_2(\mu_1, \mu_2)$ exactly requires solving a linear programming problem with computational complexity roughly $\mathcal O(n^3)$. When either $m$ or $n$ is large, computing pairwise $W_2$ distances has complexity $\mathcal O(m^2n^3)$, which is computationally prohibitive.

To address these challenges, various approximation methods have been proposed. In particular, to reduce the computational cost of evaluating $W_2$, two prominent approaches are entropic regularization \cite{cuturi2013sinkhorn} and sliced Wasserstein distances \cite{rabin2011wasserstein, bonnotte2013unidimensional}. The former reduces the cubic dependence on the support size to quadratic by smoothing the optimal transport plan via an entropy term that biases toward the product measure. The latter is defined as an average of one-dimensional $W_2$ distances computed over random projections, which can be computed using MC simulation and yields substantial computational savings. However, both approaches can significantly distort the geometry induced by $W_2$ and do not address the quadratic scaling in cohort size.

To mitigate the dependence on cohort size, a common strategy is to linearize the $W_2$ geometry around a reference measure, yielding the so-called linearized $W_2$ distance \cite{wang2013linear}. In this framework, measures are embedded into the formal tangent space of the reference measure via their respective optimal transport maps. This reduces the computational burden from solving $\mathcal O(m^2)$ optimal transport problems to computing just $\mathcal O(m)$ embeddings, after which pairwise distances can be evaluated via standard $\ell_2$ comparisons. However, different choices of the reference measure can lead to different approximations. In practice, it is often necessary to combine multiple strategies to obtain an efficient pipeline for approximately computing pairwise $W_2$ distances across large cohorts. More recently, alternative approaches based on learning transport maps as black-box models have also been explored \cite{haviv2024wasserstein}.

Eldan's $\alpha$-distance provides an alternative route to alleviating the computational burden of exact $W_2$ computation across multiple distributions. Throughout, we treat the ambient dimension $d$ as fixed and focus on the scaling in the cohort size $m$, the support size $n$, and the precision $\e$, the regime of primary interest for large collections of distributions.

Because the joint Eldan's $\alpha$-scheme yields a global coupling, computing its induced distance via the truncated MC estimator of \Cref{sec:newcomp} with precision $\e$ requires simulating $\mathcal O(\e^{-2})$ independent trajectories per measure, all driven by the same underlying Brownian motion. Since the $m$ measures are embedded independently, the scheme is parallelizable and its cost scales linearly in $m$. From this perspective, Eldan's $\alpha$-distance is similar to other embedding-based metrics such as the linearized $W_2$ distance; the methods differ, however, in their dependence on the support size $n$, as discussed below.

For each measure, the cost of a single trajectory is proportional to the number of discretization steps $L$. Under the assumptions of \Cref{thm:logconcave-0} (log-concave measures, $\alpha=0$) and \Cref{thm:1/2-gaussian} (Gaussian measures, $\alpha=1/2$), $L$ can be taken as small as $\mathcal O(\sqrt{d}\e^{-1}\mathrm{polylog}(d/\e))$; for general finitely supported measures at $\alpha=0$, \Cref{thm:eperror} admits choosing $L=\mathcal O(\mathrm{poly}(\e^{-1}))$, with a constant that may depend exponentially on $d$. The per-step cost is $\mathcal O(nd)$ for $\alpha=0$, where the constant control $C_t=I$ requires only a weighted mean over the $n$ support points, and $\mathcal O(d^3+nd^2)$ for $\alpha>0$, where each step additionally forms and inverts the $d\times d$ regularized covariance matrix. When $d$ is fixed and small, the per-step cost is $\mathcal O(n)$. 

Combining the $\mathcal O(\e^{-2})$ trajectories and summing over the $m$ measures, the total cost of computing all trajectory embeddings is $\mathcal O(mn\mathrm{poly}(\e^{-1}))$, accounting only for the dominant cost and ignoring the final $\ell_2$ comparisons. Fixing $\e$, this cost is \emph{linear} in both the cohort size $m$ and the support size $n$, improving substantially on the $\mathcal O(m^2 n^3)$ cost of exact pairwise $W_2$, and on linearized $W_2$, which is linear in $m$ but cubic in $n$ (or quadratic in $n$ with entropic regularization).

Unlike the linearized $W_2$, Eldan's $\alpha$-distance is canonical and does not require the specification of a reference measure. Given the connection between Eldan's $0$-distance and linearized optimal transport in Wiener space discussed in \Cref{sec:dan}, one might expect Eldan's $\alpha$-distance to coincide with the linearized $W_2$ distance when the standard Gaussian measure is taken as the reference. This need not hold. The linearized $W_2$ distance with respect to the standard Gaussian measure on $\R^d$ differs from the linearized distance defined on Wiener space via F\"ollmer processes: the former embeds distributions through static Brenier maps between terminal distributions, whereas the latter uses adapted drifts that respect the filtration of the underlying Brownian motion. Consequently, the $\R^d$ formulation ignores the causal structure in the path space, which is essential for comparing distributions arising from stochastic processes \cite{mikulincer2024brownian}.

\subsection{Approximate Wasserstein barycenters}

Wasserstein barycenters were introduced by \citet{agueh2011barycenters} as a means of characterizing the first-order statistics of a set of measures under the $W_2$ geometry. Given probability measures $\mu_1, \ldots, \mu_m$ on $\R^d$ with bounded second moments and weights $w_1, \ldots, w_m$ satisfying $\sum_{i=1}^mw_i=1$, their Wasserstein barycenter, $\mu^*$, is defined as the solution to the variational problem:
\begin{align}
\mu^*\coloneqq\argmin_{\mu: \ \E_\mu[\|x\|_2^2]<\infty}\sum_{i=1}^mw_iW_2^2(\mu, \mu_i). \label{barycenter}
\end{align}
When all weights are equal, $\mu^*$ coincides with the Fr\'echet mean of $\{\mu_i\}_{i=1}^m$ under the $W_2$ metric. 
  
Computation of Wasserstein barycenters is considered a challenging task in general, and various approximation methods based on regularization have been proposed; see, e.g., \cite{cuturi2014fast, benamou2015iterative}. Here we propose an alternative approach to get an approximate solution for the Wasserstein barycenter problem based on the canonical coupling induced by the joint Eldan's $\alpha$-scheme.

To grasp the idea, note that Wasserstein barycenters can be formulated as a multi-marginal optimal transport problem:
\begin{align*}
\gamma^*\coloneqq\argmin_{\gamma}\E_{x_1, \ldots, x_m\sim\gamma}\left[\sum_{i=1}^m\left\|x_i-\sum_{j=1}^mw_jx_j\right\|_2^2\right], 
\end{align*}
where the minimum is taken over the set of couplings of $\mu_1, \ldots, \mu_m$. The Wasserstein barycenter in this case is equal to the law of $\sum_{j=1}^mw_jx_j$ under $\gamma^*$ \cite[Section 4]{agueh2011barycenters}. Instead of searching for the optimal coupling, the joint Eldan's $\alpha$-scheme provides a coupling that yields an approximation of $\gamma^*$. The following result quantifies the optimality of this coupling to the one associated with the Wasserstein barycenter through how well Eldan's $\alpha$-distances approximate $W_2$ between $\mu_i$ and $\mu^*$.
\begin{theorem}\label{thm:CWa}
Let $\{\mu_i\}_{i=1}^m$ be measures on $\R^d$ and let $\mu^*$ be their Wasserstein barycenter as defined in \eqref{barycenter}. Assume the joint Eldan's $\alpha$-scheme for $\{\mu_i\}_{i=1}^m$ and $\mu^*$ exists, and their induced coupling is denoted by $(a_{1, \infty}, \ldots, a_{m, \infty}, a^*_{\infty})$. Let $\bar{\mu}$ denote the law of $\sum_{j=1}^mw_ja_{j,\infty}$. Then, 
\begin{align*}
\sum_{i=1}^mw_iW_2^2(\bar{\mu}, \mu_i)\leq C\sum_{i=1}^mw_iW_2^2(\mu^*, \mu_i),\quad\quad C\coloneqq\max_i\frac{\ds^2(\mu^*, \mu_i)}{W_2^2(\mu^*, \mu_i)}\geq 1. 
\end{align*}
\end{theorem}  
\begin{proof}
The proof follows from a direct estimate as follows: 
\begin{align*}
\sum_{i=1}^mw_iW_2^2(\bar{\mu}, \mu_i)&\leq\E\left[\sum_{i=1}^mw_i\left\|a_{i, \infty}-\sum_{j=1}^mw_ja_{j, \infty}\right\|_2^2\right]\quad\quad\text{(definition of $W_2$)}\\
&\leq\E\left[\sum_{i=1}^mw_i\left\|a_{i, \infty}-a^*_\infty\right\|_2^2\right]\quad\quad\text{(bias-variance decomposition)}\\
& = \sum_{i=1}^mw_i \ds^2(\mu^*, \mu_i)\quad\quad\text{(definition of $\ds$)}\\
& \leq C\sum_{i=1}^mw_iW_2^2(\mu^*, \mu_i).\quad\quad\text{(definition of $C$)}   
\end{align*}
\end{proof}

\begin{remark}
In general, $C$ depends on the measures $\{\mu_i\}_{i=1}^m$ and $\mu^*$ and admits no universal upper bound. For $\alpha = 0$, the topological equivalence between $\dszero$ and $W_2$ in \Cref{thm:metric-equiv} implies that $C<\infty$ for measures supported on a common compact set. In principle, an explicit bound on $C$ can be obtained by tracking the equivalence constants in the proof of \Cref{thm:metric-equiv}. However, such a bound is unlikely to be sharp in general, and $C$ may vary considerably across different measures.
\end{remark}

In practice, the computation of the approximated barycenter $\bar{\mu}$ can be performed using MC simulation. This only requires simulating the joint Eldan's $\alpha$-scheme of $\{\mu_i\}_{i=1}^m$, which we have shown before to have a computational complexity scaling linearly with the cohort size $m$, the (average) support size of $\mu_i$, and the MC sample size. Unfortunately, the sample size requirement suffers from the curse of dimensionality; therefore, our approach is most effective when $d$ is small. Even in low-dimensional settings, approximating Wasserstein barycenters remains computationally challenging when both the number of measures $m$ and their support sizes are large. We provide some numerical evidence supporting the proposed method in \Cref{sec:numbdd}.

\section{Numerical experiments}\label{sec:num}

This section presents numerical experiments to verify our theoretical findings and to illustrate the utility of Eldan's $\alpha$-distance as a computational tool for distributional data analysis. We first simulate Eldan's $\alpha$-scheme for a range of $\alpha$ values to examine the corresponding localization behavior. We then evaluate Eldan's $\alpha$-distance on a benchmark dataset as a surrogate for the Wasserstein distance for pairwise distance estimation, and subsequently apply it to compute approximate Wasserstein barycenters. All experiments are implemented in Python and conducted on a MacBook Air (M4, 2025) with 10 CPU cores.

To simulate Eldan's $\alpha$-scheme, we follow the approach introduced in \Cref{sec:newcomp} by applying the Euler--Maruyama method to the finite-dimensional system \eqref{finite_system}. For $\alpha = 0$, we use the mixed step size $\Delta t_i = \max\{1, t_{i-1}\}h$, whereas for $\alpha > 0$, we use a uniform step size $\Delta t_i = h$. Furthermore, for $\alpha > 0$, we regularize the corresponding control process as $C_t = (\Sigma_t + \delta^{1/\alpha} I)^{-\alpha}$ using a small constant $\delta > 0$ to avoid potential degeneracy of the pseudoinverse. This maintains a consistent level of numerical stability by guaranteeing that $\|C_t\|_2 \leq 1/\delta$ for any value of $\alpha$. Following the discussion in \Cref{sec:newcomp}, we fix the error parameter $\e = 0.05$ and choose $h = \e/\sqrt{d}$ and $\delta = \e/(d\sqrt{\log (d/\e)})$, as suggested by the lower bounds in \Cref{thm:logconcave-0} and \Cref{thm:1/2-gaussian}.
 
\subsection{Localization rates}\label{sec:num1}

We apply Eldan's $\alpha$-scheme to sample from two distributions in $\R^d$ with $d = 10$. The first is the uniform distribution on the hypercube $[-1, 1]^d$. The second is a three-component Gaussian mixture with equal mixing weights, where the component means are drawn independently from $\mathcal N(0, 0.1 I)$ and all components share a common covariance matrix $0.1 I$.

The true distribution is discretized using an empirical measure over $10^4$ i.i.d. samples. The values of $\alpha$ are chosen from $\{0, 0.3, 0.5, 0.8, 1\}$, and the time horizon is set to $T = 5$. We simulate $10^4$ trajectories of SL for each $\alpha$ and use the same seeds across all $\alpha$. At each step, we compute the average trace of $\Sigma_t$ of the SL process $\mu_t$ and plot its logarithm against time. For numerical stability and ease of visualization, we plot $\log_{10}(\max\{\E[\mathrm{tr}(\Sigma_t)], 10^{-6}\})$, since complete localization would otherwise lead to divergence in the log-scale. The results are presented in Figure~\ref{fig:111}. 

As shown in Figure~\ref{fig:111}, the overall trend of the localization rates is consistent with \Cref{thm:01}, with larger values of $\alpha$ leading to faster localization on average. In particular, when $\alpha < 0.5$, the convergence is polynomial within the tested regime, as indicated by the convex shape of the curves. When $\alpha = 0.5$, the convergence is exponential, evidenced by the linear shape of the curve. For $\alpha > 0.5$, the convergence is at least exponential and exhibits a faster rate during the initial phase.
\begin{figure}[htbp]
  \centering 
  \begin{subfigure}{0.42\textwidth}{\includegraphics[width=\linewidth, trim={0.0cm 0.1cm 0cm 0cm},clip]{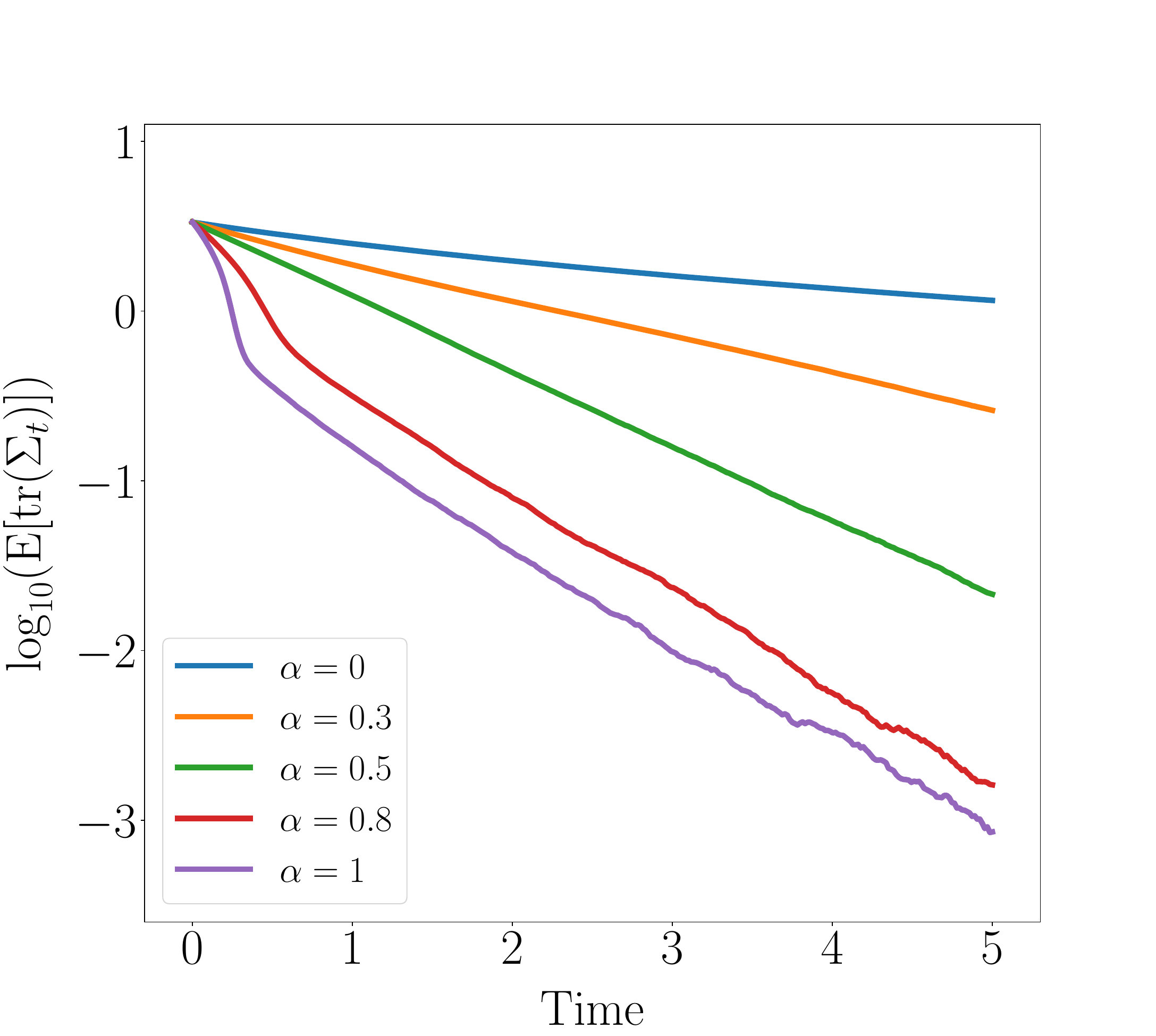}}
\end{subfigure}\hspace{1.6cm}
\begin{subfigure}{0.42\textwidth}{\includegraphics[width=\linewidth, trim={0.0cm 0.1cm 0cm 0cm},clip]{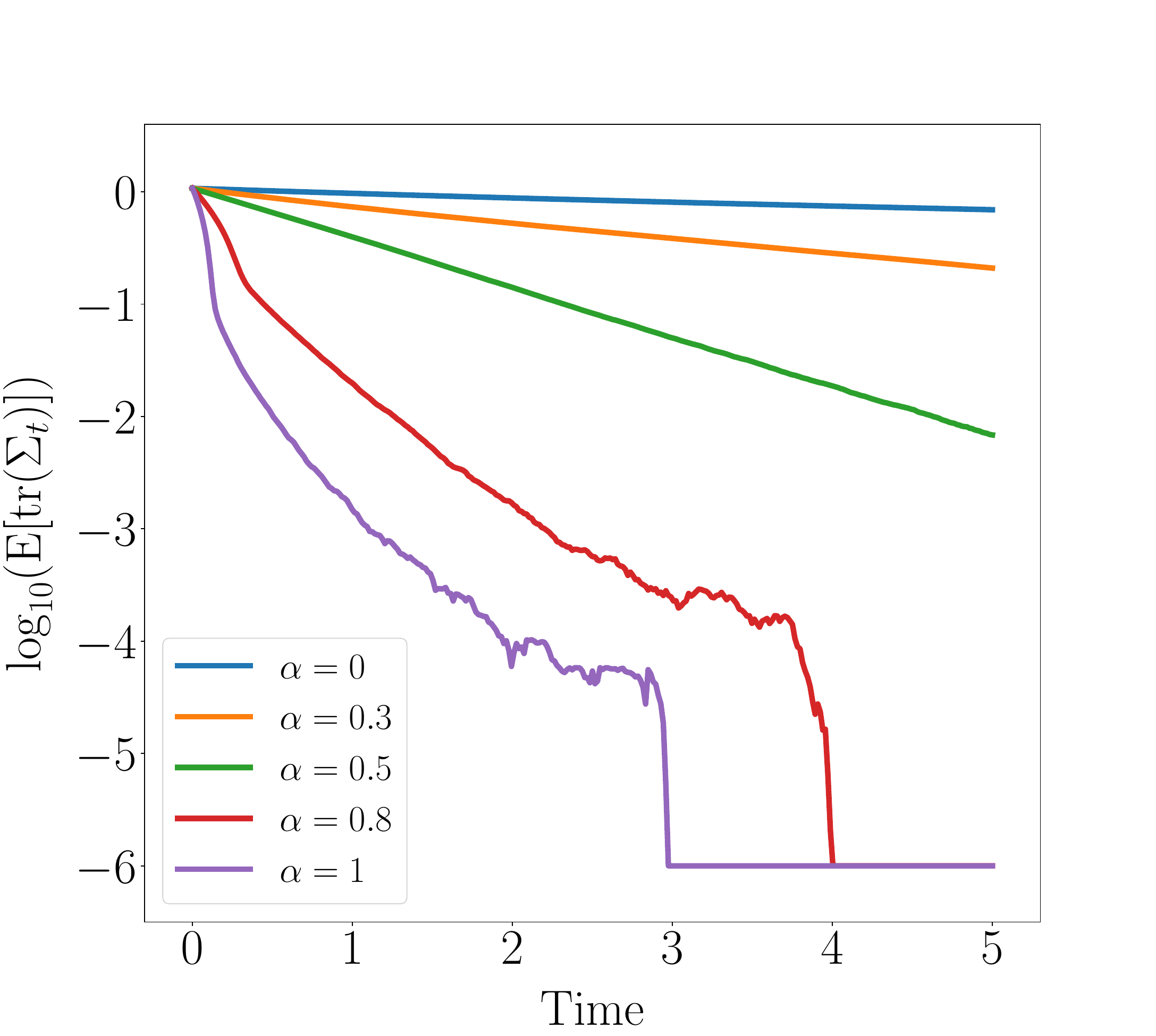}}
\end{subfigure}
\caption{Localization rates of Eldan's $\alpha$-scheme for two different distributions in $\R^{d}$ with $d=10$: a uniform distribution over $[-1 ,1]^d$ (left) and a Gaussian mixture with three components (right).} \label{fig:111}
\end{figure}

\subsection{Pairwise distance estimation}\label{sec:numbdd}

We now consider the ModelNet40 dataset \citep{wu20153d}, which consists of 3D synthetic computer-aided design (CAD) models of 40 object classes. We select 6 classes with distinct geometric structures, namely ``airplane'', ``chair'', ``cone'', ``guitar'',  ``laptop'', and ``person'', and randomly sample 50 instances from each class. Each object is represented as a point cloud with 2,048 points obtained via uniform random subsampling. Each point cloud is centered at the origin and scaled to ensure comparability across classes. We treat each point cloud as an empirical probability measure. 

We compute the pairwise distances of the objects under various metrics, including $W_2$, Eldan's $\alpha$-distance with $\alpha=0$ and $\alpha = 1/2$, sliced-$W_2$, and the linearized optimal transport (LOT) distance (linearized $W_2$). For Eldan's $\alpha$-distance, we use truncation times $T = d/\e$ and $T = \log(d/\e)$ for the cases $\alpha = 0$ and $\alpha = 1/2$, respectively, as suggested by our theoretical analysis. For sliced-$W_2$, we use 100 random projection directions. For the LOT distance, we consider two reference measures: the standard isotropic Gaussian $\mathcal N(0, I_3)$ and a centered anisotropic Gaussian with covariance $\Sigma = \mathrm{diag}(10^{-4}, 1, 1)$, which is near-degenerate in the first coordinate. This choice is specifically designed to illustrate the dependence of LOT on the reference measure. Both reference measures are discretized using 2,048 i.i.d.\ samples. 

We do not consider entropically regularized $W_2$ in this experiment since the support size remains moderate enough that exact solvers remain feasible; moreover, the computational benefits of regularization over exact solvers are limited at this scale without GPU acceleration. The exact $W_2$, sliced-$W_2$, and LOT computations are implemented using the Python Optimal Transport (POT) library \cite{flamary2021pot}. All methods are parallelized across 10 CPU cores under a similar computational setup to ensure a fair comparison. The estimated pairwise distance matrices, their alignment with $W_2$, and their relative errors to $W_2$ are reported in \Cref{fig:coupling}.

In this example, both Eldan's $\alpha$-distance and LOT approximate $W_2$ well overall. However, LOT's accuracy depends on the choice of reference measure (the isotropic reference performs better since the data is intrinsically 3D), whereas Eldan's $\alpha$-distance is canonically defined and requires no such choice. The sliced-$W_2$ systematically underestimates $W_2$ by a noticeable margin while exhibiting a positive correlation. 

In terms of computational cost, computing the full pairwise $W_2$ matrix requires more than 80 minutes, while sliced-$W_2$ requires roughly 7.5 minutes. In contrast, Eldan's $\alpha=0$ and $\alpha=1/2$ take approximately 36 and 44 seconds, respectively. This is slightly slower than LOT, which requires approximately 30 seconds on average. Crucially, both Eldan and LOT embed the distributions into a linear space, thereby avoiding the pairwise quadratic dependence on the number of instances for the expensive distance calculation. It is worth noting, however, that as the support size of the point clouds increases, methods relying on repeated exact $W_2$ solves, including the standard $W_2$ and the LOT embedding step, become increasingly expensive, making regularization indispensable for large-scale applications. 

\begin{figure}[htbp]
    \centering 
    \begin{subfigure}{\textwidth}
        \centering
        \includegraphics[width=\linewidth, trim={3cm 0.2cm 0.2cm 0.0cm}, clip]{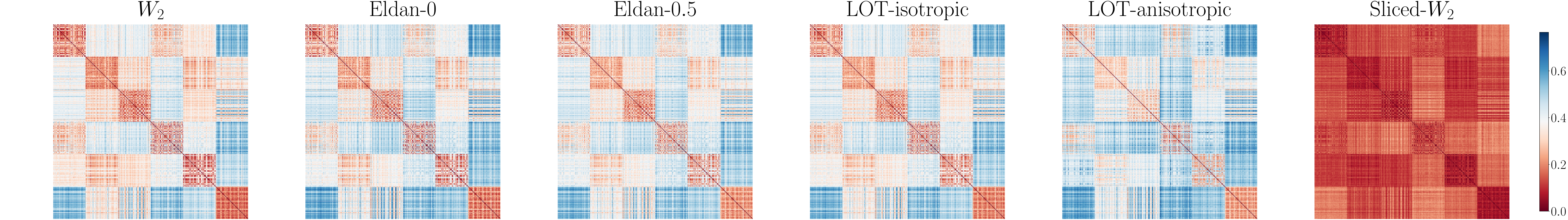}
    \end{subfigure}
    \vspace{0.2cm}
    
    \begin{subfigure}{\textwidth} 
        \centering
        \includegraphics[width=0.95\linewidth, trim={0cm 0cm 0cm 0cm}, clip]{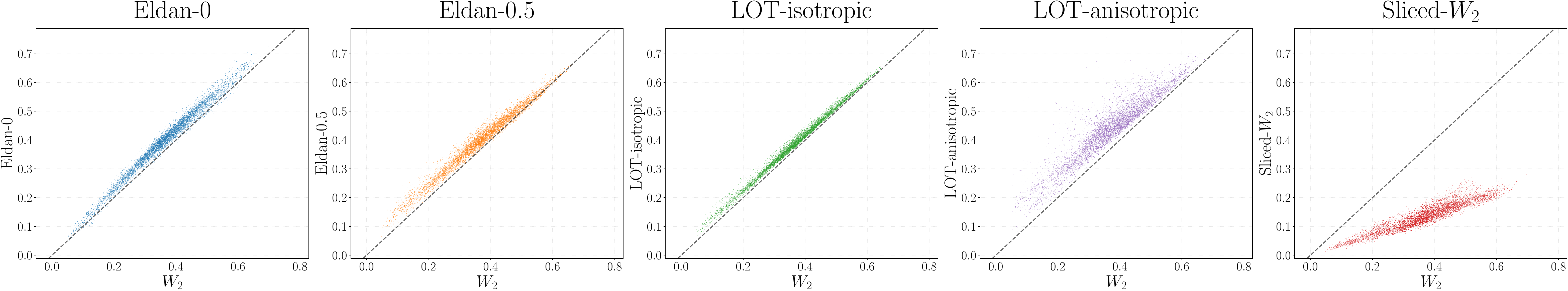}
    \end{subfigure}
    \vspace{0.2cm}
    
    \begin{subfigure}{\textwidth} 
        \centering
        \includegraphics[width=0.618\linewidth, trim={0cm 0cm 0cm 0cm}, clip]{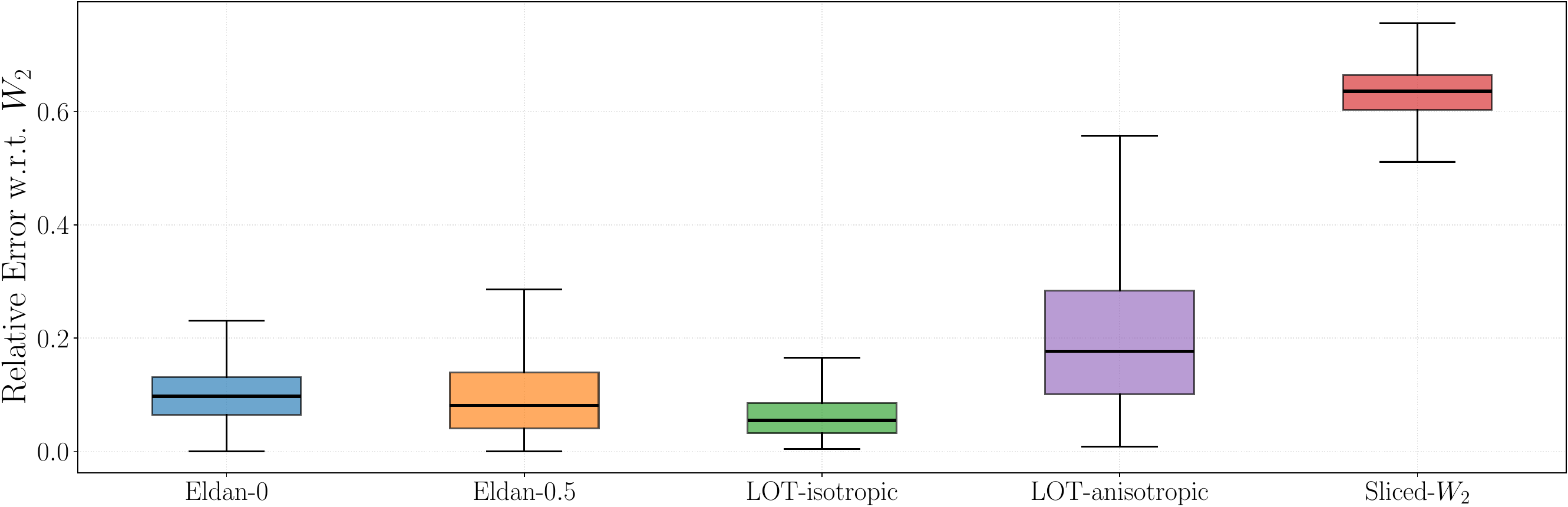}
    \end{subfigure}

    \caption{Heatmaps of the estimated pairwise distance matrices under different metrics (top panel). For Eldan's $\alpha$-distance, LOT, and sliced-$W_2$, we compare them against $W_2$, with diagonal reference lines indicating perfect agreement (middle panel). Moreover, we report their relative error with respect to $W_2$ (bottom panel).}
    \label{fig:coupling}
\end{figure}

\subsection{Approximate Wasserstein barycenters}

Under the same setup as \Cref{sec:numbdd}, we compute approximate Wasserstein barycenters for each of the 6 object classes using the couplings given by joint Eldan's $\alpha$-scheme. For the near-ground-truth reference, we apply the fixed-point iteration algorithm by \citet{alvarez2016fixed}, which alternates between computing optimal transport maps from the current barycenter estimate to each data point cloud and updating the barycenter as their average. The iteration terminates when the total displacement of the barycenter points falls below some threshold. 

In our implementation, we set this threshold as $10^{-3}$ with a maximum of 100 iterations. Since all point clouds share the same support size of 2,048, the barycenter is sought in the space of discrete measures supported on exactly 2,048 points with uniform weights. This choice corresponds to the largest support size for which the associated transport plans remain full-rank. Increasing the barycenter support further leads to rank-deficient plans, which cause convergence issues.

For the joint Eldan approach, we simulate 2,048 independent SL trajectories for each of the 50 instances. We then compute the sample mean of the resulting embeddings across the instances, yielding a 2,048-point cloud that serves as our barycenter estimate. We apply this procedure for both $\alpha = 0$ and $\alpha = 1/2$. By restricting the trajectory size to the same support size of 2,048, we ensure a fair comparison of both accuracy and computational cost against the fixed-point reference (though the Eldan approach allows for simulating a larger number of points at proportionally higher computational cost). We also compute the barycenter loss for each instance within each class as a baseline. The estimated barycenters and the corresponding statistics are reported in \Cref{fig:barycenter_grid} and \Cref{tab:barycenter_loss}, respectively. 

The approximate Wasserstein barycenters computed by the joint Eldan's $0$ and $1/2$ schemes are visually close in both shape (support) and color (density) to the fixed-point reference, and their average Wasserstein losses are substantially lower than the mean instance loss. Among the two, Eldan's $\alpha = 1/2$ achieves a loss ratio closer to one. This is consistent with \Cref{thm:CWa}, which bounds the loss ratio of Eldan's estimate $\bar{\mu}$ relative to the ground truth barycenter $\mu^*$ by $C = \max_i \ds^2(\mu^*, \mu_i)/W_2^2(\mu^*, \mu_i)$. Since Eldan's $\alpha = 1/2$ distance better approximates $W_2$ on this dataset (as shown in \Cref{fig:coupling}), the constant $C$ is closer to one, yielding a tighter guarantee. In terms of computational cost, the fixed-point algorithm requires approximately 227 seconds per class, whereas Eldan's $\alpha = 0$ and $\alpha = 1/2$ take approximately 50 and 56 seconds, respectively.

\begin{figure}[htbp]
    \centering 
    \begin{subfigure}{0.42\textwidth} 
        \centering
        \includegraphics[width=\linewidth, trim={0cm 0cm 0cm 0cm}, clip]{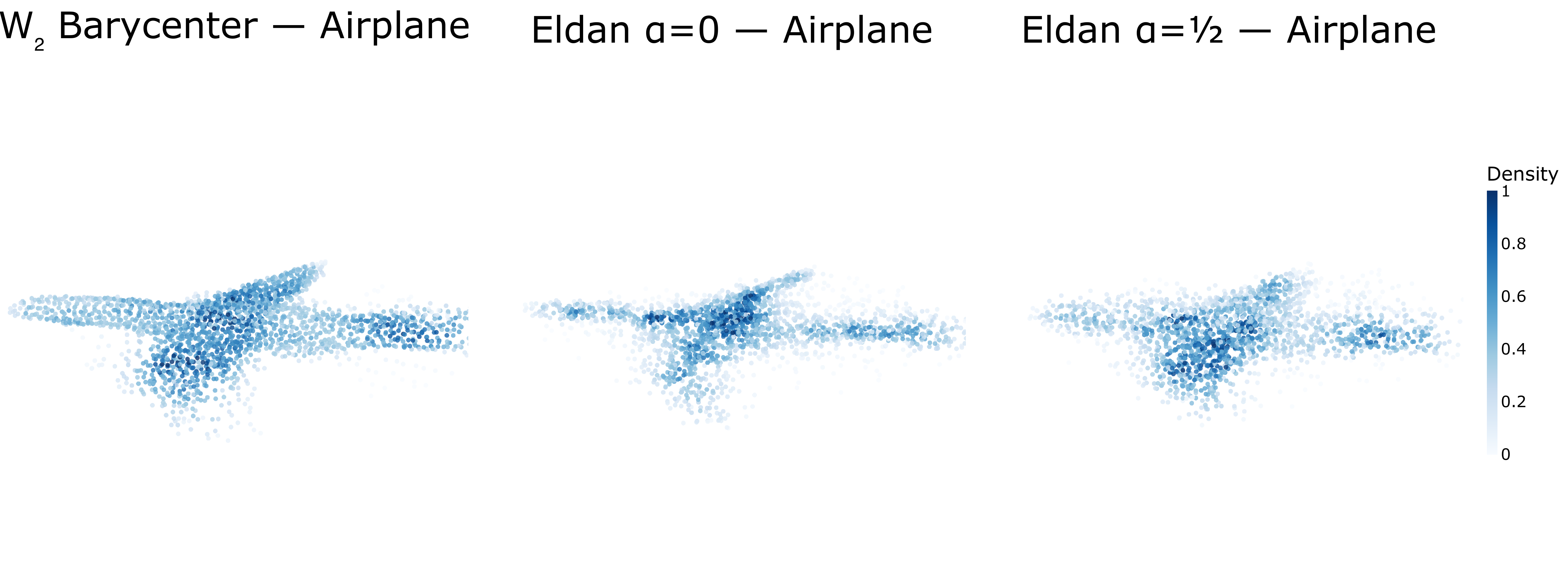}
    \end{subfigure}\hfill
    \begin{subfigure}{0.42\textwidth} 
        \centering
        \includegraphics[width=\linewidth, trim={0cm 0cm 0cm 0cm}, clip]{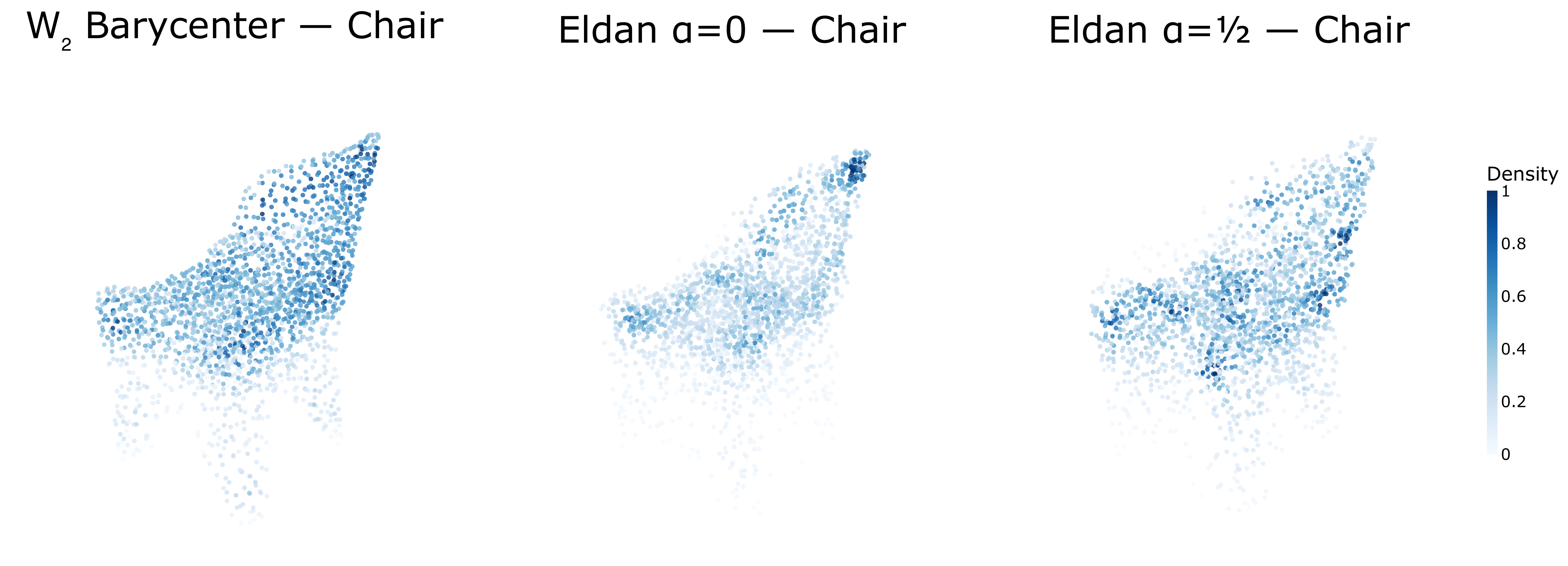}
    \end{subfigure}
    \vspace{1em} 
    \begin{subfigure}{0.42\textwidth} 
        \centering
        \includegraphics[width=\linewidth, trim={0cm 0cm 0cm 0cm}, clip]{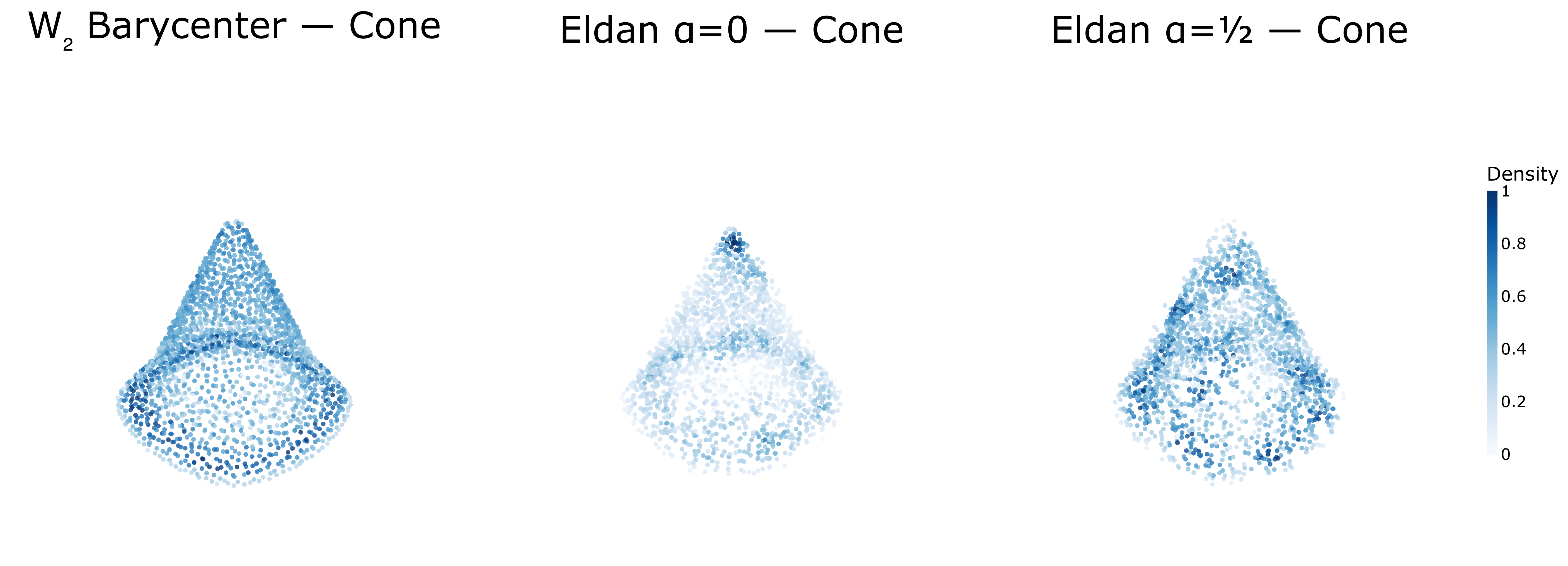}
    \end{subfigure}\hfill
    \begin{subfigure}{0.42\textwidth} 
        \centering
        \includegraphics[width=\linewidth, trim={0cm 0cm 0cm 0cm}, clip]{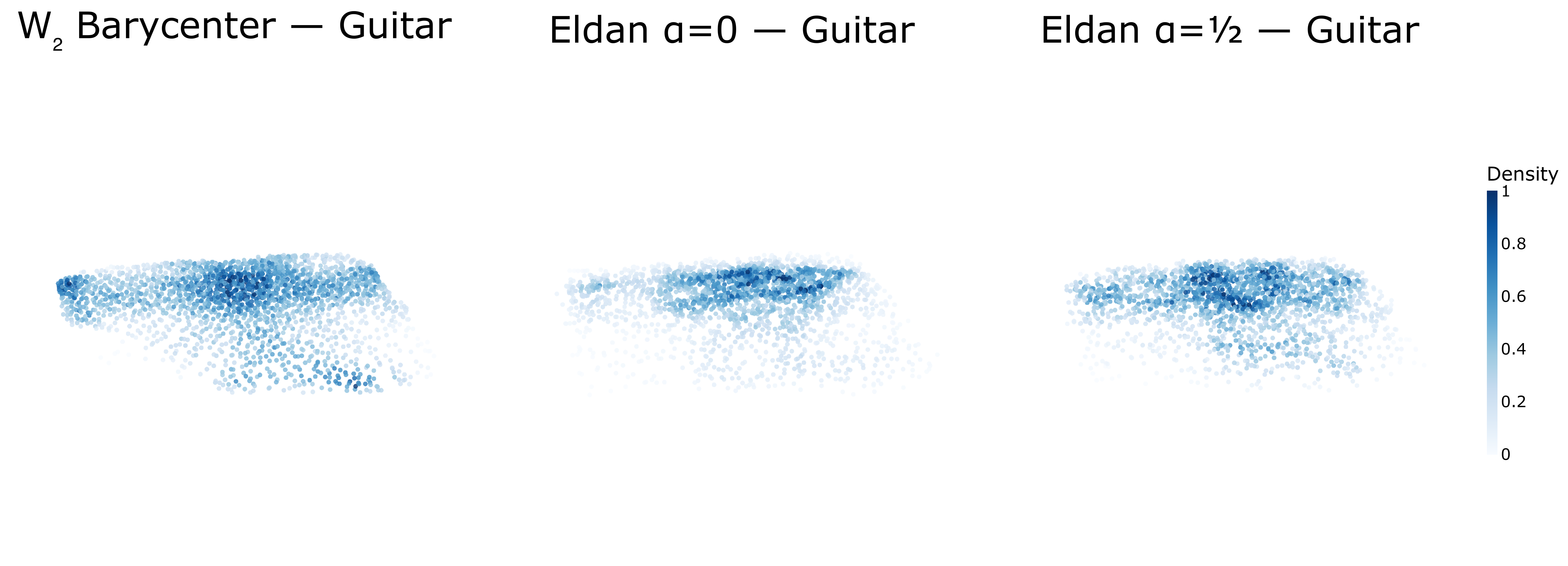}
    \end{subfigure}
    \begin{subfigure}{0.42\textwidth} 
        \centering
        \includegraphics[width=\linewidth, trim={0cm 0cm 0cm 0cm}, clip]{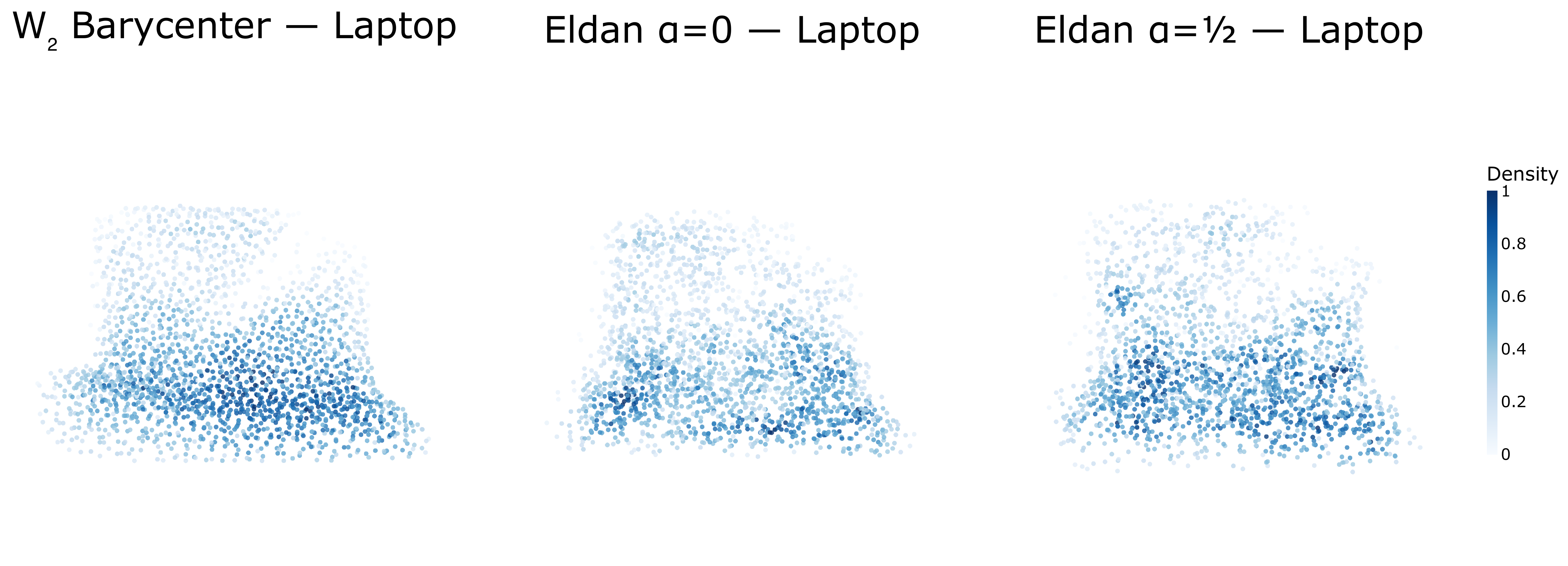}
    \end{subfigure}\hfill
    \begin{subfigure}{0.42\textwidth} 
        \centering
        \includegraphics[width=\linewidth, trim={0cm 0cm 0cm 0cm}, clip]{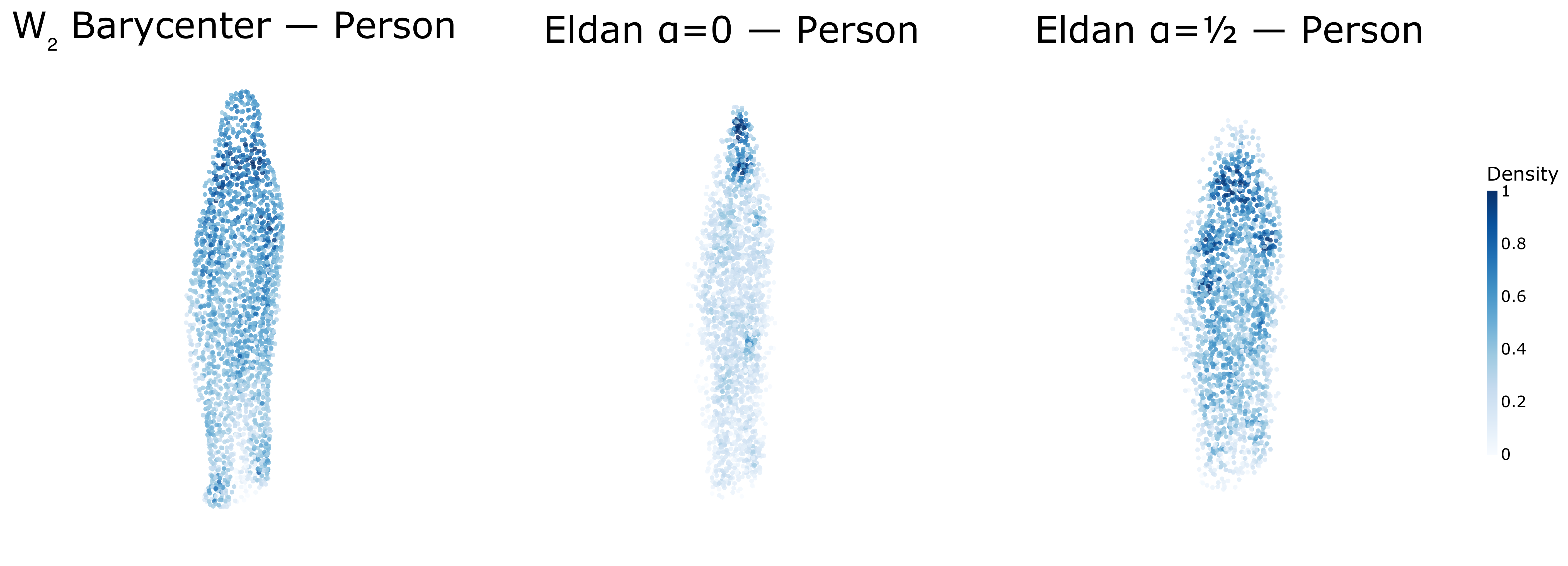}
    \end{subfigure}
    \caption{Approximate Wasserstein barycenters of the 6 object classes obtained using the joint Eldan's $0$- and $1/2$- schemes, as compared to the near-ground-truth computed using the fixed-point algorithm. Each point is colored by its local density, estimated via the $k$-nearest neighbor density estimator with $k = 30$, where the density at each point is proportional to $k$ divided by the volume of the $d$-dimensional ball with radius equal to the distance to its $k$th nearest neighbor. Density values are normalized to $[0,1]$.}
    \label{fig:barycenter_grid}
\end{figure}

\begin{table}[ht]
\centering
\caption{Average Wasserstein barycenter losses $\sum_{i=1}^m W_2^2(\widehat{\mu}, \mu_i)/m$ in \eqref{barycenter} with equal weights for each object class. The fixed-point algorithm \citep{alvarez2016fixed} serves as the near-ground-truth reference. Optimality ratios are relative to the fixed-point loss. The final three columns summarize the statistics of the instance initialization loss that serves as a baseline.}
\label{tab:barycenter_loss}
\small
\setlength{\tabcolsep}{5pt}
\begin{tabular}{l ccc cc ccc}
\toprule
& \multicolumn{3}{c}{Barycenter loss ($\times 10^{-2}$) and time} 
& \multicolumn{2}{c}{Optimality ratio} 
& \multicolumn{3}{c}{Instance loss} \\
\cmidrule(lr){2-4} \cmidrule(lr){5-6} \cmidrule(lr){7-9}
Class 
& Fixed-point 
& Eldan-$0$ 
& Eldan-$1/2$ 
& $\alpha = 0$ 
& $\alpha = 1/2$
& Min & Mean & Max \\
\midrule
Airplane & 3.03 (165 s) & 3.36 (50 s) & 3.23 (54 s) & $1.10$ & $1.07$ & $3.65$ & $5.84$ & $9.63$ \\
Chair    & 2.63 (221 s) & 3.00 (47 s) & 2.89 (54 s) & $1.12$ & $1.10$ & $3.43$ & $4.95$ & $8.31$ \\
Cone     & 3.28 (277 s) & 3.62 (52 s) & 3.45 (60 s) & $1.10$ & $1.05$ & $3.83$ & $6.34$ & $10.98$ \\
Guitar   & 5.73 (226 s) & 6.04 (47 s) & 5.80 (55 s) & $1.05$ & $1.01$ & $7.73$ & $11.13$ & $18.69$ \\
Laptop   & 4.05 (267 s) & 4.39 (51 s) & 4.36 (56 s) & $1.08$ & $1.08$ & $4.87$ & $7.48$ & $15.20$ \\
Person   & 1.83 (205 s) & 2.18 (52 s) & 2.00 (56 s) & $1.19$ & $1.09$ & $2.07$ & $3.47$ & $7.79$ \\
\bottomrule
\end{tabular}
\end{table}

\section{Conclusion and future work}\label{sec:conc}

In this paper, we unify various existing SL schemes and extend them into a joint framework as an algorithmic coupling method. We introduce two coupling approaches: a canonical method that induces Eldan's $\alpha$-distance on the space of probability measures, and a local method that extrapolates the $W_2$-optimal coupling for Gaussian measures to log-concave measures. We study the theoretical and computational aspects of Eldan's $\alpha$-distance and investigate its application as a surrogate for $W_2$ in several distributional data analysis tasks, including pairwise distance estimation over large cohorts of distributions and approximate Wasserstein barycenter computation. 

Since many results in this paper are conceptually nascent, they open several directions for future investigation. We conclude by listing a few directions that we believe merit further study and leave for future work.

\paragraph{Further understanding of Eldan's $\frac{1}{2}$-distance.} Most of our theoretical results in this article are established for Eldan's $\alpha$-distance with $\alpha = 0$ due to technical tractability. It would be interesting to know whether stronger or more general results can be obtained for $\alpha = 1/2$, which, as suggested by our computations and numerical experiments, more closely resembles the $W_2$ distance. Moreover, our analysis of the numerical method for $\alpha = 1/2$ is restricted to Gaussian measures; extending this to log-concave measures is a natural but nontrivial open problem. It is also of interest to study error analysis in the setting of finitely supported measures. This problem is generally nontrivial: even in the simplest case where $\mu$ is a Bernoulli distribution with parameter $1/2$, the corresponding SL reduces to the Wright--Fisher diffusion, and developing efficient simulation methods requires sophisticated numerical tools \cite{jenkins2017exact}.

\paragraph{Computation of Eldan's $\alpha$-distance in high dimensions.} For $\alpha = 0$, estimating Eldan's $\alpha$-distance requires evaluating the associated mean process. In low-dimensional settings, this evaluation can often be performed exactly or estimated in an unbiased manner. In high dimensions, however, the evaluation becomes nontrivial and typically requires additional approximation. Rigorously quantifying the resulting computational overhead is essential for understanding how Eldan's $\alpha$-distance scales with dimension in practice, beyond its dependence on support and cohort sizes. For $\alpha > 0$, one must additionally approximate covariance information. Understanding how to reduce the computational cost of this step is another direction for future work.

\paragraph{Iterative approximation of Wasserstein barycenters using the extrapolation scheme.} For the application to approximate Wasserstein barycenter computation, we directly use the coupling induced by joint Eldan's $\alpha$-scheme to obtain an approximate solution. Because $W_2$ is intrinsically local, it is possible that approximation accuracy can be further improved by combining the alternative local extrapolation scheme, which yields an approximate optimal coupling with respect to a reference measure, with existing fixed-point algorithms to iteratively refine the barycenter. The implementation of this extrapolation scheme requires regularization and is expected to incur a computational cost comparable to computing Eldan's $\alpha$-distance for $\alpha = 1/2$.

\section*{Acknowledgment}
We would like to thank the referees for their constructive comments, which led to a significant enrichment of this article. Y. Xu would like to thank Dan Mikulincer for his kind explanation of some results in \cite{eldan2020clt} and for his suggestion on \Cref{sec:dan}. Y. Xu also thanks Joseph Lehec and Fei Pu for providing helpful references during the preparation of this paper. We would like to thank Xiyue Han for reading through Section~\ref{sec:3} in an early version of the draft.  T. Alberts is supported in part by the Simons Foundation Travel Support for Mathematicians grant MPS-TSM-00002584 and NSF grant DMS-2136198. Y. Xu is supported by start-up funding from the University of Kentucky and by the AMS-Simons Travel Grant 3048116562. Q. Ye is supported in part by the NSF under grants DMS-2208314, IIS-2327113, and ITE-2433190. The authors acknowledge the use of Claude solely for proofreading, language editing, and improving the flowchart in \Cref{fig:flowc}. All mathematical results, proofs, and conclusions were developed and verified by the authors.

\printbibliography

\end{document}